\documentclass{amsart}
\usepackage{amssymb}
\usepackage{amscd}
\usepackage{a4wide}
\usepackage[utf8]{inputenc}

  \newtheorem{proposition}{Proposition}[section]
  \newtheorem{lemma}[proposition]{Lemma}
  \newtheorem{corollary}[proposition]{Corollary}
  \newtheorem{theorem}[proposition]{Theorem}

  \theoremstyle{definition}
  \newtheorem{definition}[proposition]{Definition}

  \theoremstyle{remark}
  \newtheorem{remark}[proposition]{Remark}

\begin{document}

\title{Covariant Bimodules Over Monoidal Hom-Hopf Algebras}
\author{SERKAN KARA\c{C}UHA}
\address{Department of Mathematics, FCUP, University of Porto, Rua Campo Alegre
687, 4169-007 Porto, Portugal}
\email{s.karacuha34@gmail.com, karacuha@itu.edu.tr}
\keywords{adjoint Hom-actions, Hom-module algebra, covariant Hom-bimodules, Hom-Yetter-Drinfel'd modules}

\begin{abstract}
Covariant Hom-bimodules are introduced and the structure theory of them in the Hom-setting is studied in a detailed way.
The category of bicovariant Hom-bimodules is proved to be a (pre)braided monoidal category and its structure theory is also provided in coordinate form. The notion of Hom-Yetter-Drinfel'd modules is presented and it is shown that the category of Hom-Yetter-Drinfel'd modules is a (pre)braided tensor category as well. As one of the main results, a (pre)braided monoidal equivalence between these tensor categories is verified, which extends the fundamental theorem of Hom-Hopf modules.
\end{abstract}

\maketitle
\section{Introduction}
Covariant bimodules have been studied in \cite{Woronowicz} to construct differential calculi on Hopf algebras over a field $k$. The concept of bicovariant bimodule (or Hopf bimodule) in \cite{Woronowicz} is considered as Hopf algebraic analogue to the notion of vector fibre bundle over a Lie group equipped with the left and right actions of the group, that is, the object analogous to the module of $1$-forms is the bicovariant bimodule which is an $H$-bimodule and an $H$-bicomodule satisfying Hopf module compatibility conditions between each of $H$-actions and each of $H$-coactions. The structure theory of covariant bimodules in a coordinate-free setting was introduced in \cite{Schauenburg}, where bicovariant bimodules are termed two-sided two-cosided Hopf modules; see also \cite{KlimykSchmudgen} for a survey of the theory in both abstract Hopf algebra language and coordinate form. With regard to knot theory and solutions of the quantum Yang-Baxter equation, the notion of a Yetter Drinfel'd module over a bialgebra  $H$ has been investigated profoundly in \cite{Yetter,RadfordTowber}, where it is defined as an $H$-module and an $H$-comodule with a compatibility condition different than the one describing a Hopf module. One of the most essential features in \cite{Yetter,RadfordTowber} is the fact that Yetter-Drinfel'd modules over a bialgebra $H$ constitute a prebraided monoidal (=tensor) category which is braided monoidal one if $H$ is a Hopf algebra with an invertible antipode. For a symmetric tensor category admitting (co-)equalizers the main result (Thm. $5.7$) in \cite{Schauenburg} expresses that the structure theorem of Hopf modules extends to an equivalence between the category of bicovariant bimodules and the category of Yetter-Drinfel'd modules over an Hopf algebra $H$. If these categories are endued with monoidal structures specific to each one, the aforementioned equivalence is braided monoidal as well, in case $H$ has a bijective antipode.

Hom-type algebras have been introduced in the form of Hom-Lie algebras in \cite{HartwigLarssonSilvestrov}, where the Jacobi identity were twisted along a linear endomorphism. This Hom-type generalization of algebras appeared as a part of the study of discretizations and deformations of vector fields and differential calculi (see \cite{AmmarMakhlouf,Hu,LarssonSilvestrov,LarssonSilvestrov1,LarssonSilvestrov2,Liu,RichardSilvestrov}) regarding the effort to deform the Witt and the Virasoro algebras. In the meantime, Hom-associative algebras have been introduced in \cite{MakhloufSilvestrov} to give rise to a Hom-Lie algebra using the commutator bracket. Other Hom-type structures such as Hom-coalgebras, Hom-bialgebras, Hom-Hopf algebras and their properties have been considered in \cite{MakhloufSilvestrov1,MakhloufSilvestrov2,Yau}. Definitions of Hom-bialgebra and Hom-Hopf algebra proposed in \cite{MakhloufSilvestrov2} contain two different endomorphisms governing the Hom-associativity and Hom-coassociativity, and in \cite{Yau} the two endomorphisms are the same to twist both associativity and coassociativity. One of the main approaches to provide a Hom-type generalization of algebras is the so-called \emph{twisting principle} which has been first suggested in \cite{Yau1}. It has since then been used to construct Hom-type objects and related algebraic structures from the classical ones and appropriate endomorphisms; for instance see \cite{AmmarMakhlouf,FregierGohr,Gohr,MakhloufPanaite,Yau2,Yau3,Yau4,Yau5}. Another important approach has been developed in a framework of tensor categories in \cite{CaenepeelGoyvaerts}, where the authors constructed a symmetric monoidal category $\mathcal{\widetilde{H}(C)}$ for a monoidal category $\mathcal{C}$ such that the associativity constraints are non-trivial by comprising automorphisms and their inverses of the related objects. By putting  $\mathcal{C}=\mathcal{M}_k$, the category of modules over a commutative ring $k$, the algebras, coalgebras, bialgebras, Hopf and Lie algebras in $\widetilde{\mathcal{H}}(\mathcal{M}_k)$ match up with their Hom-type counterparts with slight variations. These are called \emph{monoidal} Hom-algebras, Hom-coalgebras, etc. In Section 3 of \cite{CaenepeelGoyvaerts} a generalization of the fundamental theorem of Hopf modules has also been given in the Hom-setting.

In the present paper we extend the fundamental theorem of Hom-Hopf modules to a (pre-)braided monoidal equivalence between the category $^{H}_{H}\widetilde{\mathcal{H}}(\mathcal{M}_k)_H^H$ of bicovariant Hom-bimodules and the category $\widetilde{\mathcal{H}}(\mathcal{YD})^{H}_H$ of right-right Hom-Yetter-Drinfel'd modules over a monoidal Hom-Hopf algebra $(H,\alpha)$, where we equip the category $^{H}_{H}\widetilde{\mathcal{H}}(\mathcal{M}_k)_H^H$ with the tensor product over $H$, which is defined by a coequalizer modified by the related automorphisms and their inverses, and the category $\widetilde{\mathcal{H}}(\mathcal{YD})^{H}_H$ with the tensor product over a commutative ring $k$ with diagonal action and codiagonal coaction. In the meanwhile, we generalize the structure theory of covariant bimodules to the Hom-setting, and by considering the connection between bicovariant bimodules and Yetter-Drinfel'd modules over a Hopf algebra in the Hom-context we propose the notion of Hom-Yetter-Drinfeld'd modules. Below one sees how the rest of the paper proceeds.

In Section 2, we review definitions of some algebraic objects like algebra, coalgebra, bialgebra and Hopf algebra in the tensor category $\mathcal{\widetilde{H}(C)}$; moreover we recall basic definitions and propositions about monoidal Hom-structures and the structure theorem of Hom-Hopf modules. In Section 3, we consider definitions of Hopf modules, covariant bi(co)modules and Yetter-Drinfel'd modules within the tensor category $\mathcal{\widetilde{H}(C)}$. We then prove, in terms of commutation relations, some results regarding bijections between right (co)module structures on an object and right (co)module structures on the corresponding free left module turning it into a certain kind of covariant bimodule, e.g. the result stating the existence of a one-to-one correspondence between right module structures on the object $(V,\nu)$ and right module structures making $(H\otimes V,\alpha\otimes\nu)$ a left-covariant bimodule (see \cite{Schauenburg} for the classical version of these results). In Section 4 (Section 5, Section 6), we introduce the notion of left-covariant Hom-bimodules (right-covariant Hom-bimodules, bicovariant Hom-bimodules) to have a twisted, generalized version of the concept of left(right-, bi)-covariant bimodules, and furthermore we show that the category of left-covariant Hom-bimodules (right-covariant and bicovariant Hom-bimodules) is a tensor category. The structures of left-covariant and bicovariant Hom-bimodules are given in coordinate forms. What we additionally get in Section 6 is the fact that the category of bicovariant Hom-bimodules is (pre-)braided monoidal category. In Section 7, we present the definition of a Hom-Yetter-Drinfel'd module as a deformation of the classical one and prove that the category of Hom-Yetter-Drinfel'd modules is a (pre)braided tensor category. In \cite{MakhloufPanaite}, the twisting principle has been used to study Yetter-Drinfel'd modules over Hom-bialgebras; since different approaches have been used, the definitions and results obtained in this paper vary from the ones in that reference. On the other hand, the compatibility condition for Hom-Yetter-Drinfel'd modules acquired here is equivalent to the ones in \cite{ChenZhang} and \cite{LiuShen}. As the main consequence of the section, we demonstrate that the category of Hom-Yetter-Drinfel'd modules is braided monoidal equivalent to the category of bicovariant Hom-bimodules, which is among the essential purposes of this paper.

\section{Preliminaries}
Throughout the paper, $k$ denotes a fixed commutative ring and $\otimes$ means tensor product over $k$, i.e., $\otimes_k$. For a $k$-coalgebra $C$ and a right $k$-comodule $M$ over $C$, we drop the summation sign in the Sweedler's notation and write $\Delta:C\to C\otimes C,\;c\mapsto c_1\otimes c_2$ for the comultiplication and $\rho:M\to M\otimes C,\;m\mapsto m_{(0)}\otimes m_{(1)}$ for the coaction, respectively. For a solid and useful knowledge about (braided) tensor categories one can see the references \cite{Kassel} and \cite{MacLane}. One can find detailed explanations and proofs for the following definitions and propositions in \cite{CaenepeelGoyvaerts}.

 We associate to a category $\mathcal{C}$ a new category $\mathcal{H(C)}$ whose objects are ordered pairs $(A,\alpha)$, with $A\in\mathcal{C}$ and $\alpha\in Aut_\mathcal{C}(A)$, and morphisms $f:(A,\alpha)\to(B,\beta)$ are morphisms $f:A\to B$ in $\mathcal{C}$ satisfying
 \begin{equation}\beta\circ f=f\circ\alpha.\end{equation} This category is termed \emph{Hom-category} associated to $\mathcal{C}$.
If $\mathcal{C}=(\mathcal{C},\otimes,I,a,l,r)$ is a monoidal category then so is $\mathcal{H(C)}=(\mathcal{H(C)},\otimes,(I,I),a,l,r)$ with the tensor product
 \begin{equation}\label{Hom-tensor-prod}(A,\alpha)\otimes(B,\beta)=(A\otimes B,\alpha\otimes\beta )\end{equation}
 for $(A,\alpha)$ and $(B,\beta)$ in $\mathcal{H(C)}$, and the tensor product of morphisms is given by the tensor product of morphisms in $\mathcal{C}$. With the following proposition we have a modified version of the category $\mathcal{H(C)}$.

 \begin{proposition}If \: $\mathcal{C}=(\mathcal{C},\otimes,I,a,l,r)$ is a monoidal category, then $\mathcal{\widetilde{H}(C)}=(\mathcal{H(C)},\otimes,(I,I),\tilde{a},\tilde{l},\tilde{r})$ with the tensor product given by (\ref{Hom-tensor-prod}), with the associativity constraint $\tilde{a}$ defined by
 \begin{equation}\tilde{a}_{A,B,C}=a_{A,B,C}\circ((\alpha\otimes id_B)\otimes \gamma^{-1})=(\alpha\otimes (id_B\otimes \gamma^{-1}))\circ a_{A,B,C},\end{equation} for $(A,\alpha)$, $(B,\beta)$, $(C,\gamma)$ $\in$ $\mathcal{H(C)}$, and the right and left unit constraints $\tilde{r}$, $\tilde{l}$ given by
  \begin{equation}\tilde{r}_A=\alpha\circ r_A=r_A\circ(\alpha\otimes I);\:\: \tilde{l}_A=\alpha\circ l_A=l_A\circ(I\otimes\alpha), \end{equation}
  is also a monoidal category.
  On the other hand, if $\mathcal{C}=(\mathcal{C},\otimes,I,a,l,r,c)$ is a braided monoidal category, then $\mathcal{\widetilde{H}(C)}=(\mathcal{H(C)},\otimes,(I,I),\tilde{a},\tilde{l},\tilde{r},c)$ is also a braided monoidal category with the same commutativity constraint $c$.
 \end{proposition}

\begin{proposition}Let $\mathcal{C}$ be a monoidal category. Then the functor $(F,\varphi_0,\varphi_2):\mathcal{H(C)}\to\mathcal{\widetilde{H}(C)} $ defined as $F:\mathcal{H(C)}\to\mathcal{H(C)}$ identity functor, $\varphi_0:I\to I$ the identity, and for $(A,\alpha)$,$(B,\beta)$ $\in$ $\mathcal{H(C)}$, $$\varphi_2(A,B)=\alpha\otimes\beta:F(A)\otimes F(B)=F(A\otimes B)\to F(A\otimes B)=A\otimes B,$$
is strong monoidal. Hence, the monoidal categories $\mathcal{H(C)}$ and $\mathcal{\widetilde{H}(C)}$ are tensor isomorphic. If $\mathcal{C}$ is a braided monoidal category, the tensor functor $(F,\varphi_0,\varphi_2)$ is an isomorphism of braided tensor categories.
\end{proposition}

Let $\mathcal{C}$ be a monoidal category; we describe algebras and coalgebras in $\mathcal{\widetilde{H}(C)}$ as follows
\begin{definition}An algebra $\tilde{A}$ and a coalgebra $\tilde{C}$ in $\mathcal{\widetilde{H}(C)}$ are of the following forms, respectively,
\begin{enumerate}
\item $\tilde{A}=(A,\alpha,\tilde{m}_A,\eta_A)$, where $\tilde{m}_A=m_A\circ(\alpha\otimes\alpha)=\alpha\circ m_A$, $(A,m_A,\eta_A)$ is an algebra in $\mathcal{C}$, and $\alpha$ is an algebra automorphism of $A$.
    \item $\tilde{C}=(C,\gamma,\tilde{\Delta},\varepsilon_C)$, where $\tilde{\Delta}_C=(\gamma^{-1}\otimes\gamma^{-1})\circ\Delta_C=\Delta_C\circ\gamma^{-1}$, $(C,\Delta_C,\varepsilon_C)$ is a coalgebra in $\mathcal{C}$, and $\gamma$ is a coalgebra automorphism of $C$.
\end{enumerate}
\end{definition}

\begin{definition}We now describe modules and comodules on algebras $\tilde{A}$ and coalgebras $\tilde{C}$ in $\mathcal{\widetilde{H}(C)}$, respectively, hereinbelow.
\begin{enumerate}
\item A left $\tilde{A}$-module comprises an object $(M,\mu)$ $\in$ $\mathcal{\widetilde{H}(C)}$ together with a morphism $\phi:A\otimes M\to M$ in $\mathcal{\widetilde{H}(C)}$, satisfying the commutation relations $$\phi\circ(id_A\otimes\phi)\circ\tilde{a}_{A,A,M}=\phi\circ(\tilde{m}_A\otimes id_M);\: \phi\circ(\eta_A\otimes id_M)=\tilde{l}_M,$$
    where $\phi$ is said to be a left action of $\tilde{A}$ on $(M,\mu)$. Let $(N,\nu)$ be another left $\tilde{A}$-module with the action $\varphi:A\otimes N\to N$. Then a morphism $f:(M,\mu)\to(N,\nu)$ in $\mathcal{\widetilde{H}(C)}$ is called left $\tilde{A}$-linear if the relation $f\circ\phi=\varphi\circ(id_A\otimes f)$ holds. We indicate by $_A\mathcal{\widetilde{H}(M)}$ the category of left $\tilde{A}$-modules and left $\tilde{A}$-linear morphisms.

    \item A left $\tilde{C}$-comodule consists of an object $(M,\mu)$ $\in$ $\mathcal{\widetilde{H}(C)}$ together with a morphism $\rho:M \to C\otimes M$ in $\mathcal{\widetilde{H}(C)}$ fulfilling the following conditions
        $$\tilde{a}_{C,C,M}\circ(\tilde{\Delta}_C\otimes id_M)\circ\rho=(id_A\otimes \rho)\circ\rho;\: (\varepsilon_C\otimes id_M)\circ\rho=\tilde{l}^{-1}_M,$$
        where $\rho$ is called a left coaction of $\tilde{C}$ on $(M,\mu)$. A morphism $f:(M,\mu)\to(N,\nu)$ in $\mathcal{\widetilde{H}(C)}$ is termed left $\tilde{C}$-colinear if the condition $\sigma\circ f=(id_C\otimes f)\circ\rho,$ where $\sigma:N\to C\otimes N$ is the left coaction of $\tilde{C}$ on $(N,\nu)$. $^C\mathcal{\widetilde{H}(M)}$ denotes the category of left $\tilde{C}$-comodules and left $\tilde{C}$-linear morphisms.
\end{enumerate}
\end{definition}

Let $\mathcal{C}$ now be a braided tensor category. Then we have
\begin{definition} A bialgebra $\tilde{H}$ in $\mathcal{\widetilde{H}(C)}$ is the sextuple $\tilde{H}=(H,\alpha,\tilde{m}_H,\eta_H,\tilde{\Delta}_H,\varepsilon_H)$, where $\tilde{m}_H=m_H\circ(\alpha\otimes\alpha)=\alpha\circ m_H$, $\tilde{\Delta}_H=(\alpha^{-1}\otimes\alpha^{-1})\circ\Delta_H=\Delta_H\circ\alpha^{-1}$, $(H,m_H,\eta_H,\Delta_H,\varepsilon_H)$ is a bialgebra in $\mathcal{C}$, and $\alpha$ is a bialgebra automorphism of $H$. $\tilde{H}$ is a Hopf algebra in $\mathcal{\widetilde{H}(C)}$ if $H$ is a Hopf algebra in $\mathcal{C}$.
\end{definition}

If we take $\mathcal{C}$ as the category $\mathcal{M}_k$ of $k$-modules, where $k$ is a commutative ring, then the associativity and the unit constraints are given by $$\tilde{a}_{A,B,C}((a\otimes b)\otimes c)=\alpha(a)\otimes(b\otimes\gamma^{-1}(c)), $$
$$\tilde{l}_A(x\otimes a)=x\alpha(a)=\tilde{r}_A(a\otimes x).$$

\begin{definition} An algebra in $\widetilde{\mathcal{H}}(\mathcal{M}_k)$ is called a monoidal Hom-algebra and a coalgebra in
$\widetilde{\mathcal{H}}(\mathcal{M}_k)$ is termed a monoidal Hom-coalgebra, that is, respectively,
\begin{enumerate}
\item A monoidal Hom-algebra is an object $(A,\alpha) \in \widetilde{\mathcal{H}}(\mathcal{M}_k)$ together with a $k$-linear map $m_A:A\otimes A\to A,\: a\otimes b\mapsto ab=m_A(a\otimes b)$ and an element $1_A \in A$ such that
\begin{equation}\alpha(ab)=\alpha(a)\alpha(b)\:\: ;\:\: \alpha(1_A)=1_A,\end{equation}
\begin{equation}\label{Hom-associativity_Weak_unitality}\alpha(a)(bc)=(ab)\alpha(c)\:\: ;\:\: a1_A=1_Aa=\alpha(a),\end{equation}
for all $a,b,c \in A$.
\item A monoidal Hom-coalgebra is an object $(C,\gamma) \in \widetilde{\mathcal{H}}(\mathcal{M}_k)$ together with $k$-linear maps $\Delta:C\to C\otimes C,\:\Delta(c)=c_1\otimes c_2$ and $\varepsilon :C\to k$ such that
    \begin{equation}\Delta(\gamma(c))=\gamma(c_1)\otimes \gamma(c_2)\:;\: \varepsilon(\gamma(c))=\varepsilon(c),\end{equation}
    \begin{equation}\label{Hom_coassociativity_Weak_counitality}\gamma^{-1}(c_1)\otimes \Delta(c_2)=c_{11}\otimes(c_{12}\otimes \gamma^{-1}(c_2))\:;\: c_1\varepsilon(c_2)=\gamma^{-1}(c)=\varepsilon(c_1)c_2,\end{equation}
    for all $c \in C.$
\end{enumerate}
\end{definition}

\begin{definition}Now we consider modules and comodules over a Hom-algebra and a Hom-coalgebra, respectively.
\begin{enumerate}
\item A right $(A,\alpha)$-Hom-module consists of an object $(M,\mu) \in \widetilde{\mathcal{H}}(\mathcal{M}_k)$ together with a $k$-linear map $\psi:M\otimes A\to M,\: \psi(m\otimes a)=ma$, in $\widetilde{\mathcal{H}}(\mathcal{M}_k)$ satisfying the following
   \begin{equation}\mu(m)(ab)= (ma)\alpha(b)\:; \: m1_A=\mu(m),\end{equation} for all $m\in M$ and $a, b \in A$. $\psi$ being a morphism in $ \widetilde{\mathcal{H}}(\mathcal{M}_k)$ means that
   \begin{equation}\mu(ma)= \mu(m)\alpha(a).\end{equation} We call a morphism $f:(M,\mu)\to(N,\nu)$ in $\widetilde{\mathcal{H}}(\mathcal{M}_k)$ right $A$-linear if $f(ma)=f(m)a$ for all $m \in M$ and $a \in A$. We denote by $\widetilde{\mathcal{H}}(\mathcal{M}_k)_{A}$ the category of right $(A,\alpha)$-Hom-modules and $A$-linear morphisms.
   \item A right $(C,\gamma)$-Hom-comodule consists of an object $(M,\mu) \in \widetilde{\mathcal{H}}(\mathcal{M}_k)$ together with a $k$-linear map $\rho:M\to M\otimes C,\: \rho(m)=m_{[0]}\otimes m_{[1]}$, in $\widetilde{\mathcal{H}}(\mathcal{M}_k)$ such that
       \begin{equation}\mu^{-1}(m_{[0]})\otimes \Delta(m_{[1]})=m_{[0][0]}\otimes(m_{[0][1]}\otimes \gamma^{-1}(m_{[1]}))\: ; \: m_{[0]}\varepsilon(m_{[1]})=\mu^{-1}(m)\end{equation} for all $m \in M.$ $\rho$ being a morphism in $\widetilde{\mathcal{H}}(\mathcal{M}_k)$ stands for \begin{equation}\rho(\mu(m))= \mu(m_{[0]})\otimes \gamma(m_{[1]}),\end{equation} and we call a morphism $f:(M,\mu)\to(N,\nu)$ in $\widetilde{\mathcal{H}}(\mathcal{M}_k)$ right $C$-colinear if $f(m_{[0]})\otimes m_{[1]}=f(m)_{[0]}\otimes f(m)_{[1]}$ for all $m \in M$. $\widetilde{\mathcal{H}}(\mathcal{M}_k)^{C}$ denotes the category of right $(C,\gamma)$-Hom-comodules and $C$-colinear morphisms.
\end{enumerate}
\end{definition}

\begin{definition}  A bialgebra in $\widetilde{\mathcal{H}}(\mathcal{M}_k)$ is called a monoidal Hom-bialgebra and a Hopf algebra in $\widetilde{\mathcal{H}}(\mathcal{M}_k)$ is called a monoidal Hom-Hopf algebra, in other words
\begin{enumerate}
\item A monoidal Hom-bialgebra $(H,\alpha)$ is a sextuple $(H,\alpha,m,\eta,\Delta,\varepsilon)$ where $(H,\alpha,m,\eta)$ is a monoidal Hom-algebra and $(H,\alpha,\Delta,\varepsilon)$ is a monoidal Hom-coalgebra  such that
    \begin{equation}\Delta(hh')=\Delta(h)\Delta(h')\:;\: \Delta(1_H)=1_H\otimes1_H,\end{equation}
    \begin{equation}\varepsilon(hh')=\varepsilon(h)\varepsilon(h')\:;\: \varepsilon(1_H)=1,\end{equation}
    for any $h,h' \in H$.
    \item A monoidal Hom-Hopf algebra $(H,\alpha)$ is a septuple  $(H,\alpha,m,\eta,\Delta,\varepsilon,S)$ where $(H,\alpha,m,\eta,\Delta,\varepsilon)$ is a monoidal Hom-bialgebra and $S:H \to H$ is a morphism in $\widetilde{\mathcal{H}}(\mathcal{M}_k)$ such that $S\ast id_H=id_H\ast S=\eta\circ\varepsilon$.
\end{enumerate}
\end{definition}
$S$ is called antipode and it has the following properties
\begin{equation*} S(gh)=S(h)S(g)\: ; \: S(1_H)=1_H\:;\end{equation*}
\begin{equation*} \Delta(S(h))=S(h_2)\otimes S(h_1)\: ; \: \varepsilon\circ S=\varepsilon,\end{equation*}
for any elements $g,h$ of the monoidal Hom-Hopf algebra $H$.

\begin{definition}Let $(H,\alpha)$ be a monoidal Hom-Hopf algebra. Then an object $(M,\mu)$ in $\widetilde{\mathcal{H}}(\mathcal{M}_k)$ is called a left \emph{$(H,\alpha)$-Hom-Hopf module} if $(M,\mu)$ is both a left $(H,\alpha)$-Hom-module and a left $(H,\alpha)$-Hom-comodule such that the compatibility relation
\begin{equation}\rho(hm)=h_1m_{(-1)}\otimes h_2m_{(0)}\end{equation} holds for $h \in H$ and $m \in M$, where $\rho: M \to H\otimes M$ is a left $H$-coaction on $M$. A morphism of two $(H,\alpha)$-Hom-Hopf modules is a $k$-linear map which is both left $H$-linear and left $H$-colinear. The category of left $(H,\alpha)$-Hom-Hopf modules and the morphisms between them is denoted by $_{H}^{H}\widetilde{\mathcal{H}}(\mathcal{M}_k)$.
\end{definition}
We have also the following fundamental result:
\begin{theorem}\label{inverse-functors}$(F,G)$ is a pair of inverse equivalences, where the functors $F$ and $G$ are defined by
\begin{equation}F=(H\otimes -,\alpha\otimes -):\widetilde{\mathcal{H}}(\mathcal{M}_k) \to\: _{H}^{H}\widetilde{\mathcal{H}}(\mathcal{M}_k),\end{equation}
\begin{equation}G=\: ^{coH}(-):\: _{H}^{H}\widetilde{\mathcal{H}}(\mathcal{M}_k) \to \widetilde{\mathcal{H}}(\mathcal{M}_k).\end{equation}
\end{theorem}
Above, we get $^{coH}M=\{m\in M|\rho(m)=1_H\otimes \mu^{-1}(m)\}$ for a left $(H,\alpha)$-Hom-Hopf module $(M,\mu)$, which is called the submodule of \emph{left Hom-coinvariants}, and $(^{coH}M,\mu|_{^{coH}M})$ is in $\widetilde{\mathcal{H}}(\mathcal{M}_k)$.

\section{Some Correspondences in $\mathcal{\widetilde{H}(C)}$}
Let $\tilde{H}$ be a bialgebra in $\mathcal{\widetilde{H}(C)}$. We denote by $_{H}\mathcal{\widetilde{H}}(\mathcal{M})$ the category of left $\tilde{H}$-modules, by $^{H}\mathcal{\widetilde{H}}(\mathcal{M})$ the category of left $\tilde{H}$-comodules, by $_{H}\mathcal{\widetilde{H}}(\mathcal{M})_{H}$ the category of $\tilde{H}$-bimodules, etc.

\begin{definition}
An object $(M,\mu)$ in $\mathcal{\widetilde{H}(C)}$ is said to be a \emph{left} $\tilde{H}$-\emph{Hopf} \emph{module} if it is a left $\tilde{H}$-module with an action $\psi: H\otimes M \to M$ and a left $\tilde{H}$-comodule with a coaction $\rho: M \to H\otimes M$ such that the compatibility condition
 \begin{equation}\label{comp-cond-Hopf-mod}\rho\circ\psi=\psi'\circ(id_H\otimes \rho)\end{equation}
 is fulfilled.
\end{definition}
Above, $\psi'$ indicates the diagonal action of $\tilde{H}$ on the tensor product $H\otimes M$ of its modules, that is,
\begin{eqnarray*}\psi'
&=&(\tilde{m}_H\otimes\psi)\circ\tilde{a}_{H\otimes H,H,M}\circ(\tilde{a}^{-1}_{H,H,H}\otimes id_M)\circ((id_H\otimes\tau_{H,H})\otimes id_M)\\
& &\circ(\tilde{a}_{H,H,H}\otimes id_M)\circ\tilde{a}^{-1}_{H\otimes H,H,M}\circ(\tilde{\Delta}_H\otimes(id_H\otimes id_M)).
\end{eqnarray*}
Condition (\ref{comp-cond-Hopf-mod}) expresses the $\tilde{H}$-linearity of $\rho$, which is equivalent to the $\tilde{H}$-colinearity of $\psi$, i.e.,
$\rho\circ\psi=(id_H\otimes \psi)\circ\rho',$ where
\begin{eqnarray*}\rho'
&=&(\tilde{m}_H\otimes(id_H\otimes id_M))\circ\tilde{a}_{H\otimes H,H,M}\circ(\tilde{a}^{-1}_{H,H,H}\otimes id_M)\circ((id_H\otimes\tau_{H,H})\otimes id_M)\\
& &\circ(\tilde{a}_{H,H,H}\otimes id_M)\circ\tilde{a}^{-1}_{H\otimes H,H,M}\circ(\tilde{\Delta}_H\otimes\rho)
\end{eqnarray*}
is the codiagonal coaction of $\tilde{H}$ on the tensor product $H\otimes M$ of its comodules.

A morphism $f:(M,\mu)\to(N,\nu)$ between left $\tilde{H}$-Hopf modules, in $\mathcal{\widetilde{H}(C)}$, is a \emph{Hopf module morphism} if it is both $\tilde{H}$-linear and $\tilde{H}$-colinear. Left $\tilde{H}$-Hopf modules together with $\tilde{H}$-Hopf module morphisms constitute a category which is denoted by $^{H}_{H}\mathcal{\widetilde{H}}(\mathcal{M})$. The categories $^{H}\mathcal{\widetilde{H}}(\mathcal{M})_{H}$, $_{H}\mathcal{\widetilde{H}}(\mathcal{M})^{H}$ and $\mathcal{\widetilde{H}}(\mathcal{M})^{H}_{H}$ are formed in a similar way.

\begin{definition}$(M,\mu)$ $\in$ $\mathcal{\widetilde{H}(C)}$ is called \emph{left-covariant} $\tilde{H}$-\emph{bimodule} if it is a left $\tilde{H}$-module with an action $\psi: H\otimes M \to M$, a right $\tilde{H}$-module with an action $\phi: M\otimes H \to M$ and a left $\tilde{H}$-comodule with a coaction $\rho: M \to H\otimes M$ such that the following compatibility conditions
 \begin{equation}\label{comp-cond-left-cov-1}\phi\circ\psi''=\psi\circ(id_H\otimes \phi),\end{equation}
 \begin{equation}\label{comp-cond-left-cov-2}\rho\circ\psi=\psi'\circ(id_H\otimes \rho),\end{equation}
 \begin{equation}\label{comp-cond-left-cov-3}\rho\circ\phi=\phi'\circ(\rho\otimes id_H)\end{equation}
are satisfied, where $\psi''$ is the diagonal left action of $\tilde{H}$ on $M\otimes H$, $\psi'$ is the diagonal action defined above and $\phi'$ is the diagonal right action of $\tilde{H}$ on $H\otimes M$.
\end{definition}
We denote by $^{H}_{H}\mathcal{\widetilde{H}}(\mathcal{M})_H$ the category of left-covariant $\tilde{H}$-bimodules together with those morphisms in $\mathcal{\widetilde{H}(C)}$ that are left and right $\tilde{H}$-linear and left $\tilde{H}$-colinear. In the same way the category $_{H}\mathcal{\widetilde{H}}(\mathcal{M})_H^H$ \emph{right-covariant} $\tilde{H}$-\emph{bimodules} is defined.

We also make the following definition dual to the previous one:

\begin{definition}Any object $(M,\mu)$ in $\mathcal{\widetilde{H}(C)}$ is termed \emph{right-covariant} $\tilde{H}$-\emph{bicomodule} if it is a left $\tilde{H}$-comodule with a coaction $\rho: M \to H\otimes M$, a right $\tilde{H}$-comodule with a coaction $\sigma:M \to M\otimes H$ and a right $\tilde{H}$-module with an action $\phi: M\otimes H \to M$ with the requirements that $\sigma$ is left $\tilde{H}$-colinear and $\phi$ is both left and right $\tilde{H}$-colinear.
\end{definition}
By $^{H}\mathcal{\widetilde{H}}(\mathcal{M})^H_H$ we indicate the category of right-covariant $\tilde{H}$-bicomodules together with the morphisms that are left and right $\tilde{H}$-colinear and right $\tilde{H}$-linear. One can define the category $^{H}_H\mathcal{\widetilde{H}}(\mathcal{M})^H$ in the same manner.

\begin{theorem}Let $\tilde{A}$ and $\tilde{B}$ be two algebras in $\widetilde{\mathcal{H}}(\mathcal{C})$. Give the canonical left $\tilde{A}$-module structure to $(A\otimes V,\alpha\otimes\nu)$, where $(V,\nu)$ is an object in $\widetilde{\mathcal{H}}(\mathcal{C})$. There is a one-to-one correspondence between
\begin{enumerate}
\item right $\tilde{B}$-module structures making $(A\otimes V,\alpha\otimes\nu)$ an $(\tilde{A},\tilde{B})$-bimodule,
\item morphisms $f: (V\otimes B,\nu\otimes\beta)\to (A\otimes V,\alpha\otimes\nu)$ in $\widetilde{\mathcal{H}}(\mathcal{C})$ satisfying the following commutation relations
\begin{equation}\label{cond-1-for-f}f\circ(id_V\otimes \tilde{m}_B)\circ\tilde{a}_{V,B,B}=(\tilde{m}_A\otimes id_V)\circ\tilde{a}^{-1}_{A,A,V}\circ(id_A\otimes f)\circ \tilde{a}_{A,V,B}\circ(f\otimes id_B)\end{equation}
\begin{equation}\label{cond-2-for-f}f\circ(id_V\otimes \eta_B)= (\eta_A\otimes id_V)\circ\tilde{l}^{-1}_V\circ\tilde{r}_V.\end{equation}
\end{enumerate}
\end{theorem}

\begin{proof}The canonical left $\tilde{A}$-module structure on $A\otimes V$ is given by $\varphi:=(\tilde{m}_A\otimes id_V)\tilde{a}^{-1}_{A,A,V}:A\otimes(A\otimes V)\to A\otimes V$ and let $\phi:(A\otimes V)\otimes B\to A\otimes V$ be a right $\tilde{B}$-module structure which makes $A\otimes V$ an $(\tilde{A},\tilde{B})$-bimodule. If we take $f=\phi\circ\tilde{a}^{-1}_{A,V,B}\circ(\eta_A\otimes(id_V\otimes id_B))\circ\tilde{l}^{-1}_{V\otimes B}$, then

 \begin{eqnarray*}\lefteqn{(\tilde{m}_A\otimes id_V)\circ\tilde{a}^{-1}_{A,A,V}\circ(id_A\otimes f)\circ\tilde{a}_{A,V,B}\circ(f \otimes id_B)}\hspace{1em}\\
 &=&(\tilde{m}_A\otimes id_V)\circ\tilde{a}^{-1}_{A,A,V}\circ(id_A\otimes \phi)\circ(id_A\otimes\tilde{a}^{-1}_{A,V,B})\circ(id_A\otimes(\eta_A\otimes(id_V\otimes id_B)))\\
 & &\circ(id_A\otimes\tilde{l}^{-1}_{V\otimes B})\circ\tilde{a}_{A,V,B}\circ(f\otimes id_B)\\
 &=&\phi\circ((\tilde{m}_A\otimes id_V)\otimes id_B)\circ(\tilde{a}^{-1}_{A,A,V}\otimes id_B)\circ\tilde{a}^{-1}_{A,A\otimes V,B}\circ(id_A\otimes\tilde{a}^{-1}_{A,V,B})\\
 & &\circ(id_A\otimes(\eta_A\otimes(id_V\otimes id_B)))\circ(id_A\otimes\tilde{l}^{-1}_{V\otimes B})\circ\tilde{a}_{A,V,B}\circ(f\otimes id_B)\\
 &=&\phi\circ((\tilde{m}_A\otimes id_V)\otimes id_B)\circ(\tilde{a}^{-1}_{A,A,V}\otimes id_B)\circ\tilde{a}^{-1}_{A,A\otimes V,B}\circ(id_A\otimes((\eta_A\otimes id_V)\otimes id_B))\\
 & &\circ(id_A\otimes\tilde{a}^{-1}_{I,V,B})\circ(id_A\otimes\tilde{l}^{-1}_{V\otimes B})\circ\tilde{a}_{A,V,B}\circ(f\otimes id_B)\\
 &=&\phi\circ((\tilde{m}_A\otimes id_V)\otimes id_B)\circ(\tilde{a}^{-1}_{A,A,V}\otimes id_B)\circ((id_A\otimes(\eta_A\otimes id_V))\otimes id_B)\circ\tilde{a}^{-1}_{A,I\otimes V,B}\\
 & &\circ(id_A\otimes\tilde{a}^{-1}_{I,V,B})\circ(id_A\otimes\tilde{l}^{-1}_{V\otimes B})\circ\tilde{a}_{A,V,B}\circ(f\otimes id_B)\\
 &=&\phi\circ((\tilde{m}_A\otimes id_V)\otimes id_B)\circ(((id_A\otimes\eta_A)\otimes id_V)\otimes id_B)\circ(\tilde{a}^{-1}_{A,I,V}\otimes id_B)\circ\tilde{a}^{-1}_{A,I\otimes V,B}\\
 & &\circ(id_A\otimes\tilde{a}^{-1}_{I,V,B})\circ(id_A\otimes\tilde{l}^{-1}_{V\otimes B})\circ\tilde{a}_{A,V,B}\circ(f\otimes id_B)\\
 &=&\phi\circ((\tilde{r}_A\otimes id_V)\otimes id_B)\circ(\tilde{a}^{-1}_{A,I,V}\otimes id_B)\circ\tilde{a}^{-1}_{A,I\otimes V,B}\circ(id_A\otimes\tilde{a}^{-1}_{I,V,B})\circ(id_A\otimes\tilde{l}^{-1}_{V\otimes B})\\
 & &\circ\tilde{a}_{A,V,B}\circ(f\otimes id_B)\\
 &=&\phi\circ((id_A\otimes\tilde{l}_V)\otimes id_B)\circ\tilde{a}^{-1}_{A,I\otimes V,B}\circ(id_A\otimes\tilde{a}^{-1}_{I,V,B})
 \circ(id_A\otimes\tilde{l}^{-1}_{V\otimes B})\circ\tilde{a}_{A,V,B}\circ(f\otimes id_B)\\
 &=&\phi\circ((id_A\otimes\tilde{l}_V)\otimes id_B)\circ\tilde{a}^{-1}_{A,I\otimes V,B}\circ(id_A\otimes(\tilde{l}^{-1}_V\otimes id _B))
 \circ\tilde{a}_{A,V,B}\circ(f\otimes id_B)\\
 &=&\phi\circ((id_A\otimes\tilde{l}_V)\otimes id_B)\circ\tilde{a}^{-1}_{A,I\otimes V,B}\circ\tilde{a}_{A,I\otimes V,B}
 \circ((id_A\otimes\tilde{l}^{-1}_V)\otimes id _B)\circ(f\otimes id_B)\\
 &=&\phi\circ((id_A\otimes\tilde{l}_V)\otimes id_B)\circ((id_A\otimes\tilde{l}^{-1}_V)\otimes id _B)\circ(f\otimes id_B)\\
 &=&\phi\circ(\phi\otimes id_B)(\tilde{a}^{-1}_{A,V,B}\otimes id_B)\circ((\eta_A\otimes(id_V\otimes id_B))\otimes id_B)\circ(\tilde{l}^{-1}_{V\otimes B}\otimes id_B)\\
 &=&\phi\circ((id_A\otimes id_V)\otimes\tilde{m}_B)\circ\tilde{a}_{A\otimes V,B,B}\circ(\tilde{a}^{-1}_{A,V,B}\otimes id_B)\circ((\eta_A\otimes(id_V\otimes id_B))\otimes id_B)\\
 & &\circ(\tilde{l}^{-1}_{V\otimes B}\otimes id_B)\\
 &=&\phi\circ\tilde{a}^{-1}_{A,V,B}\circ(id_A\otimes (id_V\otimes\tilde{m}_B))\circ(id_A\otimes\tilde{a}_{V,B,B})\circ\tilde{a}_{A,V\otimes B,B}
 \circ(\tilde{a}_{A,V,B}\otimes id_B)\\
 & &\circ(\tilde{a}^{-1}_{A,V,B}\otimes id_B)\circ((\eta_A\otimes(id_V\otimes id_B))\otimes id_B)\circ(\tilde{l}^{-1}_{V\otimes B}\otimes id_B)\\
 &=&\phi\circ\tilde{a}^{-1}_{A,V,B}\circ(id_A\otimes (id_V\otimes\tilde{m}_B))\circ(id_A\otimes\tilde{a}_{V,B,B})\circ(\eta_A\otimes((id_V\otimes id_B)\otimes id_B))\\
 & &\circ\tilde{a}_{I,V\otimes B,B}\circ(\tilde{l}^{-1}_{V\otimes B}\otimes id_B)\\
 &=&\phi\circ\tilde{a}^{-1}_{A,V,B}\circ(id_A\otimes (id_V\otimes\tilde{m}_B))\circ(\eta_A\otimes(id_V\otimes (id_B\otimes id_B)))\circ(id_I\otimes\tilde{a}_{V, B,B})\\
 & &\circ\tilde{a}_{I,V\otimes B,B}\circ(\tilde{l}^{-1}_{V\otimes B}\otimes id_B)\\
 &=&\phi\circ\tilde{a}^{-1}_{A,V,B}\circ(id_A\otimes (id_V\otimes\tilde{m}_B))\circ(\eta_A\otimes(id_V\otimes (id_B\otimes id_B)))\circ(id_I\otimes\tilde{a}_{V, B,B})\circ\tilde{l}^{-1}_{(V\otimes B)\otimes B}\\
 &=&\phi\circ\tilde{a}^{-1}_{A,V,B}\circ(\eta_A\otimes(id_V\otimes id_B))\circ(id_I\otimes(id_V\otimes\tilde{m}_B))\circ(id_I\otimes\tilde{a}_{V, B,B})\circ\tilde{l}^{-1}_{(V\otimes B)\otimes B}\\
 &=&\phi\circ\tilde{a}^{-1}_{A,V,B}\circ(\eta_A\otimes(id_V\otimes id_B))\circ\tilde{l}^{-1}_{V\otimes B}\circ(id_V\otimes\tilde{m}_B)\circ\tilde{a}_{V,B,B}\\
 &=&f\circ(id_V\otimes\tilde{m}_B)\circ\tilde{a}_{V,B,B}\:,
\end{eqnarray*}
where in the second equation the bimodule compatibility condition (i.e., left $\tilde{A}$-linearity of $\phi$) $\phi\circ(\varphi\otimes id_B)\circ\tilde{a}^{-1}_{A,A\otimes V,B}=\varphi\circ(id_A\otimes \phi)$ was used and the twelfth equation has followed by the associativity condition for $\phi$,
\begin{eqnarray*}f(id_V\otimes\eta_B)&=&\phi\circ\tilde{a}^{-1}_{A,V,B}\circ(\eta_A\otimes(id_V\otimes id_B))\circ\tilde{l}^{-1}_{V\otimes B}\circ(id_V\otimes\eta_B)\\
&=&\phi\circ\tilde{a}^{-1}_{A,V,B}\circ\tilde{a}_{A,V,B}\circ((\eta_A\otimes id_V)\otimes id_B)\circ(\tilde{l}^{-1}_{V}\otimes id_B)\circ(id_V\otimes\eta_B)\\
&=&\phi\circ((\eta_A\otimes id_V)\otimes id_B)\circ((id_I\otimes id_V)\otimes\eta_B)\circ(\tilde{l}^{-1}_{V}\otimes id_I)\\
&=&\phi\circ((id_A\otimes id_V)\otimes \eta_B)\circ((\eta_A\otimes id_V)\otimes id_I)\circ(\tilde{l}^{-1}_{V}\otimes id_I)\\
&=&\phi\circ((id_A\otimes id_V)\otimes \eta_B)\circ\tilde{r}^{-1}_{A\otimes V}\circ(\eta_A\otimes id_V)\circ\tilde{l}^{-1}_V\circ\tilde{r}^{-1}_V\\
&=&\tilde{r}_{A\otimes V}\circ\tilde{r}^{-1}_{A\otimes V}\circ(\eta_A\otimes id_V)\circ\tilde{l}^{-1}_V\circ\tilde{r}^{-1}_V\\
&=&(\eta_A\otimes id_V)\circ\tilde{l}^{-1}_V\circ\tilde{r}^{-1}_V \:,
\end{eqnarray*}
where in the penultimate equation the unity condition for $\phi$ has been used.

On the other hand, suppose that $f:V\otimes B\to A\otimes V$ is given fulfilling the relations (\ref{cond-1-for-f}) and (\ref{cond-2-for-f}). If we put $\phi=\varphi\circ(id_A\otimes f)\circ\tilde{a}_{A,V,B}$, where $\varphi$ is the canonical left $\tilde{A}$-module structure on $A\otimes V$, then $\phi$ is $\tilde{A}$-linear:
\begin{eqnarray*}\varphi\circ(id_A\otimes \phi)&=&\varphi\circ(id_A\otimes \varphi)\circ(id_A\otimes(id_A\otimes f))\circ(id_A\otimes \tilde{a}_{A,V,B})\\
&=&\varphi\circ(\tilde{m}_A\otimes(id_A\otimes id_V))\circ\tilde{a}^{-1}_{A,A,A\otimes V}\circ(id_A\otimes(id_A\otimes f))\circ(id_A\otimes \tilde{a}_{A,V,B})\\
&=&\varphi\circ(\tilde{m}_A\otimes(id_A\otimes id_V))\circ((id_A\otimes id_A)\otimes f)\circ\tilde{a}^{-1}_{A,A,V\otimes B}\circ(id_A\otimes \tilde{a}_{A,V,B})\\
&=&\varphi\circ(id_A\otimes f)\circ(\tilde{m}_A\otimes(id_V\otimes id_B))\circ\tilde{a}^{-1}_{A,A,V\otimes B}\circ(id_A\otimes \tilde{a}_{A,V,B})\\
&=&\varphi\circ(id_A\otimes f)\circ(\tilde{m}_A\otimes(id_V\otimes id_B))\circ\tilde{a}_{A\otimes A,V,B}\circ(\tilde{a}^{-1}_{A,A,V}\otimes id_B)\circ\tilde{a}^{-1}_{A,A\otimes V,B}\\
&=&\varphi\circ(id_A\otimes f)\circ\tilde{a}_{A,V,B}\circ((\tilde{m}_A\otimes id_V)\otimes id_B)\circ(\tilde{a}^{-1}_{A,A,V}\otimes id_B)\circ\tilde{a}^{-1}_{A,A\otimes V,B}\\
&=&\varphi\circ(id_A\otimes f)\circ\tilde{a}_{A,V,B}\circ(\varphi\otimes id_B)\circ\tilde{a}^{-1}_{A,A\otimes V,B}\\
&=&\phi\circ(\varphi\otimes id_B)\circ\tilde{a}^{-1}_{A,A\otimes V,B}\:,
\end{eqnarray*}
where we have used the associativity of $\varphi$ in the second equality and the fifth equality has followed by the fact that $\tilde{a}$ satisfies the Pentagon Axiom.

In what follows we prove that $\phi$ makes $A\otimes V$ a right $\tilde{B}$-module: $\phi$ is associative due to
\begin{eqnarray*}\phi\circ(\phi\otimes id_B)&=&\phi\circ(\varphi\otimes id_B)\circ((id_A\otimes f)\otimes id_B)\circ(\tilde{a}_{A,V,B}\otimes id_B)\\
&=&\varphi\circ(id_A\otimes\phi)\circ\tilde{a}_{A,A\otimes V,B}\circ((id_A\otimes f)\otimes id_B)\circ(\tilde{a}_{A,V,B}\otimes id_B)\\
&=&\varphi\circ(id_A\otimes\phi)\circ(id_A\otimes (f\otimes id_B))\circ\tilde{a}_{A,V\otimes B,B}\circ(\tilde{a}_{A,V,B}\otimes id_B)\\
&=&\varphi\circ(id_A\otimes (\tilde{m}_A\otimes id_V)\circ\tilde{a}^{-1}_{A,A,V}\circ(id_A\otimes f)\circ\tilde{a}_{A,V,B}\circ(f\otimes id_B))\\
& &\circ\tilde{a}_{A,V\otimes B,B}\circ(\tilde{a}_{A,V,B}\otimes id_B)\\
&=&\varphi\circ(id_A\otimes f\circ(id_V\otimes \tilde{m}_B)\circ\tilde{a}_{V,B,B})\circ\tilde{a}_{A,V\otimes B,B}\circ(\tilde{a}_{A,V,B}\otimes id_B)\\
&=&\varphi\circ(id_A\otimes f)\circ\tilde{a}_{A,V,B}\circ\tilde{a}^{-1}_{A,V,B}\circ(id_A\otimes(id_V\otimes\tilde{m}_B))\circ(id_A\otimes\tilde{a}_{V,B,B})\\
& &\circ\tilde{a}_{A,V\otimes B,B}\circ(\tilde{a}_{A,V,B}\otimes id_B)\\
&=&\varphi\circ(id_A\otimes f)\circ\tilde{a}_{A,V,B}\circ((id_A\otimes id_V)\otimes\tilde{m}_B)\circ\tilde{a}^{-1}_{A,V,B\otimes B}
\circ(id_A\otimes\tilde{a}_{V,B,B})\\
& &\circ\tilde{a}_{A,V\otimes B,B}\circ(\tilde{a}_{A,V,B}\otimes id_B)\\
&=&\phi\circ((id_A\otimes id_V)\otimes\tilde{m}_B)\circ\tilde{a}_{A\otimes V,B,B}\:,
\end{eqnarray*}
where the second equality is a consequence of the left $\tilde{A}$-linearity of $\phi$, the fifth one results from (\ref{cond-1-for-f}) and in the last equality Pentagon Axiom for $\tilde{a}$ has been applied, and $\phi$ satisfies the unity condition in so far as
\begin{eqnarray*}\lefteqn{\phi\circ((id_A\otimes id_V)\otimes \eta_B)}\hspace{3em}\\
&=&(\tilde{m}_A\otimes id_V)\circ\tilde{a}^{-1}_{A,A,V}\circ(id_A\otimes f)\circ\tilde{a}_{A,V,B}\circ((id_A\otimes id_V)\otimes \eta_B)\\
&=&(\tilde{m}_A\otimes id_V)\circ\tilde{a}^{-1}_{A,A,V}\circ(id_A\otimes f)\circ(id_A\otimes (id_V\otimes \eta_B))\circ\tilde{a}_{A,V,I}\\
&=&(\tilde{m}_A\otimes id_V)\circ\tilde{a}^{-1}_{A,A,V}\circ(id_A\otimes f\circ(id_V\otimes \eta_B))\circ\tilde{a}_{A,V,I}\\
&=&(\tilde{m}_A\otimes id_V)\circ\tilde{a}^{-1}_{A,A,V}\circ(id_A\otimes(\eta_A\otimes id_V)\circ\tilde{l}^{-1}_{V}\circ\tilde{r})\circ\tilde{a}_{A,V,I}\\
&=&(\tilde{m}_A\otimes id_V)\circ((id_A\otimes \eta_A)\otimes id_V)\circ\tilde{a}^{-1}_{A,I,V}\circ(id_A\otimes\tilde{l}^{-1}_{V})
\circ(id_A\otimes\tilde{r})\circ\tilde{a}_{A,V,I}\\
&=&(\tilde{m}_A\circ(id_A\otimes \eta_A)\otimes id_V)\circ\tilde{a}^{-1}_{A,I,V}\circ(id_A\otimes\tilde{l}^{-1}_{V})\circ\tilde{r}_{A\otimes V}\\
&=&(\tilde{r}_A\otimes id_V)\circ\tilde{a}^{-1}_{A,I,V}\circ(id_A\otimes\tilde{l}^{-1}_{V})\circ\tilde{r}_{A\otimes V}\\
&=&(id_A\otimes\tilde{l}_{V})\circ\tilde{a}_{A,I,V}\circ\tilde{a}^{-1}_{A,I,V}\circ(id_A\otimes\tilde{l}^{-1}_{V})\circ\tilde{r}_{A\otimes V}\\
&=&\tilde{r}_{A\otimes V}\:,
\end{eqnarray*}
where we have used (\ref{cond-2-for-f}) in the fourth equality and the Triangle Axiom in the penultimate one.
\end{proof}
We now make use of the theorem above, which has been given in the general tensor category context, to prove the undermentioned theorem.

\begin{theorem}\label{bijection-left-cov-bimod}Let $(V,\nu)$ $\in$ $\mathcal{\widetilde{H}}(\mathcal{C})$ and let $(H\otimes V,\alpha\otimes\nu)$ $\in$ $^{H}_H\mathcal{\widetilde{H}}(\mathcal{M})$ with the canonical $\tilde{H}$-module and $\tilde{H}$-comodule structures, that is, with the action  $(\tilde{m}_H\otimes id_V)\tilde{a}^{-1}_{H,H,V}:H\otimes(H\otimes V)\to H\otimes V$ and the coaction $\tilde{a}_{H,H,V}(\tilde{\Delta}\otimes id_V):H\otimes V\to H\otimes(H\otimes V)$. Then there is a bijection between right $\tilde{H}$-module structures making $(H\otimes V,\alpha\otimes\nu)$ a left-covariant $\tilde{H}$-bimodule and right $\tilde{H}$-module structures on $(V,\nu)$.
\end{theorem}

\begin{proof}By performing the previous theorem to the left $\tilde{H}$-module $(H\otimes V,\alpha\otimes\nu)$ in the category $^{H}\mathcal{\widetilde{H}}(\mathcal{M})$ we obtain a bijection between right $\tilde{H}$-module structures making $H\otimes V$ an $\tilde{H}$-bimodule and left $\tilde{H}$-colinear morphisms $f: (V\otimes H,\nu\otimes\alpha)\to (H\otimes V,\alpha\otimes\nu)$ fulfilling
\begin{enumerate}
\item\label{cond-1-for-f-AB=H} $f\circ(id_V\otimes \tilde{m}_H)\circ\tilde{a}_{V,H,H}=(\tilde{m}_H\otimes id_V)\circ\tilde{a}^{-1}_{H,H,V}\circ(id_H\otimes f)\circ \tilde{a}_{H,V,H}\circ(f\otimes id_H),$
\item\label{cond-2-for-f-AB=H} $f\circ(id_V\otimes \eta_H)= (\eta_H\otimes id_V)\circ\tilde{l}^{-1}_V\circ\tilde{r}_V.$
\end{enumerate}
For any object $(X,\chi)$ in $^{H}\mathcal{\widetilde{H}}(\mathcal{M})$ with the coaction $\rho:X\to H\otimes X$, there is the bijective mapping $$F_X:\:^{\mathcal{H}}Hom(X, H\otimes V) \to Hom(X,V),\: f\mapsto \tilde{l}_{V}\circ(\varepsilon_H\otimes id_V)\circ f$$
with the inverse given by $g \mapsto (id_H\otimes g)\circ \rho$. Let us take $f: V\otimes H\to H\otimes V$ and put $\psi=\tilde{l}_{V}\circ(\varepsilon_H\otimes id_V)\circ f$, then we get

\begin{eqnarray*}\lefteqn{F_{(V\otimes H)\otimes H}((\tilde{m}_H\otimes id_V)\circ\tilde{a}^{-1}_{H,H,V}\circ(id_H\otimes f)\circ \tilde{a}_{H,V,H}\circ(f\otimes id_H))}\hspace{3em}\\
&=&\tilde{l}_{V}\circ(\varepsilon_H\otimes id_V)\circ(\tilde{m}_H\otimes id_V)\circ\tilde{a}^{-1}_{H,H,V}\circ(id_H\otimes f)\circ \tilde{a}_{H,V,H}\circ(f\otimes id_H)\\
&=&\tilde{l}_{V}\circ(\varepsilon_H\circ\tilde{m}_H\otimes id_V)\circ\tilde{a}^{-1}_{H,H,V}\circ(id_H\otimes f)\circ \tilde{a}_{H,V,H}\circ(f\otimes id_H)\\
&=&\tilde{l}_{V}\circ(\tilde{l}_{I}\circ(\varepsilon_H\otimes\varepsilon_H)\otimes id_V)\circ\tilde{a}^{-1}_{H,H,V}\circ(id_H\otimes f)\circ \tilde{a}_{H,V,H}\circ(f\otimes id_H)\\
&=&\tilde{l}_{V}\circ(\varepsilon_H\otimes id_V)\circ(id_H\otimes \tilde{l}_V)\circ(id_H\otimes(\varepsilon_H\otimes id_V))\circ(id_H\otimes f)\circ \tilde{a}_{H,V,H}\circ(f\otimes id_H)\\
&=&\tilde{l}_{V}\circ(\varepsilon_H\otimes id_V)\circ f\circ(\tilde{l}_V\otimes id_H)\circ((\varepsilon_H\otimes id_V)\otimes id_H)\circ\tilde{a}^{-1}_{H,V,H} \circ\tilde{a}_{H,V,H}\circ(f\otimes id_H)\\
&=&\psi\circ(\psi\otimes id_H),
\end{eqnarray*}
and
\begin{eqnarray*}F_{(V\otimes H)\otimes H}(f\circ(id_V\otimes \tilde{m}_H)\circ\tilde{a}_{V,H,H})&=&\tilde{l}_{V}\circ(\varepsilon_H\otimes id_V)\circ f\circ(id_V\otimes \tilde{m}_H)\circ\tilde{a}_{V,H,H}\\
&=&\psi\circ(id_V\otimes \tilde{m}_H)\circ\tilde{a}_{V,H,H}.
\end{eqnarray*}
Thus the associativity of $\psi$ holds if and only if $F_{(V\otimes H)\otimes H}(f\circ(id_V\otimes \tilde{m}_H)\circ\tilde{a}_{V,H,H})=F_{(V\otimes H)\otimes H}((\tilde{m}_H\otimes id_V)\circ\tilde{a}^{-1}_{H,H,V}\circ(id_H\otimes f)\circ \tilde{a}_{H,V,H}\circ(f\otimes id_H))$, which is equivalent to the relation (\ref{cond-1-for-f-AB=H}) due to the fact that $F_{(V\otimes H)\otimes H}$ is a bijective map. By a similar argument, we get the equivalence between the unity of $\psi$ and the relation (\ref{cond-2-for-f-AB=H}) since $F_{V\otimes I}(f\circ(id_V\otimes\eta_H))=\psi\circ(id_V\otimes \eta_H)$ and $F_{V\otimes I}((\eta_H\otimes id_V)\circ\tilde{l}^{-1}_V\circ\tilde{r}_V)=\tilde{r}_V$ and $F_{V\otimes I}$ is a bijection.
\end{proof}
By applying the theorem above in the opposite category we get
\begin{corollary}\label{bijection-left-cov-bicomod}Let $(V,\nu)$ $\in$ $\mathcal{\widetilde{H}}(\mathcal{C})$ and let $(H\otimes V,\alpha\otimes\nu)$ $\in$ $^{H}_H\mathcal{\widetilde{H}}(\mathcal{M})$ with the canonical $\tilde{H}$-module and $\tilde{H}$-comodule structures. There is a one-to-one correspondence between right $\tilde{H}$-comodule structures on $H\otimes V$ making it a left-covariant $\tilde{H}$-bicomodule and the right $\tilde{H}$-comodule structures on $V$.
\end{corollary}

\begin{definition}$(M,\mu)$ $\in$ $\mathcal{\widetilde{H}(C)}$ is called \emph{bicovariant} $\tilde{H}$-\emph{bimodule} if it is a left $\tilde{H}$-module with an action $\psi: H\otimes M \to M$, a right $\tilde{H}$-module with an action $\phi: M\otimes H \to M$, a left $\tilde{H}$-comodule with a coaction $\rho: M \to H\otimes M$ and a right $\tilde{H}$-comodule with a coaction $\sigma: M \to M\otimes H$ such that in addition to (\ref{comp-cond-left-cov-1})-(\ref{comp-cond-left-cov-3}) the following compatibility conditions also hold
 \begin{equation}\rho''\circ\sigma=(id_H\otimes \sigma)\circ\rho,\end{equation}
 \begin{equation}\sigma\circ\psi=\psi''\circ(id_H\otimes \sigma),\end{equation}
 \begin{equation}\sigma\circ\phi=\phi''\circ(\sigma\otimes id_H),\end{equation}
 where $\rho''$ is the codiagonal left coaction of $\tilde{H}$ on $M\otimes H$, $\psi''$ is the diagonal left action of $\tilde{H}$ on $M\otimes H$ and $\phi''$ is the diagonal right action of $\tilde{H}$ on $M\otimes H$.
\end{definition}
We denote by $^{H}_{H}\mathcal{\widetilde{H}}(\mathcal{M})^{H}_H$ the category of bicovariant $\tilde{H}$-bimodules together with those morphisms in $\mathcal{\widetilde{H}(C)}$ that are $\tilde{H}$-linear and $\tilde{H}$-colinear on both sides.

\begin{definition}A \emph{right-right Yetter-Drinfel'd module} $(V,\nu)$ in $\mathcal{\widetilde{H}(C)}$ is a right $\tilde{H}$-module with an action $\psi: V\otimes H \to V$ and a right $\tilde{H}$-comodule with a coaction $\rho: V \to V\otimes H$ such that the following compatibility condition
 \begin{equation*}\psi'\circ(\rho\otimes id_H)=(id_V\otimes \tilde{m}_H)\circ(id_V\otimes\tau_{H,H})\circ\tilde{a}_{V,H,H}\circ(\rho\circ\psi\otimes id_H)\circ\tilde{a}^{-1}_{V,H,H}\circ(id_V\otimes\tau_{H,H})\circ(id_V\otimes \tilde{\Delta}_H)\end{equation*}
 holds, where $\psi'$ is the diagonal action of $\tilde{H}$ on $V\otimes H$. This condition is expressed elementwise as follows, for $h\in H$ and $v\in V$,
 \begin{equation}\label{Y-D-condition}v_{(0)}\lhd \alpha^{-1}(h_1)\otimes\alpha(v_{(1)})h_2=(v\lhd h_2)_{(0)}\otimes h_1(v\lhd h_2)_{(1)},\end{equation}
 if we write $\psi(v\otimes h)=v\lhd h$ and $\rho(v)=v_{(0)}\otimes v_{(1)}$.
\end{definition}
The category of right-right Yetter-Drinfel'd modules together with those morphisms in $\mathcal{\widetilde{H}(C)}$ that are both $\tilde{H}$-linear and $\tilde{H}$-colinear is indicated by $\mathcal{YD}^H_H$.

\begin{theorem}Let $(V,\nu)$ $\in$ $\mathcal{\widetilde{H}}(\mathcal{C})$ and let $(H\otimes V,\alpha\otimes\nu)$ $\in$ $^{H}_H\mathcal{\widetilde{H}}(\mathcal{M})$ with the canonical $\tilde{H}$-module and $\tilde{H}$-comodule structures. Then there is a one-to-one correspondence  between
 \begin{enumerate}
 \item right $\tilde{H}$-module structures and right $\tilde{H}$-comodule structures making $(H\otimes V,\alpha\otimes\nu)$ bicovariant $\tilde{H}$-bimodule,  \item   right-right Yetter-Drinfel'd module structures on $(V,\nu)$.
 \end{enumerate}
\end{theorem}

\begin{proof} The right $\tilde{H}$-module structure $v\otimes h\mapsto v\lhd h$ and the right $\tilde{H}$-comodule structure $v\mapsto v_{(0)}\otimes v_{(1)}$ on $(V,\nu)$ are induced by the correspondences in (\ref{bijection-left-cov-bimod}) and (\ref{bijection-left-cov-bicomod}). What is left to finish the proof is to show the equivalence of the right $\tilde{H}$-Hopf module condition on $H\otimes V$ to the compatibility condition (\ref{Y-D-condition}) on $V$. Let's write $\phi':(H\otimes V)\otimes H\to H\otimes V$ for the diagonal right action on $H\otimes V$, $\phi'':((H\otimes V)\otimes H)\otimes H\to (H\otimes V)\otimes H$ for the diagonal right action on $(H\otimes V)\otimes H$ and $\sigma':H\otimes V\to (H\otimes V)\otimes H$ for the codiagonal right coaction on $H\otimes V$. So we have, for $g,h \in H$ and $v \in V$,

\begin{eqnarray}\label{L-right-Hopf-mod-cond}\nonumber\sigma'(\phi'((g\otimes v)\otimes h))&=&\sigma'(\alpha(g)h_1\otimes v\lhd\alpha^{-1}(h_2))\\
\nonumber&=&(\alpha^{-1}((\alpha(g)h_1)_1)\otimes(v\lhd\alpha^{-1}(h_2))_{(0)})\otimes(\alpha(g)h_1)_2\alpha((v\lhd\alpha^{-1}(h_2))_{(1)})\\
\nonumber&=&(g_1\alpha^{-1}(h_{11})\otimes(v\lhd\alpha^{-1}(h_2))_{(0)})\otimes(\alpha(g_2)h_{12})\alpha((v\lhd\alpha^{-1}(h_2))_{(1)})\\
\nonumber&=&(g_1\alpha^{-1}(h_1)\otimes(v\lhd\alpha^{-1}(h_{22}))_{(0)})\otimes(\alpha(g_2)h_{21})\alpha((v\lhd\alpha^{-1}(h_{22}))_{(1)})\\
\nonumber&=&(g_1\alpha^{-1}(h_1)\otimes(v\lhd\alpha^{-1}(h_{22}))_{(0)})\otimes\alpha(g_2)(h_{21}\alpha((v\lhd\alpha^{-1}(h_{22}))_{(1)}))\\
&=&(g_1\alpha^{-1}(h_1)\otimes(v\lhd\alpha^{-1}(h_2)_2)_{(0)})\otimes\alpha(g_2)\alpha(\alpha^{-1}(h_2)_1(v\lhd\alpha^{-1}(h_2)_2)_{(1)}),
\end{eqnarray}

\begin{eqnarray}\label{R-right-Hopf-mod-cond}\nonumber\lefteqn{\phi''((\sigma'\otimes id_H)((g\otimes v)\otimes h))}\hspace{5em}\\
\nonumber&=&\phi''(((\alpha^{-1}(g_1)\otimes v_{(0)})\otimes g_2\alpha(v_{(1)}))\otimes h)\\
\nonumber&=&(\alpha^{-1}(g_1)\otimes v_{(0)})\cdot\alpha^{-1}(h_1)\otimes\alpha(g_2\alpha(v_{(1)}))h_2\\
\nonumber&=&(\alpha(\alpha^{-1}(g_1))\alpha^{-1}(h_1)_1\otimes v_{(0)}\lhd\alpha^{-1}(\alpha^{-1}(h_1)_2))\otimes(\alpha(g_2)\alpha^{2}(v_{(1)}))h_2\\
\nonumber&=&(g_1\alpha^{-1}(h_{11})\otimes v_{(0)}\lhd\alpha^{-1}(\alpha^{-1}(h_{12})))\otimes\alpha(g_2)(\alpha^{2}(v_{(1)})h_2)\\
\nonumber&=&(g_1\alpha^{-1}(h_1)\otimes v_{(0)}\lhd\alpha^{-1}(\alpha^{-1}(h_{21})))\otimes\alpha(g_2)(\alpha^{2}(v_{(1)})h_{22})\\
&=&(g_1\alpha^{-1}(h_1)\otimes v_{(0)}\lhd\alpha^{-1}(\alpha^{-1}(h_2)_1))\otimes\alpha(g_2)\alpha(\alpha(v_{(1)})\alpha^{-1}(h_2)_2).
\end{eqnarray}
Thereby if the condition (\ref{Y-D-condition}) on $(V,\nu)$ holds then the right hand sides of (\ref{L-right-Hopf-mod-cond}) and (\ref{R-right-Hopf-mod-cond}) are equal, and thus the left hand sides of (\ref{L-right-Hopf-mod-cond}) and (\ref{R-right-Hopf-mod-cond}) are equal, that is, the requirement that $(H\otimes V,\alpha\otimes\nu)$ be a right Hopf module is fulfilled. Conversely, if we assume that $H\otimes V$ is a right $\tilde{H}$-Hopf module, then by applying $ (\varepsilon_H\otimes(id_V\otimes id_H))\circ\tilde{a}_{H,V,H}$ to the equation $(\sigma'\circ\phi')((1_H\otimes v)\otimes h)=(\phi''\circ(\sigma'\otimes id_H))((1_H\otimes v)\otimes h)$ we obtain the condition (\ref{Y-D-condition}) on $V$.
\end{proof}

\section{Left-Covariant Hom-Bimodules}

\begin{definition}Let $(A,\alpha)$ and $(B,\beta)$ be two monoidal Hom-algebras. A \emph{left} $(A,\alpha)$, \emph{right} $(B,\beta)$ \emph{Hom-bimodule} consists of an object $(M,\mu) \in \widetilde{\mathcal{H}}(\mathcal{M}_k)$ together with morphisms $\phi:A\otimes M\to M$, $\phi(a\otimes m)=am$ and $\varphi:M\otimes B\to M$, $\varphi(m\otimes b)=mb$, in $\widetilde{\mathcal{H}}(\mathcal{M}_k)$ such that
\begin{equation}\alpha(a)(a'm)=(aa')\mu(m) \qquad and \qquad 1_Am=\mu(m), \end{equation}
\begin{equation}(mb')\beta(b)=\mu(m)(b'b) \qquad and \qquad m1_B=\mu(m) \end{equation}
with the compatibility condition
\begin{equation}(am)\beta(b)=\alpha(a)(mb) \end{equation}
for all $a\in A$, $b\in B$ and $m\in M$.
The fact that $\phi$ and $\varphi$ are morphisms in $\widetilde{\mathcal{H}}(\mathcal{M}_k)$ means that
\begin{equation}\mu(am)=\alpha(a)\mu(m), \end{equation}
\begin{equation}\mu(mb)=\mu(m)\beta(b) \end{equation}
respectively. A morphism $f:(M,\mu)\to (N,\nu)$ in $\widetilde{\mathcal{H}}(\mathcal{M}_k)$ is called a morphism of $[(A,\alpha),(B,\beta)]$-Hom-bimodules if it preserves both $A$-action and $B$-action, that is, \begin{equation} f(am)=af(m)\end{equation} and \begin{equation}f(mb)=f(m)b ,\end{equation} respectively, and the following property is satisfied \begin{equation}f((am)\beta(b))=\alpha(a)(f(m)b),\end{equation} for all $a\in A$, $b\in B$ and $m\in M$.
\end{definition}

\begin{lemma}\label{adjoint-Hom-actions}Let $(H,\alpha)$ be a monoidal Hom-Hopf algebra and $(M,\mu)$ a $(H,\alpha)$-Hom-bimodule. For $h \in H$ and $m\in M$,
\begin{enumerate}
\item the linear map $$M\otimes H\to M,\: m\otimes h\mapsto \widetilde{ad}_R(h)(m)=(S(h_1)\mu^{-1}(m))\alpha(h_2)$$ defines a right $(H,\alpha)$-Hom-module structure on $(M,\mu)$, and
    \item the linear mapping $$H\otimes M\to M,\: h\otimes m\mapsto \widetilde{ad}_L(h)(m)=\alpha(h_1)(\mu^{-1}(m)S(h_2))$$ gives $(M,\mu)$ a left $(H,\alpha)$-Hom-module structure.
\end{enumerate}
\end{lemma}

\begin{proof}
\begin{enumerate}
\item\label{proof-right-adj-Hom-act} We first set $m\lhd h=\widetilde{ad}_R(h)(m)$ for $h \in H$ and $m\in M$. Let $g$ also be in $H$, then
\begin{eqnarray*} \mu(m)\lhd (hg)&=&(S((hg)_1)\mu^{-1}(\mu(m)))\alpha((hg)_2)\\
&=&(S(g_1)S(h_1)m)(\alpha(h_2)\alpha(g_2))\\
&=&(\alpha(S(g_1))(S(h_1)\mu^{-1}(m)))(\alpha(h_2)\alpha(g_2))\\
&=&\alpha^{2}(S(g_1))((S(h_1)\mu^{-1}(m))(h_2g_2))\\
&=&\alpha^{2}(S(g_1))(((\alpha^{-1}(S(h_1))\mu^{-2}(m))h_2)\alpha(g_2))\\
&=&(\alpha(S(g_1))((\alpha^{-1}(S(h_1))\mu^{-2}(m))h_2))\alpha^{2}(g_2))\\
&=&(S(\alpha(g_1))\alpha^{-1}((S(h_1)\mu^{-1}(m))\alpha(h_2)))\alpha^{2}(g_2))\\
&=&(S(\alpha(g_1))\alpha^{-1}(m\lhd h))\alpha(\alpha(g_2))\\
&=&(m\lhd h)\lhd \alpha(g).
\end{eqnarray*}
$m\lhd 1_H=(S(1_H)\mu^{-1}(m))\alpha(1_H)=(1_H\mu^{-1}(m))1_H=m1_H=\mu(m)$, which finishes the proof.
\item The proof is carried out as in (1).
\end{enumerate}
\end{proof}

\begin{remark}\label{adjoint-Hom-actions_remark} Since a monoidal Hom-Hopf algebra $(H,\alpha)$ is a $(H,\alpha)$-Hom bimodule, by taking $(M,\mu)$ as $(H,\alpha)$ in the above lemma, the mappings   $\widetilde{ad}_R$ and $\widetilde{ad}_L$ give us the so-called {\it right} and {\it left adjoint Hom-action} of $(H,\alpha)$ on itself, respectively.
\end{remark}

\begin{definition} Let $(B,\beta)$ be a monoidal Hom-bialgebra. A right $(B,\beta)$-{\it Hom-module algebra} $(A,\alpha)$ is a monoidal Hom-algebra and a right $(B,\beta)$-Hom-module with a Hom-action $\rho_{A}: A\otimes B \to A,\: a\otimes b \mapsto a\lhd b$ such that, for any $a,a' \in A$ and $b\in B$
\begin{equation}\label{module-Hom-algebra-conds}(aa')\lhd b =(a\lhd b_1)(a'\lhd b_2)\:\: and \:\: 1_A\lhd b=\varepsilon(b)1_A. \end{equation}
\end{definition}

\begin{proposition} The right adjoint Hom-action $\widetilde{ad}_R$ (resp. the left adjoint Hom-action $\widetilde{ad}_L$ ) turns the monoidal Hom-Hopf algebra $(H,\alpha)$ into a right $(H,\alpha)$-Hom-module algebra (resp. a left  $(H,\alpha)$-Hom-module algebra).
\end{proposition}

\begin{proof} We prove only the case of $\widetilde{ad}_R$. Since we have already verified in the Lemma (\ref{adjoint-Hom-actions}) and Remark (\ref{adjoint-Hom-actions_remark}) that $\widetilde{ad}_R$ determines a right $(H,\alpha)$-Hom-module structure on itself, we are left to prove that the conditions in (\ref{module-Hom-algebra-conds}) are accomplished: In fact,

\begin{eqnarray*}(g\lhd k_1)(h\lhd k_2)&=&((S(k_{11})\alpha^{-1}(g))\alpha(k_{12}))((S(k_{21})\alpha^{-1}(h))\alpha(k_{22}))\\
&=&((S(k_{11})\alpha^{-1}(g))\alpha(k_{12}))(\alpha(S(k_{21}))(\alpha^{-1}(h)k_{22}))\\
&=&(S(\alpha(k_{11}))g)(\alpha(k_{12})(S(k_{21})(\alpha^{-2}(h)\alpha^{-1}(k_{22}))))\\
&=&(S(\alpha(k_{11}))g)((k_{12}S(k_{21}))(\alpha^{-1}(h)k_{22}))\\
&=&(S(k_1)g)((\alpha(k_{211})S(\alpha(k_{212})))(\alpha^{-1}(h)k_{22}))\\
&=&(S(k_1)g)(\alpha(\varepsilon(k_{21})1_H)(\alpha^{-1}(h)k_{22}))\\
&=&(S(k_1)g)(1_H(\alpha^{-1}(h)\alpha^{-1}(k_2)))\\
&=&(S(k_1)g)(hk_2)\\
&=&\alpha(S(k_1))((\alpha^{-1}(g)\alpha^{-1}(h))k_2)\\
&=&(S(k_1)\alpha^{-1}(gh))\alpha(k_2)\\
&=&(gh)\lhd k,
\end{eqnarray*}
where the fifth line is a consequence of
\begin{equation}h_1\otimes h_{211}\otimes h_{212}\otimes h_{22}=\alpha(h_{11})\otimes\alpha^{-1} (h_{12})\otimes\alpha^{-1}( h_{21})\otimes h_{22},\end{equation}
which follows from the relation
\begin{equation}(id\otimes(\Delta\otimes id))\circ(id\otimes\Delta)\circ\Delta=(id\otimes\tilde{a}^{-1}_{H,H,H})\circ\tilde{a}_{H,H,H\otimes H}\circ(id_{H\otimes H}\otimes\Delta)\circ(\Delta\otimes id)\circ\Delta,\end{equation}
and
\begin{eqnarray*}1_H\lhd h&=&(S(h_1)\alpha^{-1}(1_H))\alpha(h_2)=\alpha(S(h_1))\alpha(h_2)\\
&=&\alpha(\varepsilon(h)1)=\varepsilon(h)1_H.
\end{eqnarray*}
In the case of $\widetilde{ad}_L$, similar computations are performed.
\end{proof}

\begin{definition}A \emph{left-covariant} $(H,\alpha)$-\emph{Hom-bimodule} is an $(H,\alpha)$-Hom-bimodule $(M,\mu) \in \widetilde{\mathcal{H}}(\mathcal{M}_k)$ which is a left $(H,\alpha)$-Hom-comodule, with Hom-coaction $\rho: M\to H\otimes M$, $m\mapsto m_{(-1)}\otimes m_{(0)}$, in $\widetilde{\mathcal{H}}(\mathcal{M}_k)$ such that
\begin{equation}\rho((hm)\alpha(g))=\Delta(\alpha(h))(\rho(m)\Delta(g)).\end{equation}
\end{definition}

We here recall the left coinvariant of $(H,\alpha)$ on $(M,\mu)$ for a left $(H,\alpha)$-Hom-Hopf module $(M,\mu)$, $^{coH}M=\{m\in M|\rho(m)=1_H\otimes \mu^{-1}(m)\}$, which is in $\widetilde{\mathcal{H}}(\mathcal{M}_k)$.

\begin{lemma}Let $(M,\mu)$ be a left-covariant Hom-bimodule over $(H,\alpha)$. There exists a unique $k$-linear projection $P_L:M\longrightarrow\: ^{coH}M$, $m\mapsto S(m_{(-1)})m_{(0)}$, in $\widetilde{\mathcal{H}}(\mathcal{M}_k)$, such that, for all $h\in H$ and $m\in M,$
\begin{equation}\label{structure-eqn-left-cov-Hom-bimod1}P_L(hm)=\varepsilon(h)\mu(P_L(m)). \end{equation}
We also have the following relations
\begin{equation}\label{structure-eqn-left-cov-Hom-bimod2}m=m_{(-1)}P_L( m_{(0)}), \end{equation}
\begin{equation}\label{structure-eqn-left-cov-Hom-bimod3}P_L(mh)=\widetilde{ad}_R(h)P_L(m). \end{equation}
\end{lemma}

\begin{proof} We show that $P_L(m)$ is in $^{coH}M: $ Indeed,
\begin{eqnarray*}\rho(P_L(m))&=&\rho(S(m_{(-1)})m_{(0)})=(S(m_{(-1)})m_{(0)})_{(-1)}\otimes(S(m_{(-1)})m_{(0)})_{(0)}\\
&=&S(m_{(-1)})_1m_{(0)(-1)}\otimes S(m_{(-1)})_2m_{(0)(0)}\\
&=&S(m_{(-1)2})m_{(0)(-1)}\otimes S(m_{(-1)1})m_{(0)(0)}\\
&=&S(\alpha(m_{(0)(-1)1}))\alpha(m_{(0)(-1)2})\otimes S(\alpha^{-1}(m_{(-1)})m_{(0)(0)}\\
&=&\alpha(S(m_{(0)(-1)1})m_{(0)(-1)2})\otimes \alpha^{-1}(S(m_{(-1)}))m_{(0)(0)}\\
&=&\alpha(\varepsilon(m_{(0)(-1)}1_H))\otimes \alpha^{-1}(S(m_{(-1)}))m_{(0)(0)}\\
&=&1_H\otimes \alpha^{-1}(S(m_{(-1)}))\varepsilon(m_{(0)(-1)})m_{(0)(0)}\\
&=&1_H\otimes \alpha^{-1}(S(m_{(-1)}))\mu^{-1}(m_{(0)})\\
&=&1_H\otimes \mu^{-1}(S(m_{(-1)}m_{(0)}))=1_H\otimes \mu^{-1}(P_L(m)),
\end{eqnarray*}
where in the fifth equality we have used
\begin{equation}m_{(-1)1}\otimes m_{(-1)2}\otimes m_{(0)(-1)}\otimes m_{(0)(0)}=\alpha^{-1}(m_{(-1)})\otimes \alpha(m_{(0)(-1)1}) \otimes \alpha(m_{(0)(-1)2})\otimes m_{(0)(0)},\end{equation}
which results from the fact that the following relation holds:
\begin{equation}(\Delta\otimes id)\circ(id\otimes\rho)\circ\rho=\tilde{a}^{-1}_{H,H,H\otimes M}\circ (id\otimes\tilde{a}_{H,H,M})\circ(id\otimes(\Delta\otimes id))\circ(id\otimes\rho)\circ \rho.\end{equation}
Now we prove that $M=H\cdot ^{coH}M$
\begin{eqnarray*} m_{(-1)}P_L(m_{(0)})&=& m_{(-1)}(S(m_{(0)(-1)})m_{(0)(0)})\\
&=&(\alpha^{-1}(m_{(-1)})S(m_{(0)(-1)}))\mu(m_{(0)(0)})\\
&=& (m_{(-1)1}S(m_{(-1)2}))m_{(0)}\\
&=&\varepsilon(m_{(-1)}1_Hm_{(0)}\\
&=&\mu(\varepsilon(m_{(-1)}m_{(0)})\\
&=&\mu(\mu^{-1}(m))=m,
\end{eqnarray*}
where we have used the Hom-coassociativity condition for the left Hom-comodules in the third equation.
\begin{eqnarray*} P_L(hm)&=& S(h_1m_{(-1)})(h_2m_{(0)})\\
&=&(S(m_{(-1)})S(h_1))(h_2m_{(0)})\\
&=& \alpha((S(m_{(-1)}))(S(h_1)(\alpha^{-1}(h_2)\mu^{-1}m_{(0)}))\\
&=&\alpha((S(m_{(-1)}))((\alpha^{-1}(S(h_1))\alpha^{-1}(h_2))m_{(0)})\\
&=&\alpha((S(m_{(-1)}))((\alpha^{-1}(S(h_1)h_2)m_{(0)})\\
&=&\varepsilon(h)(S(m_{(-1)})1_H)\mu(m_{0}\\
&=&\varepsilon(h)\mu(S(m_{(-1)}m_{0}))=\varepsilon(h)\mu(P_L(m)).
\end{eqnarray*}

\begin{eqnarray*} P_L(mh)&=& S(m_{(-1)}h_1)(m_{(0)}h_2)\\
&=&(S(h_1)S(m_{(-1)}))(m_{(0)}h_2)\\
&=&((\alpha^{-1}(S(h_1))\alpha^{-1}(S(m_{(-1)})))m_{(0)})\alpha(h_2)\\
&=&(S(h_1)(\alpha^{-1}(S(m_{(-1)}))\mu^{-1}(m_{(0)})))\alpha(h_2)\\
&=&(S(h_1)\mu^{-1}(S(m_{(-1)})m_{(0)}))\alpha(h_2)\\
&=&\widetilde{ad}_R(h)P_L(m).
\end{eqnarray*}
If $m$ belongs to $^{coH}M$, then
$$P_L(m)=S(1_H)\mu^{-1}(m)=m$$
proving that $P_L$ is a $k$-projection of $M$ onto $^{coH}M$.
Let $P'_L:M\longrightarrow ^{coH}M$ be another $k$-projection such that $P'_L(hm)=\varepsilon(h)\mu(P'_L(m))$, then, by the fact that $P'_L$ is a morphism in $\widetilde{\mathcal{H}}(\mathcal{M}_k),$ we have
\begin{eqnarray*}P'_L(m)&=&P'_L(m_{(-1)}P_L( m_{(0)}))=\varepsilon(m_{(-1)})\mu(P'_L(P_L( m_{(0)}))\\
&=&\varepsilon(m_{(-1)})\mu(P_L( m_{(0)}))=P_L(\mu(\varepsilon(m_{(-1)})m_{(0)}))\\
&=&P_L(\mu(\mu^{-1}(m)))=P_L(m),\end{eqnarray*}
which shows the uniqueness of $P_L.$
\end{proof}
\begin{proposition}\label{left-cov-structure}Let $(N,\nu)\in\widetilde{\mathcal{H}}(\mathcal{M}_k)$ be a right $(H,\alpha)$-Hom-module by the $H$-action $N\otimes H\to N,\: n\otimes h\mapsto n\lhd h$. The following morphisms
\begin{equation}\label{left-Hom-action}H\otimes(H\otimes N)\to H\otimes N,\: h\otimes(g\otimes n)\mapsto \alpha^{-1}(h)g\otimes \nu(n),  \end{equation}
\begin{equation}\label{right-Hom-action}(H\otimes N)\otimes H\to H\otimes N,\: (h\otimes n)\otimes g\mapsto hg_1\otimes n\lhd g_2,  \end{equation}
\begin{equation}\label{left-Hom-coaction}\rho:H\otimes N\to H\otimes(H\otimes N),\: h\otimes n\mapsto \alpha(h_1)\otimes(h_2\otimes\nu^{-1}(n)). \end{equation}
in $\widetilde{\mathcal{H}}(\mathcal{M}_k)$, define a left-covariant $(H,\alpha)$-Hom-bimodule structure on $(H\otimes N , \alpha\otimes\nu)$.
\end{proposition}

\begin{proof} We verify the Hom-associativity and Hom-unity conditions for the left and the right Hom-multiplications of $(H,\alpha)$ on $(H\otimes N,\alpha\otimes\nu)$ , respectively: For all $h,k,g \in H$ and $n\in N$, we get

\begin{eqnarray*}\alpha(k)(h(g\otimes n))&=&\alpha(k)(\alpha^{-1}(h)g\otimes \nu(n))=k((\alpha^{-1}(h)g)\otimes \nu^{2}(n) \\
&=&\alpha^{-1}(kh)\alpha(g)\otimes \nu^{2}(n)=(kh)((\alpha\otimes\nu)(g\otimes n)),
\end{eqnarray*}
$$1_H(g\otimes n)=\alpha^{-1}(1_H)g\otimes \nu(n)=\alpha(g)\otimes \nu(n)=(\alpha\otimes\nu)(g\otimes n),$$
\begin{eqnarray*}((\alpha\otimes\nu)(h\otimes n))(gk)&=&\alpha(h)(g_1k_1)\otimes \nu(n)\lhd(g_2k_2)=(hg_1)\alpha(k_1)\otimes (n\lhd g_2)\lhd \alpha(k_2)\\
&=&(hg_1\otimes n\lhd g_2)\alpha(k)=((h\otimes n)g)\alpha(k),
\end{eqnarray*}
$$(h\otimes n)1_H=h1_H\otimes n\lhd 1_H=(\alpha\otimes\nu)(h\otimes n).$$

We now show that the compatibility condition is satisfied:
\begin{eqnarray*}(g(h\otimes n))\alpha(k)&=&(\alpha^{-1}(g)h\otimes \nu(n))\alpha(k)=(\alpha^{-1}(g)h)\alpha(k_1)\otimes \nu(n)\lhd\alpha(k_2)\\
&=&g(hk_1)\otimes \nu(n)\lhd\alpha(k_2)=\alpha^{-1}(\alpha(g))(hk_1)\otimes \nu(n\lhd k_2)\\
&=&\alpha(g)(hk_1\otimes n\lhd k_2)=\alpha(g)((h\otimes n )k).
\end{eqnarray*}

$\rho$ satisfies the Hom-coassociativity and Hom-counity condition: Indeed, on one hand we have
\begin{eqnarray*}\Delta((h\otimes n)_{(-1)})\otimes (\alpha^{-1}\otimes\nu^{-1})((h\otimes n)_{(0)})&=&\Delta(\alpha(h_1))\otimes (\alpha^{-1}\otimes\nu^{-1})(h_2\otimes \nu^{-1}(n))\\
&=&(\alpha(h_{11})\otimes\alpha(h_{12}))\otimes(\alpha^{-1}(h_2)\otimes\nu^{-2}(n))\\
&=&(h_1\otimes\alpha(h_{21}))\otimes(h_{22}\otimes\nu^{-2}(n))\\
&=&(\alpha^{-1}((h\otimes n)_{(-1)})\otimes(h\otimes n)_{(0)(-1)} )\otimes(h\otimes n)_{(0)(0)},
\end{eqnarray*}
in the third equality we have used the relation
 \begin{equation}\label{relation-left-coass-tens-pro}\alpha(h_{11})\otimes\alpha(h_{12})\otimes h_2\otimes \nu^{-1}(n)=h_1\otimes \alpha(h_{21})\otimes \alpha(h_{22})\otimes \nu^{-1}(n), \end{equation}
which follows from
$$(\Delta\otimes id)\circ\rho=\tilde{a}^{-1}_{H,H,H\otimes N}\circ (id\otimes\tilde{a}_{H,H,N})\circ(id\otimes(\Delta\otimes id))\circ \rho.$$
On the other hand,
\begin{eqnarray*}\varepsilon((h\otimes n)_{(-1)})(h\otimes n)_{(0)}&=&\varepsilon(\alpha(h_1))(h_2\otimes \nu^{-1}(n))\\
&=&\varepsilon(h_1)h_2\otimes \nu^{-1}(n)=(\alpha^{-1}\otimes \nu^{-1})(h\otimes n).
\end{eqnarray*}

To finish the proof of the fact that the above Hom-actions and Hom-coaction of $(H,\alpha)$ on $H\otimes N$ define a left-covariant $(H,\alpha)$-Hom-bimodule structure on $(H\otimes N, \alpha\otimes\nu)$ we show that the following relation holds:

\begin{eqnarray*} \Delta(\alpha(g))(\rho(h\otimes n)\Delta(k))&=&(\alpha(g_1)\otimes\alpha(g_2))((\alpha(h_1)\otimes(h_2\otimes\nu^{-1}(n)))(k_1\otimes k_2))\\
&=&(\alpha(g_1)\otimes\alpha(g_2))(\alpha(h_1)k_1\otimes(h_2\otimes\nu^{-1}(n))k_2)\\
&=&\alpha(g_1)(\alpha(h_1)k_1)\otimes\alpha(g_2)((h_2\otimes\nu^{-1}(n))k_2)\\
&=&\alpha(g_1)(\alpha(h_1)k_1)\otimes(\alpha^{-1}(\alpha(g_2))(h_2k_{21})\otimes\nu(\nu^{-1}(n)\lhd k_{22}))\\
&=&\alpha(g_1)(\alpha(h_1)k_1)\otimes(g_2(h_2k_{21})\otimes n\lhd \alpha( k_{22}))\\
&=&\alpha(g_1)(\alpha(h_1)\alpha(k_{11}))\otimes(g_2(h_2k_{12})\otimes n\lhd k_2)\\
&=&(g_1\alpha(h_1))\alpha^{2}(k_{11})\otimes((\alpha^{-1}(g_2)h_2)\alpha(k_{12})\otimes\nu^{-1}(\nu(n)\lhd \alpha(k_2)))\\
&=&\alpha(((\alpha^{-1}(g)h)\alpha(k_1))_1)\otimes(((\alpha^{-1}(g)h)\alpha(k_1))_2\otimes\nu^{-1}(\nu(n)\lhd \alpha(k_2)))\\
&=&\rho((\alpha^{-1}(g)h)\alpha(k_1)\otimes\nu(n)\lhd \alpha(k_2))\\
&=&\rho((\alpha^{-1}(g)h\otimes\nu(n))\alpha(k))\\
&=&\rho((g(h\otimes n))\alpha(k)).
\end{eqnarray*}
\end{proof}

\begin{proposition}\label{fundamental-thm-left-cov-bimod}If $(M,\mu) \in\widetilde{\mathcal{H}}(\mathcal{M}_k) $ is a left-covariant $(H,\alpha)$-Hom-bimodule, the $k$-linear map
\begin{equation}\label{left-cov-iso}\theta:H\otimes \:^{coH}M \longrightarrow M, \: h\otimes m\mapsto hm, \end{equation}
in $\widetilde{\mathcal{H}}(\mathcal{M}_k)$ is an isomorphism of left-covariant $(H,\alpha)$-Hom-bimodules, where the right $(H,\alpha)$-Hom-module structure on
$(^{coH}M,\mu|_{^{coH}M})$ is defined by \begin{equation}\label{right_action_on_coinvariants}m\lhd h := P_L(mh)=\widetilde{ad}_R(h)m,\end{equation} for $h\in H$ and $m$ $\in$ $^{coH}M$.
\end{proposition}

\begin{proof} Define $\vartheta: M\to H\otimes\: ^{coH}M$ as follows: For any $m\in M$
 $$\vartheta(m)=m_{(-1)}\otimes S(m_{(0)(-1)})m_{(0)(0)},$$
which is shown that $\vartheta$ is the inverse of $\theta$: Indeed,
\begin{eqnarray*}\theta(\vartheta(m))&=&\theta(m_{(-1)}\otimes S(m_{(0)(-1)})m_{(0)(0)})\\
&=&m_{(-1)}(S(m_{(0)(-1)})m_{(0)(0)})=\alpha(m_{(-1)1})(S(m_{(-1)2})\mu^{-1}(m_{(0)}))\\
&=&(m_{(-1)1}S(m_{(-1)2}))m_{(0)}=\mu(\varepsilon(m_{-1})m_{0})=\mu(\mu^{-1}(m))=m,
\end{eqnarray*}
where in the second equation we have used the Hom-coassociativity condition for $(M,\mu)$ to be a left Hom-comodule. On the other hand, for $m\in \:^{coH}M$ and $h\in H$ we obtain
\begin{eqnarray*}\vartheta(\theta(h\otimes m))&=&\vartheta(hm)\\
&=&h_11_H\otimes S((h_2\mu^{-1}(m))_{(-1)})(h_2\mu^{-1}(m))_{(0)}\\
&=&\alpha(h_1)\otimes S(h_{21}\mu^{-1}(m)_{(-1)})(h_{22}\mu^{-1}(m)_{(0)})\\
&=&\alpha(h_1)\otimes S(h_{21}1_H)(h_{22}\mu^{-2}(m))\\
&=&\alpha(h_1)\otimes \alpha(S(h_{21}))(h_{22}\mu^{-2}(m))\\
&=&\alpha(h_1)\otimes (S(h_{21})h_{22})\mu^{-1}(m)\\
&=&\alpha(h_1)\otimes \varepsilon(h_2)1_H\mu^{-1}(m)\\
&=&\alpha(\alpha^{-1}(h))\otimes m=h\otimes m,
\end{eqnarray*}
where in the fourth equality the fact that the Hom-coaction of $(H,\alpha)$ on $(M,\mu)$ is a morphism in $\widetilde{\mathcal{H}}(\mathcal{M}_k) $ has been used.
Now we show that $\theta$ is both $(H,\alpha)$-bilinear and left $(H,\alpha)$-colinear:
$$\theta(g(h\otimes m))=\theta(\alpha^{-1}(g)h\otimes\mu(m))=(\alpha^{-1}(g)h)\mu(m))=g(hm)=g\theta(h\otimes m),$$
\begin{eqnarray*}\theta((h\otimes m)k)&=&\theta(hk_1\otimes m\lhd k_2)=(hk_1)(\widetilde{ad}_{R}(k_2)m)\\
&=&(hk_1)((S(k_{21})\mu^{-1}(m))\alpha(k_{22}))\\
&=&(hk_1)(\alpha(S(k_{21}))(\mu^{-1}(m)k_{22}))\\
&=&((\alpha^{-1}(h)\alpha^{-1}(k_1))\alpha(S(k_{21})))(m\alpha(k_{22}))\\
&=&(h(\alpha^{-1}(k_1)S(k_{21})))(m\alpha(k_{22}))\\
&=&(h(k_{11}S(k_{12})))(mk_2)\\
&=&\alpha(h)(m\alpha^{-1}(k))=\theta(h\otimes m)k,
\end{eqnarray*}
where the penultimate line follows from the first relation of (\ref{Hom_coassociativity_Weak_counitality}). Lastly, put $^{M}\rho: M\to H\otimes M$ and $^{H\otimes \:^{coH}M}\rho:H\otimes \:^{coH}M \to H\otimes (H\otimes \:^{coH}M)$ for the left Hom-coaction of $(H,\alpha)$ on $(M,\mu)$ and $(H\otimes \:^{coH}M,\alpha\otimes\mu|_{^{coH}M})$, resp., thus

\begin{eqnarray*}^{M}\rho(\theta(h\otimes m))&=&^{M}\rho(hm)\\
&=&h_11_H\otimes h_2\mu^{-1}(m)\\
&=&\alpha(h_1)\otimes h_2\mu^{-1}(m)\\
&=&(id\otimes\theta)(\alpha(h_1)\otimes (h_2\otimes\mu^{-1}(m)))\\
&=&(id\otimes\theta)(^{H\otimes \:^{coH}M}\rho(h\otimes m)).
\end{eqnarray*}
\end{proof}

By Propositions (\ref{left-cov-structure}) and (\ref{fundamental-thm-left-cov-bimod}), we have the following

\begin{theorem}\label{one-to-one-LeftCov-RightMod} There is a bijection, given by (\ref{left-Hom-action})-(\ref{left-Hom-coaction}) and (\ref{right_action_on_coinvariants}), between left-covariant $(H,\alpha)$-Hom-bimodules $(M,\mu)$ and the right $(H,\alpha)$-Hom-module structures on $(^{coH}M,\mu|_{^{coH}M})$.
\end{theorem}

If the antipode $S$ of the monoidal Hom-Hopf algebra $(H,\alpha)$ is invertible, we have, for $m \in \: ^{coH}M $ and $h\in H$
\begin{equation}\label{left-cov-bimod-given-by-right-action}hm=(\mu^{-1}(m)\lhd S^{-1}(h_2))\alpha(h_1).\end{equation}

Indeed;
\begin{eqnarray*}(m\lhd S^{-1}(h_2))\alpha(h_1)&=&(1_H\mu^{-1}(m\lhd S^{-1}(h_2)))\alpha(h_1)\\
&=&(1_H(\mu^{-1}(m)\lhd\alpha^{-1}( S^{-1}(h_2))))\alpha(h_1)\\
&=&(1_H\alpha(h_{11}))((\mu^{-1}(m)\lhd\alpha^{-1}( S^{-1}(h_2)))\lhd\alpha(h_{12})\\
&=&\alpha^{2}(h_{11})(\mu(\mu^{-1}(m))\lhd(\alpha^{-1}( S^{-1}(h_2))\alpha^{-1}(\alpha(h_{12}))))\\
&=&\alpha(h_1)(m\lhd(\alpha^{-1}( S^{-1}(\alpha(h_{22})))h_{21})\\
&=&\alpha(h_1)(m\lhd( S^{-1}(h_{22})h_{21})\\
&=&\alpha(h_1)(m\lhd\varepsilon(h_2)1_H)=h\mu(m),
\end{eqnarray*}
which implies that $M=\: ^{coH}M\cdot H$.

We indicate by $^{H}_{H}\widetilde{\mathcal{H}}(\mathcal{M}_k)_H$ the category of left-covariant $(H,\alpha)$-Hom-bimodules;
the objects are the left-covariant Hom-bimodules and the morphisms are the ones in $\widetilde{\mathcal{H}}(\mathcal{M}_k)$ that are $(H,\alpha)$-linear on both sides and left $(H,\alpha)$-colinear.

We next show that the category $^{H}_{H}\widetilde{\mathcal{H}}(\mathcal{M}_k)_H$ of left-covariant $(H,\alpha)$-Hom-bimodules forms a monoidal category.
\begin{definition}
Let $(M,\mu)$ be a right $(A,\alpha)$-Hom-module and $(N,\nu)$ be a left $(A,\alpha)$-Hom-module. The tensor product $(M\otimes_A N, \mu\otimes\nu)$ of $(M,\mu)$ and $(N,\nu)$ over $(A,\alpha)$ is the coequalizer of
 $\rho\otimes id_N,\: (id_M\otimes \bar{\rho})\circ \tilde{a}_{M,A,N}: (M\otimes A)\otimes N\to M\otimes N$, where $\rho:M\otimes A\to M,\:m\otimes a\mapsto ma$ and $\bar{\rho}:A\otimes N\to N,\:a\otimes n\mapsto an$, for $a\in A$, $m\in M$ and $n\in N$, are the right and left Hom-actions of $(A,\alpha)$ on $(M,\mu)$ and $(N,\nu)$ respectively. That is,
\begin{equation}\label{tensor-prod-over-Hom-alg} m\otimes_A n=\{m\otimes n \in M\otimes N| \: ma\otimes n=\mu(m)\otimes a\nu^{-1}(n),\forall a \in A\}.\end{equation}
\end{definition}

 \begin{proposition}\label{tensor_prod_left-cov_Hom_bimod} Let $(H,\alpha)$ be a monoidal Hom-Hopf algebra and $(M,\mu)$, $(N,\nu)$ be
  two left-covariant $(H,\alpha)$-Hom-bimodules. Define the $k$-linear maps
 \begin{equation}\label{left_action_left-cov_bimod}H\otimes(M\otimes_H N)\to M\otimes_H N,\: h\otimes(m\otimes_H n)=\alpha^{-1}(h)m\otimes_H \nu(n), \end{equation}
 \begin{equation}\label{right_action_left-cov_bimod}(M\otimes_H N)\otimes H\to M\otimes_H N,\: (m\otimes_H n)\otimes h=\mu(m)\otimes_H n\alpha^{-1}(h), \end{equation}
 \begin{equation}\label{left-coaction-left-cov-bimod}\rho: M\otimes_H N\to H\otimes(M\otimes_H N),\: m\otimes_H n=m_{(-1)}n_{(-1)}\otimes(m_{(0)}\otimes_H n_{(0)}), \end{equation}
 Then $(M\otimes_H N, \mu\otimes_H \nu)$ becomes a left-covariant Hom-bimodule over $(H,\alpha)$ with these structures.
 \end{proposition}

\begin{proof} We first prove that the map (\ref{left_action_left-cov_bimod}) gives $M\otimes_H N$ a left $(H,\alpha)$-Hom-module structure:
\begin{eqnarray*}\alpha(g)(h(m\otimes_H n))&=&\alpha(g)(\alpha^{-1}(h)m\otimes_H \nu(n))=g(\alpha^{-1}(h)m)\otimes_H \nu^{2}(n)\\
&=&\alpha^{-1}(gh)\mu(m)\otimes_H \nu(\nu(n))=(gh)(\mu\otimes_H\nu)(m\otimes_H n),
\end{eqnarray*}
\begin{equation*}1_H(m\otimes_H n)=\alpha^{-1}(1_H)m\otimes_H\nu(n)=\mu(m)\otimes_H\nu(n).\end{equation*}
One can  prove in the same way that the map (\ref{right_action_left-cov_bimod}) makes $M\otimes_H N$ a right Hom-module. We show that the compatibility condition
is satisfied:
\begin{eqnarray*}(g(m\otimes_H n))\alpha(h)&=&(\alpha^{-1}(g)m\otimes_H \nu(n))\alpha(h)=\mu(\alpha^{-1}(g)m)\otimes_H \nu(n)h\\
&=&g\mu(m)\otimes_H \nu(n)h=\alpha^{-1}(\alpha(g))\mu(m)\otimes_H \nu(n\alpha^{-1}(h))\\
&=&\alpha(g)(\mu(m)\otimes_H n\alpha^{-1}(h))=\alpha(g)((m\otimes_H n)h).
\end{eqnarray*}
We now demonstrate that $M\otimes_H N$ possesses a left $(H,\alpha)$-Hom-comodule structure with $\rho$ which is given by
$\rho(m\otimes_H n)=m_{(-1)}n_{(-1)}\otimes(m_{(0)}\otimes_H n_{(0)})$.

\begin{eqnarray*}\lefteqn{\Delta((m\otimes_H n)_{(-1)})\otimes(\mu^{-1}\otimes_H\nu^{-1})((m\otimes_H n)_{(0)})}\hspace{3em}\\
&=&\Delta(m_{(-1)})\Delta(n_{(-1)})\otimes(\mu^{-1}(m_{(0)})\otimes_H\nu^{-1}(n_{(0)}))\\
&=&(\alpha^{-1}(m_{(-1)})\alpha^{-1}(n_{(-1)})\otimes m_{(0)(-1)}n_{(0)(-1)})\otimes(m_{(0)(0)}\otimes_H n_{(0)(0)})\\
&=&(\alpha^{-1}((m\otimes_H n)_{(-1)})\otimes(m\otimes_H n)_{(0)(-1)})\otimes(m\otimes_H n)_{(0)(0)},
\end{eqnarray*}

\begin{eqnarray*}\varepsilon((m\otimes_H n)_{(-1)})(m\otimes_H n)_{(0)}&=&\varepsilon(m_{(-1)}n_{(-1)})m_{(0)}\otimes_H n_{(0)}\\
&=&\varepsilon(m_{(-1)})m_{(0)}\otimes_H \varepsilon(n_{(-1)})n_{(0)}\\
&=&\mu^{-1}(m)\otimes_H\nu^{-1}(n),
\end{eqnarray*}
which prove the Hom-coassociativity and Hom-counity of $\rho$, respectively.
Lastly, we verify the left covariance as follows:

\begin{eqnarray*}\Delta(\alpha(g))(\rho(m\otimes_Hn)\Delta(h))&=&(\alpha(g_1)\otimes\alpha(g_2))((m_{(-1)}n_{(-1)}\otimes(m_{(0)}\otimes_H n_{(0)}))(h_1\otimes h_2))\\
&=&\alpha(g_1)((m_{(-1)}n_{(-1)})h_1)\otimes \alpha(g_2)((m_{(0)}\otimes_H n_{(0)})h_2)\\
&=&\alpha(g_1)(\alpha(m_{(-1)})(n_{(-1)}\alpha^{-1}(h_1)))\otimes\alpha(g_2)(\mu(m_{(0)})\otimes_H n_{(0)}\alpha^{-1}(h_2))\\
&=&(g_1\alpha(m_{(-1)}))(\alpha(n_{(-1)})h_1)\otimes(g_2\mu(m_{(0)})\otimes_H\nu(n_{(0)}\alpha^{-1}(h_2)))\\
&=&(g_1\mu(m)_{(-1)})(\nu(n)_{(-1)}h_1)\otimes (g_2\mu(m)_{(0)}\otimes_H\nu(n)_{(0)}h_2)\\
&=&(g\mu(m))_{(-1)}(\nu(n)h)_{(-1)}\otimes ((g\mu(m))_{(0)}\otimes_H(\nu(n)h)_{(0)})\\
&=&\rho(g\mu(m)\otimes_H\nu(n)h)\\
&=&\rho((\alpha^{-1}(g)m\otimes_H\nu(n))\alpha(h))\\
&=&\rho((g(m\otimes_Hn))\alpha(h)).
\end{eqnarray*}

\end{proof}

\begin{proposition}\label{associator_left-cov_Hom_bimod}Let $(H,\alpha)$ be a monoidal Hom-Hopf algebra and $(M,\mu)$, $(N,\nu)$,  $(P,\pi)$ be left-covariant
$(H,\alpha)$-Hom-bimodules. Then the linear map
\begin{equation}\label{associator-formula}\tilde{a}_{M,N,P}:(M\otimes_H N)\otimes_H P\to M\otimes_H(N\otimes_H P),\: \tilde{a}_{M,N,P}((m\otimes_H n)\otimes_H p)=\mu(m)\otimes_H(n\otimes_H \pi^{-1}(p)),\end{equation}
is an isomorphism of $(H,\alpha)$-Hom-bimodules and left $(H,\alpha)$-Hom-comodules.
\end{proposition}

\begin{proof} It is clear that $\tilde{a}_{M,N,P}$ is bijective and fulfills the relation $\tilde{a}_{M,N,P}\circ(\mu\otimes\nu\otimes\pi)=(\mu\otimes\nu\otimes\pi)\circ\tilde{a}_{M,N,P}$. In what follows we prove the left and right $(H,\alpha)$-linearity, and left $(H,\alpha)$-colinearity of $\tilde{a}_{M,N,P}$: The calculation

\begin{eqnarray*}\tilde{a}_{M,N,P}(h((m\otimes_Hn)\otimes_Hp))&=&\tilde{a}_{M,N,P}(\alpha^{-1}(h)(m\otimes_Hn)\otimes_H\pi(p))\\
&=&\tilde{a}_{M,N,P}((\alpha^{-2}(h)m\otimes_H\nu(n))\otimes_H\pi(p))\\
&=&\mu(\alpha^{-2}(h)m)\otimes_H(\nu(n)\otimes_Hp)\\
&=&\alpha^{-1}(h)\mu(m)\otimes_H((\nu\otimes_H\pi)(n\otimes_H\pi^{-1}(p)))\\
&=&h(\mu(m)\otimes_H(n\otimes_H\pi^{-1}(p)))\\
&=&h\tilde{a}_{M,N,P}((m\otimes_Hn)\otimes_Hp)
\end{eqnarray*}
shows that $\tilde{a}_{M,N,P}$ is left $(H,\alpha)$-linear. By performing a similar computation, one can also affirm that
$\tilde{a}_{M,N,P}(((m\otimes_Hn)\otimes_Hp)h)=\tilde{a}_{M,N,P}((m\otimes_Hn)\otimes_Hp)h$, i.e., $\tilde{a}_{M,N,P}$ is right $(H,\alpha)$-linear too.

Now we verify the left $(H,\alpha)$-colinearity of $\tilde{a}_{M,N,P}$:
\begin{eqnarray*}\lefteqn{^{Q}\rho(\tilde{a}_{M,N,P}((m\otimes_Hn)\otimes_Hp))}\hspace{6em}\\
&=&^{Q}\rho(\mu(m)\otimes_H(n\otimes_H \pi^{-1}(p)))\\
&=&\mu(m)_{(-1)}(n\otimes_H \pi^{-1}(p))_{(-1)}\otimes(\mu(m)_{(0)}\otimes_H(n\otimes_H \pi^{-1}(p))_{(0)})\\
&=&\mu(m)_{(-1)}(n_{(-1)}\otimes_H \pi^{-1}(p)_{(-1)})\otimes(\mu(m)_{(0)}\otimes_H(n_{(0)}\otimes_H \pi^{-1}(p)_{(0)})\\
&=&\alpha(m_{(-1)})(n_{(-1)}\alpha^{-1}(p_{(-1)}))\otimes(\mu(m_{(0)})\otimes_H(n_{(0)}\otimes_H \pi^{-1}(p_{(0)}))\\
&=&(m_{(-1)}n_{(-1)})p_{(-1)}\otimes\tilde{a}_{M,N,P}((m_{(0)}\otimes_Hn_{(0)})\otimes_Hp_{(0)})\\
&=&(m\otimes_Hn)_{(-1)}p_{(-1)}\otimes \tilde{a}_{M,N,P}((m_{(0)}\otimes_Hn_{(0)})\otimes_Hp_{(0)})\\
&=&(id \otimes\tilde{a}_{M,N,P})((m\otimes_Hn)_{(-1)}p_{(-1)}\otimes((m\otimes_Hn)_{(0)}\otimes_Hp_{(0)}))\\
&=&(id \otimes\tilde{a}_{M,N,P})( ^{Q'}\rho((m\otimes_Hn)\otimes_Hp)),
\end{eqnarray*}
where $^{Q}\rho$ and $^{Q'}\rho$ are the left codiagonal Hom-coactions of $(H,\alpha)$ on the objects $Q=M\otimes_H(N\otimes_HP)$ and $Q'=(M\otimes_HN)\otimes_HP$ resp.
\end{proof}

\begin{proposition}\label{units_left-cov_Hom_bimod}Let $(H,\alpha)$ be a monoidal Hom-Hopf algebra and $(M,\mu)$ be a left-covariant $(H,\alpha)$-Hom-bimodule. Then the following linear maps
\begin{equation}\label{left-unit-formula}\tilde{l}_M:H\otimes_H M\to M,\: h\otimes_H m\mapsto hm, \end{equation}
\begin{equation}\label{right-unit-formula}\tilde{r}_M:M\otimes_H H\to M,\: m\otimes_H h\mapsto mh. \end{equation}
are isomorphisms of $(H,\alpha)$-Hom-bimodules and left $(H,\alpha)$-Hom-comodules.
\end{proposition}

\begin{proof} With the left and right $(H,\alpha)$-Hom-module structures given by Hom-multiplication $H\otimes H \to H,\: h\otimes g\mapsto m_H(h\otimes g)=hg$
and the left $(H,\alpha)$-Hom-comodule structure given by Hom-multiplication $H\to H\otimes H,\: h\mapsto h_1\otimes h_2 $, $(H,\alpha)$ is a left-covariant $(H,\alpha)$-Hom-bimodule. We show only that $\tilde{l}_M$ is $(H,\alpha)$-linear on both sides and left $(H,\alpha)$-colinear.
For $\tilde{r}_M$ the argument is analogous. Obviously, $\tilde{l}_M$ is a $k$-isomorphism with the inverse
$\tilde{l}_M^{-1}:M \to H\otimes_H M, \: m\mapsto 1\otimes \mu^{-1}(m)$ and the relation $\mu\circ\tilde{l}_M=\tilde{l}_M\circ(id_H\otimes\mu)$ is satisfied. We show now left and right $(H,\alpha)$-linearity, and $(H,\alpha)$-colinearity of $\tilde{l}_M$, respectively: For any $h,g\in H$ and $m\in M$,

\begin{eqnarray*}\tilde{l}_{M}(h(g\otimes_Hm))&=&\tilde{l}_{M}(\alpha^{-1}(h)g\otimes_H\mu(m))=(\alpha^{-1}(h)g)\mu(m)\\
&=&h(gm)=h\tilde{l}_{M}(g\otimes_H m),
\end{eqnarray*}

\begin{eqnarray*}\tilde{l}_{M}((g\otimes_H m)h)&=&\tilde{l}_{M}(\alpha(g)\otimes_H m\alpha^{-1}(h))\\
&=&\alpha(g)(m\alpha^{-1}(h))=(gm)h=\tilde{l}_{M}(g\otimes_Hm)h,
\end{eqnarray*}

\begin{eqnarray*}(id_H\otimes \tilde{l}_{M})(^{H\otimes_HM}\rho(h\otimes_Hm))&=&(id_H\otimes \tilde{l}_{M})(h_{(-1)}m_{(-1)}\otimes(h_{(0)}\otimes_Hm_{(0)}))\\
&=&(id_H\otimes \tilde{l}_{M})(h_1m_{(-1)}\otimes(h_2\otimes_Hm_{(0)}))\\
&=&h_1m_{(-1)}\otimes h_2m_{(0)}=h\:^{M}\rho(m)\\
&=&^{M}\rho(hm)=\:^{M}\rho(\tilde{l}_{M}(h\otimes_H m)),
\end{eqnarray*}
where $^{H\otimes_HM}\rho$ and $^{M}\rho$ are the left Hom-coactions of $(H,\alpha)$ on the objects $H\otimes_HM$ and $M$, respectively.
\end{proof}

\begin{theorem} Let $(H,\alpha)$ be a monoidal Hom-Hopf algebra. Then the category $^{H}_{H}\widetilde{\mathcal{H}}(\mathcal{M}_k)_H$ of left-covariant $(H,\alpha)$-Hom-bimodules forms a monoidal category, with tensor product $\otimes_H$, associativity constraints $\tilde{a}$, and left and right unity constraints $\tilde{l}$ and $\tilde{r}$, defined in Propositions \ref{tensor_prod_left-cov_Hom_bimod}, \ref{associator_left-cov_Hom_bimod} and \ref{units_left-cov_Hom_bimod}, respectively.
\end{theorem}

\begin{proof}
The naturality of $\tilde{a}$ and the fact that $\tilde{a}$ satisfies the Pentagon Axiom follow from Proposition $1.1$ in \cite{CaenepeelGoyvaerts}. Let $f:M\to M'$ be a morphism in $^{H}_{H}\widetilde{\mathcal{H}}(\mathcal{M}_k)_H$ and let $(M,\mu)$ be a left-covariant $(H,\alpha)$-Hom-bimodule. Then, for $m\in M$ and $h\in H$, we have
$$(f\circ\tilde{l}_M)(h\otimes_H m)=f(hm)=hf(m)=\tilde{l}_{M'}(h\otimes_H f(m)),$$
showing that $\tilde{l}$ is natural. The naturality of $\tilde{r}$ can be proven similarly. We finally verify that the Triangle Axiom is satisfied: For $h\in H$, $m\in M$ and $n\in N$,

\begin{eqnarray*}((id_M\otimes_H\tilde{l}_N)\circ\tilde{a}_{M,H,N})((m\otimes_H h)\otimes_H n)
&=&(id_M\otimes_H\tilde{l}_N)(\mu(m)\otimes_H(h\otimes_H\nu^{-1}(n)))\\
&=&\mu(m)\otimes_H h\nu^{-1}(n)=mh\otimes_H n\\
&=&\tilde{r}_M(m\otimes_H h)\otimes_H n=(\tilde{r}_M\otimes id_N)((m\otimes_H h)\otimes_H n).
\end{eqnarray*}
\end{proof}
In the rest of the section, we study the structure theory of left-covariant Hom-bimodules in coordinate setting.

Let $(H,\alpha)$ be a monoidal Hom-coalgebra with Hom-comultiplication $\Delta:H\to H\otimes H,\: h\mapsto h_1\otimes h_2$ and Hom-counit $\varepsilon:H\to k$. Then the dual $(H'=Hom(H,k),\bar{\alpha})$ is a monoidal Hom-algebra with the convolution product $(ff')(h)=f(h_1)f'(h_2)$ for functionals $f,f'\in H'$ and $h\in H$, as Hom-multiplication, and $\varepsilon$ as Hom-unit, where $\bar{\alpha}(f)=f\circ \alpha^{-1}$ for any $f\in H'$: For $f,g,k\in H'$ and $h\in H$,
\begin{eqnarray*}(\bar{\alpha}(f)(gk))(h)&=&\bar{\alpha}(f)(h_1)(gk)(h_2)=f(\alpha^{-1}(h_1))g(h_{21})k(h_{22})\\
&=&f(h_{11}))g(h_{12})k(\alpha^{-1}(h_{2}))=(fg)(h_1)\bar{\alpha}(k)(h_2)\\
&=&((fg)\bar{\alpha}(k))(h),
\end{eqnarray*}
which is the Hom-associativity, and
$$(\varepsilon f)(h)=\varepsilon(h_1) f(h_2)=f(\alpha^{-1}(h))=\bar{\alpha}(f)(h)=(f \varepsilon)(h),$$
which is the Hom-unity. Then we have the following

\begin{lemma}\label{Hom-actions-of-the-dual}
\begin{enumerate}
\item The linear map $H'\otimes H\to H,\: f\otimes h\mapsto f\bullet h:=\alpha^{2}(h_1)f(\alpha(h_2))$ defines a left Hom-action of $(H',\bar{\alpha})$ on $(H,\alpha)$.
\item  The linear map $H\otimes H'\to H,\: h\otimes f\mapsto h\bullet f:=f(\alpha(h_1))\alpha^{2}(h_2)$ defines a right Hom-action of $(H',\bar{\alpha})$ on $(H,\alpha)$.
\end{enumerate}
\end{lemma}
\begin{proof} We prove only the item (1). Let $f,f'\in H' $ and $h\in H$. Then,
\begin{eqnarray*}\bar{\alpha}(f)\bullet(f'\bullet h)&=&(f\circ\alpha^{-1})\bullet (\alpha^{2}(h_1)f'(\alpha(h_2)))\\
&=&\alpha^{2}(\alpha^{2}(h_1)_1)(f\circ\alpha^{-1})(\alpha(\alpha^{2}(h_1)_2))f'(\alpha(h_2))\\
&=&\alpha^{4}(h_{11})f(\alpha^{2}(h_{12}))f'(\alpha(h_2))=\alpha^{4}(\alpha^{-1}(h_1))f(\alpha^{2}(h_{21}))f'(\alpha(\alpha(h_{22})))\\
&=&\alpha^{3}(h_1)f(\alpha^{2}(h_{21}))f'(\alpha^{2}(h_{22}))=\alpha^{3}(h_1)f(\alpha^{2}(h_{2})_1)f'(\alpha^{2}(h_{2})_2)\\
&=&\alpha^{3}(h_1)(ff')(\alpha^{2}(h_{2}))=\alpha^{2}(\alpha(h)_1)(ff')(\alpha(\alpha(h)_2))\\
&=&(ff')\bullet \alpha(h),
\end{eqnarray*}
$$\varepsilon\bullet h=\alpha^{2}(h_1)\varepsilon(\alpha(h_2))=\alpha^{2}(h_1)\varepsilon(h_2)=\alpha^{2}(\alpha^{-1}(h))=\alpha(h),$$
which are the Hom-associativity and Hom-unity properties, respectively. We also have
$\bar{\alpha}(f)\bullet \alpha(h)=(f\circ \alpha^{-1})\bullet \alpha(h)=\alpha^{3}(h_1)f(\alpha(h_2))=\alpha(f\bullet h),$
which finishes the proof that $(H,\alpha)$ is a left $(H',\bar{\alpha})$-Hom-module with the given map.
\end{proof}

For the discussion below we assume $k$ as a field. Suppose that $I$ is an index set. The matrix $(v^{i}_j)_{i,j\in I}$ with entries $v^{i}_j\in H$ is said to be {\it pointwise finite} if for any $i \in I$, only a finite number of terms $v^{i}_j$ do not vanish. The matrix $(f^{i}_j)_{i,j\in I}$ of functionals $f^{i}_j\in H'$ is called pointwise finite if for arbitrary $i\in I$ and $h\in H$, all but finitely many terms $f^{i}_j(h)$ vanish. Let $(M,\mu)$ be a left-covariant $(H,\alpha)$-Hom-bimodule and $\{m_{i}\}_{i\in I}$ be a linear basis of $\:^{coH}M$. Then there exist uniquely determined coefficients $\mu^{i}_j,\:\bar{\mu}^{i}_j \in k$, which are the entries of pointwise finite matrices $(\mu^{i}_j)_{i,j\in I}$ and $(\bar{\mu}^{i}_j)_{i,j\in I}$, such that $\mu|_{^{coH}M}(m_i)=\mu^{i}_jm_j$, $(\mu|_{^{coH}M})^{-1}(m_i)=\bar{\mu}^{i}_jm_j$ (Einstein summation convention is used, i.e., there is a summation over repeating indices) satisfying $\mu^{i}_j\bar{\mu}^{j}_k=\delta_{ik}=\bar{\mu}^{i}_j\mu^j_k$. Thus, by using the above lemma, we express some of the results obtained about left-covariant Hom-bimodules in coordinate form as follows

\begin{theorem}\label{fund-thm-of-left-cov-in-coord-form}Let $(M,\mu)$ be a left-covariant $(H,\alpha)$-Hom-bimodule and $\{m_{i}\}_{i\in I}$ be a linear basis of $\:^{coH}M$. Then $\{m_{i}\}_{i\in I}$ is a free left $(H,\alpha)$-Hom-module basis of $M$ and there exists a pointwise finite matrix $(f^{i}_j)_{i,j\in I}$ of linear functionals $f^{i}_j\in H'$ satisfying, for any $h,g\in H$ and $i,j\in I$,
\begin{equation}\label{Hom-assoc-by-functionals}\mu^{i}_jf^j_k(hg)=f^{i}_j(h)f^j_k(\alpha(g)), \: f^{i}_j(1)=\mu^{i}_j,\end{equation}
\begin{equation}\label{left-cov-bimod-given-by-left-action-on-basis}m_{i}h=(\bar{\mu}^{i}_jf^j_k\bullet\alpha^{-1}(h))m_k.\end{equation}
Moreover, $\{m_{i}\}_{i\in I}$ is a free right $(H,\alpha)$-Hom-module basis of $M$ and we have
\begin{equation}\label{left-cov-bimod-given-by-right-action-on-basis}hm_{i}=m_j((\bar{\mu}^{i}_kf^k_j\circ S^{-1})\bullet\alpha^{-1}(h)).\end{equation}
\end{theorem}

\begin{proof} By the equation (\ref{structure-eqn-left-cov-Hom-bimod2}) and the fact that $P_L(m)\in \: ^{coH}M$ for any $m\in M$, we write any element $m\in M$ in the form $m=\sum_ih_im_i$, where $h_i\in H,\: i\in I$. Then, applying the left Hom-coaction to the both sides of $m=\sum_ih_im_i$, we get $\rho(m)=\sum_i\Delta(h_i)(1\otimes \mu^{-1}(m_i))$, and hence by the equations (\ref{structure-eqn-left-cov-Hom-bimod1}) and $P_L(m_i)=m_i,\: i\in I$, we have
\begin{eqnarray*}(id\otimes P_L)(\rho(m))&=&\sum_ih_{i,1}1\otimes P_L(h_{i,2}\mu^{-1}(m_i))=\sum_i\alpha(h_{i,1})\otimes \varepsilon(h_{i,2}) \mu(P_L(\mu^{-1}(m_i)))\\
&=&\sum_i\alpha(h_{i,1}\varepsilon(h_{i,2}))\otimes P_L(m_i)=\sum_ih_i\otimes m_i,
\end{eqnarray*}
where we put $\Delta(h_i)=h_{i,1}\otimes h_{i,2}$. By the linear independence of $\{m_{i}\}_{i\in I}$, we conclude that $h_i\in H$ are uniquely determined.

Since, for any $h\in H$, $m_i\lhd h=\widetilde{ad}_R(h)(m_i)\in\: ^{coH}M$, there exist $f^{i}_j(h) \in k,\:i,j\in I$ such that
\begin{equation}m_i\lhd h=f^{i}_j(h)m_j,\end{equation}
where only a finite number of $f^{i}_j(h)$ do not vanish. For any $h,g\in H$, we have
\begin{eqnarray*}\mu^{i}_jf^{j}_k(hg)m_k&=&\mu(m_i)\lhd(hg)=(m_i\lhd h)\lhd \alpha(g)\\
&=&(f^{i}_j(h)m_j)\lhd \alpha(g)=f^{i}_j(h)f^{j}_k(\alpha(g))m_k,
\end{eqnarray*}
which implies $\mu^{i}_jf^j_k(hg)=f^{i}_j(h)f^j_k(\alpha(g))$, and
$$f^{i}_j(1)m_j=m_{i}\lhd 1=\mu(m_i)=\mu^{i}_jm_j$$
concludes that $f^{i}_j(1)=\mu^{i}_j$. By using the identification of $hm_i$ with $h\otimes m_i$ and the right Hom-action of $H$ on $H\otimes\:^{coH}M$ we obtain

\begin{eqnarray*}m_ih&=&(1\otimes \mu^{-1}(m_i))h=1h_1\otimes \mu^{-1}(m_i)\lhd h_2\\
&=&\alpha(h_1)\otimes (\bar{\mu}^{i}_jm_j)\lhd h_2=\alpha(h_1)\otimes \bar{\mu}^{i}_jf^j_k(h_2)m_k\\
&=&\alpha(h_1)(\bar{\mu}^{i}_jf^j_k)(h_2)m_k=\alpha^{2}(\alpha^{-1}(h)_1)(\bar{\mu}^{i}_jf^j_k)(\alpha(\alpha^{-1}(h)_2))m_k\\
&=&(\bar{\mu}^{i}_jf^j_k\bullet \alpha^{-1}(h))m_k.
\end{eqnarray*}
The equation (\ref{left-cov-bimod-given-by-right-action}) yields

\begin{eqnarray*}hm_i&=&(\mu^{-1}(m_i)\lhd S^{-1}(h_2))\alpha(h_1)=((\bar{\mu}^{i}_jm_j)\lhd S^{-1}(h_2))\alpha(h_1)\\
&=&(\bar{\mu}^{i}_jf^j_k)(S^{-1}(h_2))m_k\alpha(h_1)=m_k\alpha(h_1)(\bar{\mu}^{i}_jf^j_k)(S^{-1}(h_2))\\
&=&m_k\alpha(h_1)(\bar{\mu}^{i}_jf^j_k\circ S^{-1})(h_2)=m_k((\bar{\mu}^{i}_jf^j_k\circ S^{-1})\bullet \alpha^{-1}(h)).
\end{eqnarray*}
Since, for any $p,s\in I$, $f^p_s\circ\alpha=\bar{\mu}^p_qf^q_r\mu^r_s$ and $(\bar{\mu}^l_sf^p_l)(hg)=(\bar{\mu}^p_rf^r_l)(h)(\bar{\mu}^l_qf^q_s)(g)$ for $h,g\in H$, we have

\begin{eqnarray*}((\bar{\mu}^j_lf^l_i)(\bar{\mu}^k_pf^p_j\circ S^{-1}))(S(h))&=&(\bar{\mu}^j_lf^l_i)(S(h)_1)(\bar{\mu}^k_pf^p_j\circ S^{-1})(S(h)_2)\\
&=&(\bar{\mu}^j_lf^l_i)(S(h_2))(\bar{\mu}^k_pf^p_j)(h_1)=(\bar{\mu}^k_p\bar{\mu}^j_lf^p_j)(h_1)f^l_i(S(h_2))\\
&=&(\bar{\mu}^k_p\bar{\mu}^p_rf^r_l\circ\alpha^{-1})(h_1)(\mu^q_i\bar{\mu}^l_sf^s_q\circ\alpha^{-1})(S(h_2))\\
&=&\bar{\mu}^k_p\mu^q_i(\bar{\mu}^p_rf^r_l)(\alpha^{-1}(h)_1)(\bar{\mu}^l_sf^s_q)(S(\alpha^{-1}(h)_2))\\
&=&\bar{\mu}^k_p\mu^q_i(\bar{\mu}^l_qf^p_l)(\alpha^{-1}(h)_1S(\alpha^{-1}(h)_2))\\
&=&\bar{\mu}^k_p\mu^q_i\bar{\mu}^l_qf^p_l(1)\varepsilon(\alpha^{-1}(h))=\bar{\mu}^k_p\mu^q_i\bar{\mu}^l_q\mu^p_l(1)\varepsilon(h)\\
&=&\delta_{lk}\delta_{li}\varepsilon(S(h))=\delta_{ki}\varepsilon(S(h)),
\end{eqnarray*}
that is, we have shown that
\begin{equation}(\bar{\mu}^j_lf^l_i)(\bar{\mu}^k_pf^p_j\circ S^{-1})=\delta_{ik}\varepsilon.\end{equation}
In a similar way, one can also prove that
\begin{equation}\label{convolution-for-coordinate-forms}(\bar{\mu}^j_lf^l_k\circ S^{-1})(\bar{\mu}^{i}_pf^p_j)=\delta_{ki}\varepsilon.\end{equation}
Since $\{m_{i}\}_{i\in I}$ is a free left $(H,\alpha)$-Hom-module basis of $M$ and the equation (\ref{left-cov-bimod-given-by-right-action-on-basis}) holds, any element $m\in M$ is also of the form $m=\sum_im_ih_i$ for some $h_i\in H$. Let us assume that $\sum_im_ih_i=0$ (all but finitely many $h_i$ vanishes, $i\in I$). So, by the equation (\ref{left-cov-bimod-given-by-left-action-on-basis}), we get $\sum_i(\bar{\mu}^{i}_jf^j_k\bullet\alpha^{-1}(h_i))m_k=0$ which implies that

$$\sum_i(\bar{\mu}^{i}_jf^j_k\bullet\alpha^{-1}(h_i))=0, \: \forall k\in I.$$

If we apply $\bar{\alpha}(\bar{\mu}^k_lf^l_p\circ S^{-1})$ from left to the both sides and use the equation (\ref{convolution-for-coordinate-forms}), we obtain
\begin{eqnarray*}0&=&\sum_i\bar{\alpha}(\bar{\mu}^k_lf^l_p\circ S^{-1})\bullet(\bar{\mu}^{i}_jf^j_k\bullet\alpha^{-1}(h_i))\\
&=&\sum_i((\bar{\mu}^k_lf^l_p\circ S^{-1})(\bar{\mu}^{i}_jf^j_k))\bullet\alpha(\alpha^{-1}(h_i))\\
&=&\sum_i\delta_{pi}\varepsilon\bullet h_i=\sum_i\delta_{pi}\alpha(h_i)=\alpha(b_p),
\end{eqnarray*}
for all $p\in I$, that is $b_p=0, \forall p\in I$. This finishes the proof that $\{m_{i}\}_{i\in I}$ is a free right $(H,\alpha)$-Hom-module basis of $M$.
\end{proof}

\section{Right-Covariant Hom-Bimodules}

\begin{definition} A {\it right-covariant} $(H,\alpha)$-{\it Hom-bimodule} is an $(H,\alpha)$-Hom-bimodule $(M,\mu) \in \widetilde{\mathcal{H}}(\mathcal{M}_k)$ which is a right $(H,\alpha)$-Hom-comodule, with Hom-coaction $\sigma: M\to M\otimes H$, $m\mapsto m_{[0]}\otimes m_{[1]}$, in $\widetilde{\mathcal{H}}(\mathcal{M}_k)$ such that
\begin{equation}\sigma((hm)\alpha(g))=\Delta(\alpha(h))(\sigma(m)\Delta(g)).\end{equation}
The set $M^{coH}=\{m\in M|\rho(m)=\mu^{-1}(m)\otimes 1_H\}$ of $M$ is called {\it right coinvariant} of $(H,\alpha)$ on $(M,\mu)$.
\end{definition}
  Without performing details, we can develop a similar theory for the right-covariant $(H,\alpha)$-Hom-bimodules as in the previous section by making the necessary changes. We define the projection by
\begin{equation}P_R:M\to \: M^{coH},\: m\mapsto m_{[0]}S(m_{[1]}), \end{equation}
which is unique with the property
\begin{equation}P_R(mh)=\varepsilon(h)\mu(P_R(m)), \: for\:all\: h\in H, m\in M.\end{equation}
Since the relation
\begin{equation}(id\otimes\Delta)\circ(\sigma\otimes id)\circ\sigma=\tilde{a}_{M\otimes H,H,H}\circ (\tilde{a}^{-1}_{M,H,H}\otimes id)\circ((id\otimes\Delta)\otimes id)\circ(\sigma\otimes id)\circ \sigma\end{equation}
holds, that is, for any $m\in M$, the following equality
\begin{equation}m_{[0][0]}\otimes m_{[0][1]}\otimes m_{[1]1}\otimes m_{[1]2}=m_{[0][0]}\otimes\alpha(m_{[0][1]1})\otimes \alpha(m_{[0][1]2})\otimes\alpha^{-1}(m_{[1]})\end{equation}
is fulfilled, one can prove that $$\sigma(P_R(m))=\mu^{-1}(P_R(m))\otimes 1_H.$$
One can also show that
\begin{equation}m=P_R(m_{[0]})m_{[1]}\end{equation}
is acquired by using the Hom-coassociativity property for the right Hom-comodules, which specifies that $M=M^{coH}\:\cdot H$.
$P_R$ also satisfies
\begin{equation} P_R(hm)=\alpha(h_1)(\mu^{-1}(P_R(m))S(h_2))\equiv\widetilde{ad}_L(h)P_R(m).\end{equation}
Since $(M,\mu)$ is an $(H,\alpha)$-Hom-bimodule, $M^{coH}$ has a left $(H,\alpha)$-Hom-module structure by the formula

\begin{equation}\label{Hom-mod-structure-on-right-coinvariants}h\rhd m:=P_R(hm)=\widetilde{ad}_L(h)m.  \end{equation}

$\widetilde{ad}_L$ is in fact a left Hom-action of $(H,\alpha)$ on $M^{coH}$:
$$\widetilde{ad}_L(1_H)m=1_H\rhd m=\alpha(1_H)(\mu^{-1}(m)S(1_H))=1_Hm=\mu(m),$$
\begin{eqnarray*}(gh)\rhd \mu(m)&=&\alpha(g_1h_1)(\mu^{-1}(\mu(m))S(g_2h_2))\\
&=&(\alpha(g_1)\alpha(h_1))(m (S(h_2)S(g_2)))\\
&=&(\alpha(g_1)\alpha(h_1))((\mu^{-1}(m)S(h_2))S(\alpha(g_2)))\\
&=&((g_1h_1)(\mu^{-1}(m)S(h_2)))\alpha(S(\alpha(g_2)))\\
&=&(\alpha(g_1)(h_1\mu^{-1}(\mu^{-1}(m)S(h_2))))\alpha(S(\alpha(g_2)))\\
&=&(\alpha(g_1)\mu^{-1}((\alpha(h_1)\mu^{-1}(\mu^{-1}(m)S(h_2)))))\alpha(S(\alpha(g_2)))\\
&=&\alpha(\alpha(g_1))(\mu^{-1}(h\rhd m)S(\alpha(g_2)))\\
&=&\alpha(g)\rhd(h\rhd m),
\end{eqnarray*}
for all $m\in M^{coH}$ and $g,h \in H$.
Once this left Hom-module structure has been given to $M^{coH}$, it can be proven, in a similar way as in the proofs of Propositions (\ref{left-cov-structure}) and (\ref{fundamental-thm-left-cov-bimod}), that the right-covariant $(H,\alpha)$-Hom-bimodule  $(M,\mu)$ is isomorphic, by the morphism in $\widetilde{\mathcal{H}}(\mathcal{M}_k)$
\begin{equation}\theta': M^{coH}\otimes H \to M, \: m\otimes h\mapsto mh ,\end{equation}
to the right-covariant $(H,\alpha)$-Hom-bimodule $M^{coH}\otimes H$ with Hom-(co)module structures defined by the following maps in $\widetilde{\mathcal{H}}(\mathcal{M}_k)$
\begin{equation}\label{right-Hom-action-on-right-bicov}(M^{coH}\otimes H)\otimes H \to M^{coH}\otimes H, \: (m\otimes h)\otimes g \mapsto \mu(m)\otimes h\alpha^{-1}(g) ,\end{equation}
\begin{equation}\label{left-Hom-action-on-right-bicov}H \otimes(M^{coH}\otimes H) \to M^{coH}\otimes H, \: g\otimes(m\otimes h)\mapsto g_1\rhd m\otimes g_2h ,\end{equation}
\begin{equation}\label{right-Hom-coaction-on-right-bicov} M^{coH}\otimes H \to (M^{coH}\otimes H)\otimes H, \: m\otimes h\mapsto (\mu^{-1}(m)\otimes h_1)\otimes \alpha(h_2) .\end{equation}
Thus we have the following

\begin{theorem}\label{one-to-one-RightCov-LeftMod}There is a one-to-one correspondence, given by (\ref{Hom-mod-structure-on-right-coinvariants}) and (\ref{right-Hom-action-on-right-bicov})-(\ref{right-Hom-coaction-on-right-bicov}), between the right-covariant $(H,\alpha)$-Hom-bimodules $(M,\mu)$ and the left $(H,\alpha)$-Hom-module structures on $(M^{coH},\mu|_{M^{coH}})$.
\end{theorem}

We denote by $_{H}\widetilde{\mathcal{H}}(\mathcal{M}_k)_H^H$ the category of right-covariant $(H,\alpha)$-Hom-bimodules whose objects are the right-covariant
$(H,\alpha)$-Hom-bimodules with those morphisms that are left and right $(H,\alpha)$-linear and right $(H,\alpha)$-colinear.

\begin{proposition}\label{tensor_prod_right-cov_Hom_bimod} Let $(H,\alpha)$ be a monoidal Hom-Hopf algebra and $(M,\mu)$, $(N,\nu)$ be two
right-covariant $(H,\alpha)$-Hom-bimodules. Along with (\ref{left_action_left-cov_bimod}) and (\ref{right_action_left-cov_bimod}), define the morphism in $\widetilde{\mathcal{H}}(\mathcal{M}_k)$
 \begin{equation}\label{right-coaction-right-cov-bimod}\sigma: M\otimes_H N\to (M\otimes_H N)\otimes H,\: m\otimes_H n\mapsto(m_{[0]}\otimes_H n_{[0]})\otimes m_{[1]}n_{[1]}, \end{equation}
 which is the right codiagonal Hom-coaction of $(H,\alpha)$ on $M\otimes_H N$. Then $(M\otimes_H N, \mu\otimes_H \nu)$ is a right-covariant $(H,\alpha)$-Hom-bimodule .
 \end{proposition}

\begin{proof} It is sufficient to prove first that $M\otimes_HN$ becomes a right $(H,\alpha)$-Hom-comodule with $\sigma$ and then to assert that the right covariance is held.

\begin{eqnarray*}\lefteqn{(\mu^{-1}\otimes_H\nu^{-1})((m\otimes_H n)_{[0]})\otimes \Delta((m\otimes_H n)_{[1]})}\hspace{3em}\\
&=&(\mu^{-1}(m_{[0]})\otimes_H\nu^{-1}(n_{[0]}))\otimes \Delta(m_{[1]})\Delta(n_{[1]})\\
&=&(m_{[0][0]}\otimes_H n_{[0][0]})\otimes (m_{[0][1]}n_{[0][1]}\otimes\alpha^{-1}(m_{[1]})\alpha^{-1}(n_{[1]}))\\
&=&(m\otimes_H n)_{[0][0]}\otimes ((m\otimes_H n)_{[0][1]}\otimes (m\otimes_H n)_{[1]}),
\end{eqnarray*}
where in the second equality the Hom-coassociativity condition for right $(H,\alpha)$-Hom-comodules has been used, and we also have
$$(m\otimes_H n)_{[0]}\varepsilon((m\otimes_H n)_{[1]})=m_{[0]}\varepsilon(m_{[1]})\otimes_H n_{[0]}\varepsilon(n_{[1]})=\mu^{-1}(m)\otimes_H\nu^{-1}(n),$$

that is, $\sigma$ satisfies the Hom-coassociativity and Hom-counity, respectively.

And with the next calculation we end the proof:

\begin{eqnarray*}\sigma((g(m\otimes_Hn))\alpha(h))&=&\sigma(g\mu(m)\otimes_H\nu(n)h)\\
&=&((g\mu(m))_{[0]}\otimes_H((\nu(n)h)_{[0]})\otimes(g\mu(m))_{[1]}(\nu(n)h)_{[1]} \\
&=&(g_1\mu(m_{[0]})\otimes_H\nu(n_{[0]})h_1)\otimes(g_2\alpha(m_{[1]}))(\alpha(n_{[1]})h_2)\\
&=&\alpha(g_1)(\mu(m_{[0]})\otimes_H n_{[0]}\alpha^{-1}(h_1))\otimes \alpha(g_2)((m_{[1]}n_{[1]})h_2)\\
&=&\Delta(\alpha(g))((m_{[0]}\otimes_H n_{[0]})h_1\otimes(m_{[1]}n_{[1]})h_2)\\
&=&\Delta(\alpha(g))(\sigma(m\otimes_Hn)\Delta(h)).
\end{eqnarray*}

\end{proof}

\begin{theorem}Let $(H,\alpha)$ be a monoidal Hom-Hopf algebra. Then $_{H}\widetilde{\mathcal{H}}(\mathcal{M}_k)_H^H$ is a tensor category, with tensor product
$\otimes_H$ is defined in Proposition (\ref{tensor_prod_right-cov_Hom_bimod}), and associativity constraint $\tilde{a}$, left unit constraint $\tilde{l}$ and right unit constraint $\tilde{r}$ are given by (\ref{associator-formula}), (\ref{left-unit-formula}) and (\ref{right-unit-formula}), respectively.
\end{theorem}

\begin{proof} What is left to be proven is that the associator $\tilde{a}_{M,N,P}$, left unitor $\tilde{l}_M$ and right unitor $\tilde{r}_M$ are all right $(H,\alpha)$-colinear.

\begin{eqnarray*}(\sigma^{Q'}\circ\tilde{a}_{M,N,P})((m\otimes_Hn)\otimes_Hp)&=&\sigma^{Q'}(\mu(m)\otimes_H(n\otimes_H\pi^{-1}(p)))\\
&=&(\mu(m)_{[0]}\otimes_H(n\otimes_H\pi^{-1}(p))_{[0]})\otimes \mu(m)_{[1]}(n\otimes_H\pi^{-1}(p))_{[0]}\\
&=&(\mu(m_{[0]})\otimes_H(n_{[0]}\otimes_H\pi^{-1}(p_{[0]})))\otimes \alpha(m_{[1]})(n_{[1]}\alpha^{-1}(p_{[1]}))\\
&=&\tilde{a}_{M,N,P}((m_{[0]}\otimes_Hn_{[0]})\otimes_Hp_{[0]})\otimes (m_{[1]}n_{[1]})p_{[1]}\\
&=&(\tilde{a}_{M,N,P}\otimes id)(((m\otimes_Hn)_{[0]}\otimes_Hp_{[0]})\otimes (m\otimes_Hn)_{[1]}p_{[1]})\\
&=&((\tilde{a}_{M,N,P}\otimes id)\circ\sigma^{Q})((m\otimes_Hn)\otimes_Hp)
\end{eqnarray*}
which stands for the right $(H,\alpha)$-colinearity of $\tilde{a}_{M,N,P}$, where $\sigma^{Q'}$ and $\sigma^{Q}$ are the right Hom-coactions of $(H,\alpha)$ on $Q'=M\otimes_H(N\otimes_HP)$ and $Q=(M\otimes_HN)\otimes_HP$.

By considering the fact that $(H,\alpha)$ is a right-covariant $(H,\alpha)$-Hom-bimodule with Hom-actions given by its Hom-multiplication and Hom-coaction by its Hom-comultiplication, we do the computation

\begin{eqnarray*}(\tilde{l}_{M}\otimes id_H)(\sigma^{H\otimes_HM}(h\otimes_H m))&=&(\tilde{l}_{M}\otimes id_H)((h_{[0]}\otimes_Hm_{[0]})\otimes h_{[1]}m_{[1]})\\
&=&(\tilde{l}_{M}\otimes id_H)((h_1\otimes_Hm_{[0]})\otimes h_2m_{[1]})\\
&=&h\sigma^{M}(m)=\sigma^{M}(hm)\\
&=&\sigma^{M}(\tilde{l}_{M}(h\otimes_Hm)),
\end{eqnarray*}
concluding $\tilde{l}_M$ is right $(H,\alpha)$-colinear. By a similar argument, $\tilde{r}_M$ as well is right $(H,\alpha)$-colinear.
\end{proof}

\section{Bicovariant Hom-Bimodules}
\begin{definition} A {\it bicovariant} $(H,\alpha)$-{\it Hom-bimodule} is an $(H,\alpha)$-Hom-bimodule $(M,\mu)$ together with $k$-linear mappings
$$\rho: M \to H \otimes M,\: m\mapsto m_{(-1)}\otimes m_{(0)},$$
$$\sigma: M \to M \otimes H,\: m\mapsto m_{[0]}\otimes m_{[1]},$$
in $\widetilde{\mathcal{H}}(\mathcal{M}_k)$, such that
\begin{enumerate}
\item $(M,\mu)$ is a left-covariant $(H,\alpha)$-Hom-bimodule with left $(H,\alpha)$-Hom-coaction $\rho$,
\item $(M,\mu)$ is a right-covariant $(H,\alpha)$-Hom-bimodule with right $(H,\alpha)$-Hom-coaction $\sigma$,
\item the following relation holds:
\begin{equation}\label{Hom-comm} \tilde{a}_{H,M,H}\circ(\rho\otimes id)\circ\sigma=(id\otimes\sigma)\circ\rho.\end{equation}
\end{enumerate}
\end{definition}
The condition (\ref{Hom-comm}) is called the Hom-commutativity of the Hom-coactions $\rho$ and $\sigma$ on $M$ and can be expressed
by Sweedler's notation as follows
$$m_{(-1)}\otimes(m_{(0)[0]}\otimes m_{(0)[1]})=\alpha(m_{[0](-1)})\otimes(m_{[0](0)}\otimes \alpha^{-1}(m_{[1]})), \: m\in M. $$

\begin{proposition}\label{bicovariant-bimod-structure}Let $(N,\nu)\in\widetilde{\mathcal{H}}(\mathcal{M}_k)$ be a right $(H,\alpha)$-Hom-module by the map $N\otimes H\to N,\: n\otimes h\mapsto n\lhd h$ and a right $(H,\alpha)$-Hom-comodule by $N\to N\otimes H,\: n\mapsto n_{(0)}\otimes n_{(1)}$ such that the compatibility condition, which is called Hom-Yetter-Drinfeld condition,
\begin{equation}\label{Yetter-Drinfeld}n_{(0)}\lhd\alpha^{-1}(h_1)\otimes n_{(1)}\alpha^{-1}(h_2)=(n\lhd h_2)_{(0)}\otimes\alpha^{-1}(h_1(n\lhd h_2)_{(1)})\end{equation}
holds for $h\in H$ and $n\in N$.
 The morphisms (\ref{left-Hom-action})-(\ref{left-Hom-coaction}) and
\begin{equation}\label{right-Hom-coaction}\sigma:H\otimes N\to (H\otimes N)\otimes H,\: h\otimes n\mapsto (h_1\otimes n_{(0)})\otimes h_2n_{(1)}, \end{equation}
in $\widetilde{\mathcal{H}}(\mathcal{M}_k)$, define a bicovariant $(H,\alpha)$-Hom-bimodule structure on $(H\otimes N , \alpha\otimes\nu)$.
\end{proposition}

\begin{proof} As has been proven in Proposition \ref{left-cov-structure}, the left-covariant $(H,\alpha)$-Hom-bimodule structure on $(H\otimes N , \alpha\otimes\nu)$ is deduced from the right $(H,\alpha)$-Hom-action on $(N,\nu)$ by the morphisms (\ref{left-Hom-action})-(\ref{left-Hom-coaction}). The morphism (\ref{right-Hom-coaction}) fulfills the Hom-coassociativity and Hom-counity:
\begin{eqnarray*}(\alpha^{-1}\otimes\nu^{-1})((h\otimes n)_{[0]})\otimes \Delta((h\otimes n)_{[1]})&=&(\alpha^{-1}(h_1)\otimes\nu^{-1}(n_{(0)}))\otimes \Delta(h_2n_{(1)})\\
&=&(\alpha^{-1}(h_1)\otimes\nu^{-1}(n_{(0)}))\otimes (h_{21}n_{(1)1}\otimes h_{22}n_{(1)2})\\
&=&(h_{11}\otimes n_{(0)(0)})\otimes (h_{12}n_{(0)(1)}\otimes \alpha^{-1}(h_2)\alpha^{-1}(n_{(1)}))\\
&=&(h\otimes n)_{[0][0]}\otimes((h\otimes n)_{[0][1]}\otimes\alpha^{-1}((h\otimes n)_{[1]})),
\end{eqnarray*}
where the fact that $(N,\nu)$ is a right $(H,\alpha)$-Hom-comodule and the relation (\ref{relation-left-coass-tens-pro}) have been used in the third equation, and we besides obtain
 $$(h\otimes n)_{[0]}\varepsilon((h\otimes n)_{[1]})=(h_1\otimes n_{(0)})\varepsilon(h_2n_{(1)})=h_1\varepsilon(h_2)\otimes n_{(0)}\varepsilon(n_{(1)})=(\alpha^{-1}\otimes\nu^{-1})(h\otimes n).$$

By again using the relation (\ref{relation-left-coass-tens-pro}) and the fact that the right $(H,\alpha)$-Hom-coaction on $(N,\nu)$ is a morphism in $\widetilde{\mathcal{H}}(\mathcal{M}_k)$, we prove the Hom-commutativity condition:
\begin{eqnarray*}\alpha(m_{[0](-1)})\otimes(m_{[0](0)}\otimes \alpha^{-1}(m_{[1]}))&=&\alpha^{2}(h_{11})\otimes((h_{12}\otimes\nu^{-1}(n_{(0)}))\otimes\alpha^{-1}(h_2)\alpha^{-1}(n_{(1)}))\\
&=&\alpha(h_1)\otimes((h_{21}\otimes\nu^{-1}(n_{(0)}))\otimes h_{22}\alpha^{-1}(n_{(1)}))\\
&=&\alpha(h_1)\otimes((h_{21}\otimes\nu^{-1}(n)_{(0)})\otimes h_{22}\nu^{-1}(n)_{(1)})\\
&=&m_{(-1)}\otimes(m_{(0)[0]}\otimes m_{(0)[1]}).
\end{eqnarray*}
For $g,h,k \in H$ and $n \in N$, we have
\begin{eqnarray}\label{right-covariance-1}\nonumber\lefteqn{\sigma((g(h\otimes n))\alpha(k))}\hspace{2em}\\
\nonumber&=&\sigma((\alpha^{-1}(g)h\otimes \nu(n))\alpha(k))\\
\nonumber&=&\sigma((\alpha^{-1}(g)h)\alpha(k_1)\otimes \nu(n)\lhd\alpha(k_2))\\
\nonumber&=&(((\alpha^{-1}(g)h)\alpha(k_1))_1\otimes (\nu(n)\lhd\alpha(k_2))_{(0)})\otimes((\alpha^{-1}(g)h)\alpha(k_1))_2(\nu(n)\lhd\alpha(k_2))_{(1)}\\
\nonumber&=&(g_1(h_1k_{11})\otimes \nu((n\lhd k_2)_{(0)}))\otimes (g_2(h_2k_{12}))\alpha((n\lhd k_2)_{(1)})\\
\nonumber&=&(\alpha^{-1}(\alpha(g_1))(h_1k_{11})\otimes \nu((n\lhd k_2)_{(0)}))\otimes \alpha(g_2)((h_2k_{12})(n\lhd k_2)_{(1)})\\
\nonumber&=&(\alpha^{-1}(\alpha(g_1))(h_1\alpha^{-1}(k_1))\otimes \nu((n\lhd \alpha(k_{22}))_{(0)}))\otimes \alpha(g_2)(\alpha(h_2)(k_{21}\alpha^{-1}((n\lhd \alpha(k_{22})_{(1)}))\\
&=&\alpha(g_1)(h_1\alpha^{-1}(k_1)\otimes(n\lhd \alpha(k_{22}))_{(0)})\otimes \alpha(g_2)(\alpha(h_2)(k_{21}\alpha^{-1}((n\lhd \alpha(k_{22})_{(1)})),
\end{eqnarray}
and
\begin{eqnarray}\label{right-covariance-2}\nonumber\Delta(\alpha(g))(\sigma(h\otimes n)\Delta(k))&=&(\alpha(g_1)\otimes \alpha(g_2))(((h_1\otimes n_{(0)})\otimes h_2n_{(1)})(k_1\otimes k_2))\\
\nonumber&=&(\alpha(g_1)\otimes \alpha(g_2))((h_1\otimes n_{(0)})k_1\otimes (h_2n_{(1)})k_2)\\
\nonumber&=&(\alpha(g_1)\otimes \alpha(g_2))((h_1k_{11}\otimes n_{(0)}\lhd k_{12})\otimes \alpha(h_2)(n_{(1)}\alpha^{-1}(k_2))\\
&=&\alpha(g_1)(h_1\alpha^{-1}(k_1)\otimes n_{(0)}\lhd k_{21})\otimes \alpha(g_2)(\alpha(h_2)(n_{(1)}k_{22})).
\end{eqnarray}
The right-hand sides of (\ref{right-covariance-1}) and (\ref{right-covariance-2}) are equal by the compatibility condition (\ref{Yetter-Drinfeld}): To see this, it is enough to set $h=\alpha(k_2)$ in (\ref{Yetter-Drinfeld}) to obtain the following
$$n_{(0)}\lhd \alpha^{-1}(\alpha(k_2)_1)\otimes n_{(1)}\alpha^{-1}(\alpha(k_2)_2)=(n\lhd \alpha(k_2)_2)_{(0)}\otimes \alpha^{-1}(\alpha(k_2)_1(n\lhd \alpha(k_2)_2)_{(1)})$$
$$\Rightarrow \: n_{(0)}\lhd \alpha^{-1}(\alpha(k_{21}))\otimes n_{(1)}\alpha^{-1}(\alpha(k_{22}))=(n\lhd \alpha(k_{22}))_{(0)}\otimes \alpha^{-1}(\alpha(k_{21}))\alpha^{-1}((n\lhd \alpha(k_{22}))_{(1)})$$
$$\Rightarrow \: n_{(0)}\lhd k_{21}\otimes n_{(1)}k_{22}=(n\lhd \alpha(k_{22}))_{(0)}\otimes k_{21}\alpha^{-1}((n\lhd \alpha(k_{22}))_{(1)}).$$
Thus we proved that $(H\otimes N , \alpha\otimes\nu)$ is a bicovariant Hom-bimodule over $(H,\alpha)$.
\end{proof}

\begin{proposition}\label{fundamental-thm-bicov-bimod} If $(M,\mu) \in\widetilde{\mathcal{H}}(\mathcal{M}_k) $ is a bicovariant  $(H,\alpha)$-Hom-bimodule, the $k$-linear map (\ref{left-cov-iso})
in $\widetilde{\mathcal{H}}(\mathcal{M}_k)$ is an isomorphism of bicovariant $(H,\alpha)$-Hom-bimodules, where the right $(H,\alpha)$-Hom-module structure on
$(^{coH}M,\mu|_{^{coH}M})$ is defined by $m\lhd h := P_L(mh)=\widetilde{ad}_R(h)(m),$ for $h\in H$ and $m$ $\in$ $^{coH}M$ and the right $(H,\alpha)$-Hom-comodule structure is obtained by the restriction of right $(H,\alpha)$-Hom-coaction on $(M,\mu)$ fulfilling the condition (\ref{Yetter-Drinfeld}).
\end{proposition}

\begin{proof} Let $(M,\mu)$ be a bicovariant  $(H,\alpha)$-Hom-bimodule with left $(H,\alpha)$-Hom-coaction $\rho:M\to h\otimes M,\:m\mapsto m_{(-1)}\otimes m_{(0)}$ and right $(H,\alpha)$-Hom-coaction $\sigma:M\to M\otimes H,\:m\mapsto m_{[0]}\otimes m_{[1]}$. By Hom-commutativity condition (\ref{Hom-comm}) we get,
$$\sigma(^{coH}M)\subseteq \: ^{coH}M\otimes H,$$
which implies that the restriction of $\varphi$ to $^{coH}M$ can be taken as the right Hom-coaction of $(H,\alpha)$ on $(^{coH}M,\mu|_{^{coH}M})$ : In fact, for $m\in \:^{coH}M$,
\begin{eqnarray*}(id\otimes \sigma)(\rho(m))&=&1\otimes(\mu^{-1}(m_{[0]})\otimes\alpha^{-1}(m_{[1]}))\\
&=&\tilde{a}_{H,M,H}((\rho\otimes id)(\sigma(m)))\\
&=&\alpha(m_{[0](-1)})\otimes(m_{[0](0)}\otimes\alpha^{-1}(m_{[1]})),
\end{eqnarray*}
which purports that $\rho(m_{[0]})=1\otimes \mu^{-1}(m_{[0]})$.

Since it has been proven in Proposition (\ref{fundamental-thm-left-cov-bimod}) that the morphism $\theta:H\otimes \:^{coH}M\to M,\; h\otimes m\mapsto hm$ in (\ref{left-cov-iso}) is an isomorphism of left-covariant $(H,\alpha)$-Hom-bimodules, we next show that it is right $(H,\alpha)$-colinear to conclude that it is an isomorphism of bicovariant $(H,\alpha)$-Hom-bimodules:

$$\sigma(\theta(h\otimes m))=h_1m_{(0)}\otimes h_2m_{(1)}=(\theta\otimes id)((h_1\otimes m_{[0]})\otimes h_2m_{[1]})=(\theta\otimes id)(\sigma^{H\otimes \:^{coH}M}(h\otimes m)),$$

where $\sigma^{H\otimes \:^{coH}M}:H\otimes \:^{coH}M\to (H\otimes \:^{coH}M)\otimes H,\: h\otimes m\mapsto (h_1\otimes m_{[0]})\otimes h_2m_{[1]}$, for $h\in H$ and $m\in \:^{coH}M$, by the equation (\ref{right-Hom-coaction}).

Due to the fact that $(M,\mu)$ is a bicovariant  $(H,\alpha)$-Hom-bimodule, the left-hand sides of (\ref{right-covariance-1}) and (\ref{right-covariance-2}) are equal: Thus, by applying $(\varepsilon\otimes id_{N\otimes H})\circ \tilde{a}_{H,N,H}$ to the right-hand sides of (\ref{right-covariance-1}) and (\ref{right-covariance-2}), we acquire the compatibility condition (\ref{Yetter-Drinfeld}).
 \end{proof}
Hence, by Propositions (\ref{bicovariant-bimod-structure}) and (\ref{fundamental-thm-bicov-bimod}), we acquire
\begin{theorem}\label{one-to-one-Bicov-YD}  There is a one-to-one correspondence, given by (\ref{left-Hom-action})-(\ref{left-Hom-coaction}), (\ref{right-Hom-coaction}) and (\ref{right_action_on_coinvariants}), between bicovariant $(H,\alpha)$-Hom-bimodules $(M,\mu)$ and pairs of a right $(H,\alpha)$-Hom-module and a right $(H,\alpha)$-Hom-comodule structures on $(^{coH}M,\mu|_{^{coH}M})$ fulfilling the compatibility condition (\ref{Yetter-Drinfeld}).
\end{theorem}
We indicate by $^{H}_{H}\widetilde{\mathcal{H}}(\mathcal{M}_k)_H^H$ the category of bicovariant $(H,\alpha)$-Hom-bimodules; the objects are the bicovariant Hom-bimodules with those morphisms that are $(H,\alpha)$-linear and $(H,\alpha)$-colinear on both sides.

\begin{proposition}\label{tensor_prod_bicov_Hom_bimod} Let $(H,\alpha)$ be a monoidal Hom-Hopf algebra and $(M,\mu)$, $(N,\nu)$ be two bicovariant $(H,\alpha)$-Hom-bimodules.
 Then, with the Hom-module and Hom-comodule structures given by (\ref{left_action_left-cov_bimod}), (\ref{right_action_left-cov_bimod}), (\ref{left-coaction-left-cov-bimod})
 and (\ref{right-coaction-right-cov-bimod}),  $(M\otimes_H N, \mu\otimes_H \nu)$ becomes a bicovariant Hom-bimodule over $(H,\alpha)$.
 \end{proposition}
\begin{proof} The only condition left to be proven to finish the proof of the statement is the Hom-commutativity of $\rho$ and $\sigma$ :
 \begin{eqnarray*}\lefteqn{(\tilde{a}_{H,M\otimes_HN,H}\circ(\rho\otimes id))(\sigma(m\otimes_Hn))}\hspace{5em}\\
 &=&\tilde{a}((\rho\otimes id)((m_{[0]}\otimes_H n_{[0]})\otimes m_{[1]}n_{[1]}))\\
 &=&\tilde{a}((m_{[0](-1)}n_{[0](-1)}\otimes (m_{[0](0)}\otimes_Hn_{[0](0)}))\otimes m_{[1]}n_{[1]})\\
 &=&\alpha(m_{[0](-1)})\alpha(n_{[0](-1)})\otimes((m_{[0](0)}\otimes_Hn_{[0](0)})\otimes\alpha^{-1}(m_{[1]})\alpha^{-1}(m_{[1]}))\\
 &=&m_{(-1)}n_{(-1)}\otimes((m_{(0)[0]}\otimes_Hn_{(0)[0]})\otimes m_{(0)[1]}n_{(0)[1]})\\
 &=&(id \otimes \sigma)(m_{(-1)}n_{(-1)}\otimes(m_{(0)}\otimes_H n_{(0)}))\\
 &=&((id\otimes \sigma)\circ \rho)(m\otimes_Hn),
 \end{eqnarray*}
 where the fourth equality follows from the Hom-commutativity of the Hom-coactions of $(H,\alpha)$ on $(M,\mu)$ and $(N,\nu)$.
\end{proof}

\begin{lemma}Let $(H,\alpha)$ be a monoidal Hom-Hopf algebra.Then the $k$-linear map $c_{M,N}:M\otimes_HN\to N\otimes_HM$ given by, for $m\in M$ and $n\in N$,
\begin{eqnarray}\label{braiding_bicov_Hom_bimod}c_{M,N}(m\otimes_Hn)&=&m_{(-1)}P_{R}(n_{[0]})\otimes_HP_{L}(m_{(0)})n_{[1]}\\
&=&m_{(-1)}(n_{[0][0]}S(n_{[0][1]}))\otimes_H(S(m_{(0)(-1)})m_{(0)(0)})n_{[1]}
\end{eqnarray}
is a morphism in $^{H}_{H}\widetilde{\mathcal{H}}(\mathcal{M}_k)_H^H$.
\end{lemma}

\begin{proof} Let $(M,\mu)$ and $(N,\nu)$ be bicovariant $(H,\alpha)$-Hom-bimodules. Since $M\otimes_H N$ is linearly spanned by each of the sets $\{hu\otimes_Hv\}$, $\{w\otimes_H zh\}$, $\{hu\otimes_H z\}$, where $h\in H$, $u\in\:\: ^{coH}M$, $v\in\:\:^{coH}N$, $w\in M^{coH}$ and $z\in N^{coH}$, we prove the statement of the lemma for such elements: Since $^{M}\rho(hu)=\Delta(h)\:^{M}\rho(u)=h\cdot\:^{M}\rho(u)=h_11_H\otimes h_2\mu^{-1}(u)=\alpha(h_1)\otimes h_2\mu^{-1}(u)$ and thus $$(id\otimes\:^{M}\rho)(^{M}\rho(hu))=\alpha(h_1)\otimes(\alpha(h_{21})\otimes h_{22}\mu^{-2}(u)),$$ we have

\begin{eqnarray}\label{first_form_of_baraiding}\nonumber c_{M,N}(hu\otimes_Hv)&=&\alpha(h_1)(v_{[0][0]}S(v_{[0][1]}))\otimes_H(S(\alpha(h_{21}))(h_{22}\mu^{-2}(u)))v_{[1]}\\
\nonumber&=&\alpha(h_1)(v_{[0][0]}S(v_{[0][1]}))\otimes_H((\alpha^{-1}(S(\alpha(h_{21})))h_{22})\mu^{-1}(u))v_{[1]}\\
\nonumber&=&\alpha(h_1)(v_{[0][0]}S(v_{[0][1]}))\otimes_H((\varepsilon(h_2)1_H)\mu^{-1}(u))v_{[1]}\\
\nonumber&=&\alpha(h_1\varepsilon(h_2))(v_{[0][0]}S(v_{[0][1]}))\otimes_H(1_H\mu^{-1}(u))v_{[1]}\\
\nonumber&=&h(v_{[0][0]}S(v_{[0][1]}))\otimes_Huv_{[1]}\\
\nonumber&=&h(\nu^{-1}(v_{[0]})S(v_{[1]1}))\otimes_Hu\alpha(v_{[1]2})\\
\nonumber&=&(\alpha^{-1}(h)\nu^{-1}(v_{[0]}))\alpha(S(v_{[1]1}))\otimes_Hu\alpha(v_{[1]2})\\
\nonumber&=&\nu(\alpha^{-1}(h)\nu^{-1}(v_{[0]}))\otimes_H\alpha(S(v_{[1]1}))\mu^{-1}(u\alpha(v_{[1]2}))\\
\nonumber&=&hv_{[0]}\otimes_H\alpha(S(v_{[1]1}))(\mu^{-1}(u)v_{[1]2})\\
\nonumber&=&hv_{[0]}\otimes_H(S(v_{[1]1})\mu^{-1}(u))\alpha(v_{[1]2})\\
\nonumber&=&hv_{[0]}\otimes_H\widetilde{ad}_{R}(v_{[1]})u\\
&=&hv_{[0]}\otimes_Hu\lhd v_{[1]}.
\end{eqnarray}
Similarly, we obtain the following equations
\begin{equation}\label{second_form_braiding_bicov_Hom_bimod}c_{M,N}(w\otimes_H zh)=w_{(-1)}\rhd z\otimes_H w_{(0)}h,\end{equation}
\begin{equation}\label{third_form_braiding_bicov_Hom_bimod}c_{M,N}(hu\otimes_H z)=h\nu^{-1}(z)\otimes_H\mu(u).\end{equation}
By using the formula (\ref{second_form_braiding_bicov_Hom_bimod}), we now prove the right $(H,\alpha)$-linearity and then in the sequel the right $(H,\alpha)$-colinearity of $c_{M,N}$:

\begin{eqnarray*}c_{M,N}((w\otimes_Hzh)g)&=&c_{M,N}(\mu(w)\otimes_H(zh)\alpha^{-1}(g))\\
&=&c_{M,N}(\mu(w)\otimes_H\nu(z)(h\alpha^{-2}(g)))\\
&=&\mu(w)_{(-1)}\rhd \nu(z)\otimes_H\mu(w)_{(0)}(h\alpha^{-2}(g))\\
&=&\alpha(w_{(-1)})\rhd \nu(z)\otimes_H\mu(w_{(0)})(h\alpha^{-2}(g))\\
&=&\nu(w_{(-1)}\rhd z)\otimes_H(w_{(0)}h)\alpha^{-1}(g))\\
&=&c_{M,N}(w\otimes_Hzh)g,
\end{eqnarray*}

\begin{eqnarray*}(c_{M,N}\otimes id_H)(\sigma^{M\otimes_HN}(w\otimes_Hzh))&=&(c_{M,N}\otimes id_H)((w_{[0]}\otimes_H(zh)_{[0]})\otimes w_{[1]}(zh)_{[1]})\\
&=&c_{M,N}(w_{[0]}\otimes_H z_{[0]}h_1)\otimes w_{[1]}(z_{[1]}h_2)\\
&=&c_{M,N}(\mu^{-1}(w)\otimes_H \nu^{-1}(z)h_1)\otimes 1_H(1_Hh_2)\\
&=&(\alpha^{-1}(w_{(-1)})\rhd \nu^{-1}(z)\otimes_H\mu^{-1}(w_{(0)})h_1)\otimes 1_H(1_Hh_2)\\
&=&(\nu^{-1}(w_{(-1)}\rhd z)\otimes_H\mu^{-1}(w_{(0)})h_1)\otimes 1_H(1_Hh_2)\\
&=&(\nu^{-1}(w_{(-1)}\rhd z)\otimes_H w_{(0)[0]}h_1)\otimes 1_H(w_{(0)[1]}h_2)\\
&=&((w_{(-1)}\rhd z)_{[0]}\otimes_H(w_{(0)}h)_{[0]})\otimes(w_{(-1)}\rhd z)_{[1]}(w_{(0)}h)_{[1]}\\
&=&\sigma^{N\otimes_HM}(w_{(-1)}\rhd z\otimes_H w_{(0)}h)\\
&=&\sigma^{N\otimes_HM}(c_{M,N}(w\otimes_Hzh)),
\end{eqnarray*}
where the sixth equality follows from the fact that $^{M}\rho(w)\in \: H\otimes M^{coH}$ and the seventh one results from $w_{(-1)}\rhd z \in N^{coH}$. Analogously, one can also show that $c_{M,N}$ is both left $(H,\alpha)$-linear and left $(H,\alpha)$-colinear, which finishes the proof.
\end{proof}

\begin{proposition}\label{quasi_braiding_bicov_Hom_bimod}Let $(H,\alpha)$ be a monoidal Hom-Hopf algebra with a bijective antipode. Then the $k$-linear map $c_{M,N}:M\otimes_HN\to N\otimes_HM$ given by (\ref{braiding_bicov_Hom_bimod}) in the above lemma is an isomorphism in
$^{H}_{H}\widetilde{\mathcal{H}}(\mathcal{M}_k)_H^H$.
\end{proposition}

\begin{proof} In the above lemma, it has already been proven that $c_{M,N}$, where $(M,\mu)$ and $(N,\nu)$ are bicovariant $(H,\alpha)$-Hom-bimodules, is a morphism in $^{H}_{H}\widetilde{\mathcal{H}}(\mathcal{M}_k)_H^H$. Hereby we prove that it is an invertible linear map to finish the proof of the proposition: Define the $k$-linear map $c^{-1}_{N,M}:N\otimes_HM\to M\otimes_HN$ by
\begin{equation}\label{inverse_braiding_bicov_Hom_bimod}
                 c^{-1}_{N,M}(n\otimes_Hm)=n_{[1]}(m_{(0)(0)}S^{-1}(m_{(0)(-1)}))\otimes_H(S^{-1}(n_{[0][1]})n_{[0][0]})m_{(-1)}.\end{equation}

For $h\in H$, $u\in\:\: ^{coH}M$, $v\in\:\:^{coH}N$, we get
\begin{eqnarray}\label{inverse_of_first_form_of_baraiding}
\nonumber c^{-1}_{M,N}(hv\otimes_Hu)&=&(h_2v_{[1]})(\mu^{-2}(u)S^{-1}(1_H))\otimes_H(S^{-1}(h_{12}v_{[0][1]})(h_{11}v_{[0][0]}))1_H\\
\nonumber&=&(h_2v_{[1]})\mu^{-1}(u)\otimes_H(S^{-1}(v_{[0][1]})S^{-1}(h_{12}))(h_{11}v_{[0][0]}))1_H\\
\nonumber&=&\alpha(h_2)(v_{[1]}\mu^{-2}(u))\otimes_H((\alpha^{-1}(S^{-1}(v_{[0][1]})S^{-1}(h_{12}))h_{11})\nu(v_{[0][0]}))1_H\\
\nonumber&=&\alpha(h_2)(v_{[1]}\mu^{-2}(u))\otimes_H((S^{-1}(v_{[0][1]})\alpha^{-1}(S^{-1}(h_{12})h_{11}))\nu(v_{[0][0]}))1_H\\
\nonumber&=&\alpha(\varepsilon(h_1)h_2)(v_{[1]}\mu^{-2}(u))\otimes_H((S^{-1}(v_{[0][1]})1_H)\nu(v_{[0][0]}))1_H\\
\nonumber&=&h(\alpha(v_{[1]2})\mu^{-2}(u))\otimes_H(\alpha(S^{-1}(v_{[1]1}))v_{[0]})1_H\\
\nonumber&=&h(\alpha(v_{[1]2})\mu^{-2}(u))\otimes_H\alpha^{2}(S^{-1}(v_{[1]1}))\nu(v_{[0]})\\
\nonumber&=&(\alpha^{-1}(h)(v_{[1]2}\mu^{-3}(u)))\alpha^{2}(S^{-1}(v_{[1]1}))\otimes_H\nu^{2}(v_{[0]})\\
\nonumber&=&h((S(S^{-1}(v_{[1]})_1)\mu^{-1}(\mu^{-2}(u)))\alpha(S^{-1}(v_{[1]})_2))\otimes_H\nu^{2}(v_{[0]})\\
\nonumber&=&h((v_{[1]2}\mu^{-3}(u))\alpha(S^{-1}(v_{[1]1})))\otimes_H\nu^{2}(v_{[0]})\\
\nonumber&=&h(\widetilde{ad}_{R}(S^{-1}(v_{[1]}))\mu^{-2}(u))\otimes_H\nu^{2}(v_{[0]})\\
&=&h(\mu^{-2}(u) \lhd S^{-1}(v_{[1]}))\otimes_H\nu^{2}(v_{[0]}),
\end{eqnarray}
and we now verify that $c^{-1}_{M,N}$ is the inverse of $c_{M,N}$ in this case;

\begin{eqnarray*}c^{-1}_{M,N}(c_{M,N}(hu\otimes_Hv))&=&c^{-1}_{M,N}(hv_{[0]}\otimes_Hu\lhd v_{[1]})\\
&=&h(\mu^{-2}(u\lhd v_{[1]})\lhd S^{-1}(v_{[0][1]}))\otimes_H\nu^{2}(v_{[0][0]})\\
&=&h((\mu^{-2}(u)\lhd \alpha^{-2}(v_{[1]}))\lhd S^{-1}(v_{[0][1]}))\otimes_H\nu^{2}(v_{[0][0]})\\
&=&h(\mu^{-1}(u)\lhd\alpha^{-1}(\alpha^{-1}(v_{[1]})S^{-1}(v_{[0][1]})))\otimes_H\nu^{2}(v_{[0][0]})\\
&=&h(\mu^{-1}(u)\lhd\alpha^{-1}(v_{[1]2} S^{-1}(v_{[1]1})))\otimes_H\nu(v_{[0]})\\
&=&h(\mu^{-1}(u)\lhd 1_H)\otimes_H\nu(v_{[0]}\varepsilon(v_{[1]}))\\
&=&hu\otimes_Hv,
\end{eqnarray*}
and
\begin{eqnarray*}c_{M,N}(c^{-1}_{M,N}(hv\otimes_Hu))&=&c_{M,N}(h(\mu^{-2}(u) \lhd S^{-1}(v_{[1]}))\otimes_H\nu^{2}(v_{[0]}))\\
&=&h\nu^{2}(v_{[0][0]})\otimes_H(\mu^{-2}(u) \lhd S^{-1}(v_{[1]}))\lhd \alpha^{2}(v_{[0][1]})\\
&=&h\nu^{2}(v_{[0][0]})\otimes_H\mu^{-1}(u) \lhd ( S^{-1}(v_{[1]})\alpha(v_{[0][1]}))\\
&=&h\nu(v_{[0]})\otimes_H\mu^{-1}(u)\lhd \alpha( S^{-1}(v_{[1]2})v_{[1]1})\\
&=&h\nu(v_{[0]}\varepsilon(v_{[1]}))\otimes_H\mu^{-1}(u)\lhd 1_H\\
&=&hv\otimes_Hu.
\end{eqnarray*}
By a similar reasoning we obtain, for $h\in H$, $w\in M^{coH}$, $z\in N^{coH}$ and $u\in\:\: ^{coH}M$, the following formulas:
\begin{equation}\label{inverse_second_form_braiding_bicov_Hom_bimod}c^{-1}_{M,N}(z\otimes_H wh)=\mu^{2}(w_{(0)})\otimes_H (S^{-1}(w_{(-1)})\rhd \nu^{-2}(z))h,\end{equation}
\begin{equation}\label{inverse_third_form_braiding_bicov_Hom_bimod}c^{-1}_{M,N}(hz\otimes_H u)=h\mu^{-1}(u)\otimes_H\nu(z),\end{equation}
and thus for each of the sets $\{w\otimes_H zh\}$, $\{hu\otimes_H z\}$ linearly spanning $M\otimes_HN$, we also get $c^{-1}_{M,N}(c_{M,N}(w\otimes_Hzh))=w\otimes_Hzh$, $c_{M,N}(c^{-1}_{M,N}(z\otimes_Hwh))=z\otimes_Hwh$, $c^{-1}_{M,N}(c_{M,N}(hu\otimes_Hz))=hu\otimes_Hz$ and $c_{M,N}(c^{-1}_{M,N}(hz\otimes_Hu))=hz\otimes_Hu$.
\end{proof}

\begin{theorem}$^{H}_{H}\widetilde{\mathcal{H}}(\mathcal{M}_k)_H^H$ is a prebraided tensor category. It is a braided monoidal category if $(H,\alpha)$ has an invertible antipode.
\end{theorem}

\begin{proof} We have already verified that $^{H}_{H}\widetilde{\mathcal{H}}(\mathcal{M}_k)_H^H$ is a tensor category, with tensor product
$\otimes_H$ is defined in Proposition (\ref{tensor_prod_bicov_Hom_bimod}), and associativity constraint $\tilde{a}$, left unit constraint $\tilde{l}$ and right
unit constraint $\tilde{r}$ are given by (\ref{associator-formula}), (\ref{left-unit-formula}) and (\ref{right-unit-formula}), respectively. Thereby, together with
the Proposition (\ref{quasi_braiding_bicov_Hom_bimod}), to demonstrate that the Hexagon Axioms for $c_{M,N}$ hold finishes the proof the statement. Since
$(M\otimes_HN)\otimes_HP$ is generated as a left  $(H,\alpha)$-Hom-module by the elements $(u\otimes_Hz)\otimes_Hp$ where
$u\in\:^{coH}M$, $z\in N^{coH}$ and $p\in P^{coH}$, it is sufficient to prove the hexagonal relations for such elements.
One can first note that $c_{M,N}(u\otimes_H z)=z\otimes_H u$ and thus

\begin{eqnarray*}\lefteqn{(\tilde{a}_{N,P,M}\circ c_{M,N\otimes_HP}\circ \tilde{a}_{M,N,P})((u\otimes_Hz)\otimes_Hp)}\hspace{4em}\\
&=&\nu(z)\otimes_H(\pi^{-1}(p)\otimes_Hu)\\
&=&(id_N\otimes c_{M,P})(\nu(z)\otimes_H(u\otimes_H\pi^{-1}(p))\\
&=&(id_N\otimes_H c_{M,P}\circ\tilde{a}_{N,M,P})((z\otimes_Hu)\otimes_Hp)\\
&=&(id_N\otimes_H c_{M,P}\circ\tilde{a}_{N,M,P}\circ (c_{M,N}\otimes_Hid_P))((u\otimes_Hz)\otimes_Hp),
\end{eqnarray*}
which asserts the first hexagon axiom and the second one is obtained by a similar reasoning.
\end{proof}
\begin{remark}The (pre)braiding $c_{M,N}$ defined by (\ref{braiding_bicov_Hom_bimod}) is called Woronowicz' (pre)braiding.
\end{remark}
\begin{lemma}Let $(H,\alpha)$ be a monoidal Hom-Hopf algebra with bijective antipode and $(M,\mu)$ be a bicovariant Hom-bimodule with left Hom-coaction $m\mapsto m_{(-1)}\otimes m_{(0)} $ and right Hom-coaction $m\mapsto m_{[0]}\otimes m_{[1]}$. Then the morphism $\Phi:M\to M$, in $\widetilde{\mathcal{H}}(\mathcal{M}_k)$,  given by
$$\Phi(m)=(S(m_{[0](-1)})m_{[0](0)})S(m_{[1]})=S(m_{(-1)})(m_{(0)[0]}S(m_{(0)[1]}))$$
is bijective. Furthermore, it restricts to an isomorphism of the subobjects $^{coH}M$ and $M^{coH}$.
\end{lemma}
\begin{proof}Let us set $\Psi(m)=(S^{-1}(m_{(0)[1]})m_{(0)[0]})S^{-1}(m_{(-1)})$. Since the following equality holds:

\begin{eqnarray*}\lefteqn{\Phi(m)_{(-1)}\otimes \Phi(m)_{(0)[0]}\otimes \Phi(m)_{(0)[1]}}\hspace{3em}\\
&=&S(\alpha(m_{[1]2}))\otimes (S(\alpha^{-1}(m_{[0](-1)2}))\mu^{-2}(m_{[0](0)}))S(\alpha^{-1}(m_{[1]1}))\otimes S(\alpha(m_{[0](-1)1})),
\end{eqnarray*}
we compute
\begin{eqnarray*} &&\Psi(\Phi(m))\\
&=&(S^{-1}(\Phi(m)_{(0)[1]})\Phi(m)_{(0)[0]})S^{-1}(\Phi(m)_{(-1)})\\
&=&(S^{-1}(S(\alpha(m_{[0](-1)1})))[(S(\alpha^{-1}(m_{[0](-1)2}))\mu^{-2}(m_{[0](0)}))S(\alpha^{-1}(m_{[1]1}))])S^{-1}(S(\alpha(m_{[1]2})))\\
&=&(\alpha(m_{[0](-1)1})[(S(\alpha^{-1}(m_{[0](-1)2}))\mu^{-2}(m_{[0](0)}))S(\alpha^{-1}(m_{[1]1}))])\alpha(m_{[1]2})\\
&=&([m_{[0](-1)1}(S(\alpha^{-1}(m_{[0](-1)2}))\mu^{-2}(m_{[0](0)}))]S(m_{[1]1}))\alpha(m_{[1]2})\\
&=&([(\alpha^{-1}(m_{[0](-1)1})S(\alpha^{-1}(m_{[0](-1)2})))\mu^{-1}(m_{[0](0)})]S(m_{[1]1}))\alpha(m_{[1]2})\\
&=&([\alpha^{-1}(m_{[0](-1)1}S(m_{[0](-1)2}))\mu^{-1}(m_{[0](0)})]S(m_{[1]1}))\alpha(m_{[1]2})\\
&=&(\varepsilon(m_{[0](-1)})m_{[0](0)}S(m_{[1]1}))\alpha(m_{[1]2})\\
&=&(\mu^{-1}(m_{[0]})S(m_{[1]1}))\alpha(m_{[1]2})\\
&=&m_{[0]}(S(m_{[1]1})m_{[1]2})=m.
\end{eqnarray*}
In a similar way, one can easily get $\Phi(\Psi(m))=m$ for any $m\in M$ meaning $\Phi$ is bijective with inverse $\Psi$. It can also be shown that $\mu\circ\Phi =\Phi\circ\mu$ and $\mu\circ\Psi =\Psi\circ\mu$. To prove the second statement in the lemma, we next show that $\Phi:\:^{coH}M\to M^{coH}$ and $\Psi:M^{coH}\:\to \: ^{coH}M$: For any $m\in\: ^{coH}M$, we obtain
\begin{eqnarray*}\Phi(m)&=&S(m_{(-1)})(m_{(0)[0]}S(m_{(0)[1]}))\\
&=&S(1)(\mu^{-1}(m_{[0]})S(\alpha^{-1}(m_{[1]})))=m_{[0]}S(m_{[1]})=P_R(m),
\end{eqnarray*}
that is, $\Phi(m)\in M^{coH}$, and for any $n\in M^{coH}$, we have
\begin{eqnarray*}\Psi(n)&=&(S^{-1}(n_{(0)[1]})n_{(0)[0]})S^{-1}(n_{(-1)})\\
&=&S^{-1}(n_{[1]})(n_{[0](0)}S^{-1}(n_{[0](-1)}))=S^{-1}(1)(\mu^{-1}(n_{(0)})S^{-1}(\alpha^{-1}(n_{(-1)})))\\
&=&n_{(0)}S^{-1}(n_{(-1)})=(n_{(0)(-1)}P_L(n_{(0)(0)}))S^{-1}(n_{(-1)})\\
&=&(n_{(0)(-1)}S^{-1}(n_{(-1)})_1)(P_L(n_{(0)(0)})\lhd S^{-1}(n_{(-1)})_2)\\
&=&(n_{(-1)2}S^{-1}(\alpha(n_{(-1)1}))_1)(P_L(\mu^{-1}(n_{(0)}))\lhd S^{-1}(\alpha(n_{(-1)1}))_2)\\
&=&(n_{(-1)2}S^{-1}(\alpha(n_{(-1)12})))(P_L(\mu^{-1}(n_{(0)}))\lhd S^{-1}(\alpha(n_{(-1)11})))\\
&=&(\alpha(n_{(-1)22})S^{-1}(\alpha(n_{(-1)21})))(P_L(\mu^{-1}(n_{(0)}))\lhd S^{-1}(n_{(-1)1}))\\
&=&\alpha(\varepsilon(n_{(-1)2})1)(P_L(\mu^{-1}(n_{(0)}))\lhd S^{-1}(n_{(-1)1}))\\
&=&1(P_L(\mu^{-1}(n_{(0)}))\lhd S^{-1}(\alpha^{-1}(n_{(-1)})))=P_L(n_{(0)})\lhd S^{-1}(n_{(-1)}),
\end{eqnarray*}
i.e., $\Psi(n)\in\: ^{coH}M$ for all $n\in M^{coH}$.
\end{proof}

We now restate the structure theory of bicovariant Hom-bimodules in coordinate form in the following, here we assume that the scalars belong to a field $k$.

\begin{theorem}Let $(M,\mu)$ be a bicovariant $(H,\alpha)$-Hom-bimodule with right $(H,\alpha)$-Hom-coaction $\varphi:M\to M\otimes H, \: m\mapsto m_{[0]}\otimes m_{[1]}$ and $\{m_{i}\}_{i\in I}$ be a linear basis of $\:^{coH}M$. Then there exist a pointwise finite matrices $(f^{i}_j)_{i,j\in I}$ and $(v^{i}_j)_{i,j\in I}$ of linear functionals $f^{i}_j\in H'$ and elements $v^{i}_j\in H$ such that for any $h,g\in H$ and $i,j,k \in I$ we have
\begin{enumerate}
\item $\mu^{i}_jf^j_k(hg)=f^{i}_j(h)f^j_k(\alpha(g)), \: f^{i}_j(1)=\mu^{i}_j;\:m_{i}h=(\bar{\mu}^{i}_jf^j_k\bullet\alpha^{-1}(h))m_k,$
\item $\varphi(m_i)=m_j\otimes v^j_i,$
where $v^{i}_j\in H,\: i,j\in I$, satisfy the relations
$$\Delta(v^l_i)=\mu^j_lv^j_k\otimes \alpha^{-1}(v^k_i),\: \varepsilon(v^k_i)=\bar{\mu}^{i}_k,$$
\item the equality
\begin{equation}v^k_i(h\bullet (f^k_j\circ\alpha))=((f^{i}_k\circ\alpha^{2})\bullet h)\alpha^{-1}(v^j_k)\end{equation}

holds. Moreover, $\{n_{i}:=m_jS(v^j_i)\}_{i\in I}$ is a linear basis of $M^{coH}$. $\{m_{i}\}_{i\in I}$ and $\{n_{i}\}_{i\in I}$ are both free left $(H,\alpha)$-Hom-module bases and free right $(H,\alpha)$-Hom-module bases of $M$.
\end{enumerate}
\end{theorem}

\begin{proof} (1) had already been proven in Theorem (\ref{fund-thm-of-left-cov-in-coord-form}). Since $\varphi(^{coH}M)\subseteq\:^{coH}M\otimes H$, there exists a pointwise finite matrix $(v^{i}_j)_{i,j\in I}$ of elements $v^{i}_j\in H$ such that $\varphi(m_i)=m_k\otimes v^{k}_i$. Let us write $\varphi(m_i)=m_{i,[0]}\otimes m_{i,[1]}$. Then, by the Hom-coassociativity and Hom-counity of $\varphi$ we have
\begin{eqnarray*}(\bar{\mu}^k_jm_j)\otimes\Delta(v^k_{i})&=&\mu^{-1}(m_{k})\otimes \Delta(v^k_{i})=\mu^{-1}(m_{i,[0]})\otimes \Delta(m_{i,[1]})\\
&=&m_{i,[0][0]}\otimes m_{i,[0][1]}\otimes \alpha^{-1}(m_{i,[1]})\\
&=&m_j\otimes v^j_k\otimes \alpha^{-1}(v^k_i),
\end{eqnarray*}
which implies $\Delta(v^l_{i})=\mu^j_lv^j_k\otimes \alpha^{-1}(v^k_i)$ by the relation $\mu^j_l\bar{\mu}^k_j=\delta_{lk}$, and
$$\bar{\mu}^{i}_km_k=\mu^{-1}(m_i)=m_{i,[0]}\varepsilon(m_{i,[1]})=m_k\varepsilon(v^k_i),$$
which finishes the proof of item (2). To prove (3), let $i\in I$ and $h\in H$. Then
\begin{eqnarray*}m_{i,[0]}\lhd\alpha^{-1}(h_1)\otimes m_{i,[1]}\alpha^{-1}(h_2)&=&m_{k}\lhd\alpha^{-1}(h_1)\otimes v^k_i\alpha^{-1}(h_2)\\
&=&f^k_j(\alpha^{-1}(h_1))m_j\otimes v^k_i\alpha^{-1}(h_2)\\
&=&m_j\otimes v^k_i(f^k_j(\alpha^{-1}(h_1))\alpha^{-1}(h_2))\\
&=&m_j\otimes v^k_i\alpha^{-1}(f^k_j(\alpha^{-1}(h_1))h_2)\\
&=&m_j\otimes v^k_i\alpha^{-1}(\alpha^{-2}(h)\bullet f^k_j),
\end{eqnarray*}

\begin{eqnarray*}(m_i\lhd h_2)_{[0]}\otimes \alpha^{-1}(h_1)\alpha^{-1}((m_i\lhd h_2)_{[1]})&=&m_{k}\otimes \alpha^{-1}(h_1)f^{i}_j(h_2)\alpha^{-1}(v^k_j)\\
&=&m_{k}\otimes \alpha^{-1}(h_1f^{i}_j(h_2))\alpha^{-1}(v^k_j)\\
&=&m_{k}\otimes \alpha^{-1}((f^{i}_j\circ\alpha)\bullet\alpha^{-2}(h))\alpha^{-1}(v^k_j).
\end{eqnarray*}
Thus, by Hom-Yetter-Drinfeld condition (\ref{Yetter-Drinfeld}), we acquire
$$v^k_i\alpha^{-1}(\alpha^{-2}(h)\bullet f^k_j)=\alpha^{-1}((f^{i}_k\circ\alpha)\bullet\alpha^{-2}(h))\alpha^{-1}(v^j_k),$$
that is, $v^k_i(\alpha^{-3}(h)\bullet (f^k_j\circ\alpha))=((f^{i}_k\circ\alpha^{2})\bullet\alpha^{-3}(h))\alpha^{-1}(v^j_k)$ holds. If we replace $\alpha^{-3}(h)$ by $h$, we get the required equality $v^k_i(h\bullet (f^k_j\circ\alpha))=((f^{i}_k\circ\alpha^{2})\bullet h)\alpha^{-1}(v^j_k)$. By the above Lemma, we obtain $n_i=\Phi(m_i)=m_{i,[0]}S(m_{i,[1]})=m_kS(v^k_i)$ for all $i\in I$. In Theorem (\ref{fund-thm-of-left-cov-in-coord-form}), we have shown that $\{m_{i}\}_{i\in I}$ is both free left $(H,\alpha)$-Hom-module basis and free right $(H,\alpha)$-Hom-module basis of $M$. Similarly, one can prove that this statement also holds for  $\{n_{i}\}_{i\in I}$.
\end{proof}

\section{Yetter-Drinfel'd Modules over Monoidal Hom-Hopf Algebras}
 In this section, we present and study the category of Yetter-Drinfeld modules over a monoidal Hom-bialgebra $(H,\alpha)$, and then demonstrate that if $(H,\alpha)$ is a monoidal Hom-Hopf algebra with an invertible antipode it is a braided monoidal category and tensor equivalent to $^{H}_{H}\widetilde{\mathcal{H}}(\mathcal{M}_k)_H^H$.

 \begin{definition} Let $(H,\alpha)$ be a monoidal Hom-bialgebra, $(N,\nu)$ be a right $(H,\alpha)$-Hom-module with Hom-action $N\otimes H\to N,\; n\otimes h\mapsto n\lhd h$ and a right $(H,\alpha)$-Hom-comodule with Hom-coaction $N\to N\otimes H,\: n\mapsto n_{(0)}\otimes n_{(1)}$. Then $(N,\nu)$ is called a {\it right-right} $(H,\alpha)$-{\it Hom-Yetter-Drinfeld module} if the condition (\ref{Yetter-Drinfeld} ) holds, that is,
 \begin{equation*}n_{(0)}\lhd\alpha^{-1}(h_1)\otimes n_{(1)}\alpha^{-1}(h_2)=(n\lhd h_2)_{(0)}\otimes\alpha^{-1}(h_1(n\lhd h_2)_{(1)}),\end{equation*}
 for all $h\in H$ and $n\in N$.
 \end{definition}

 We denote by  $\widetilde{\mathcal{H}}(\mathcal{YD})_H^H$ the category of $(H,\alpha)$-Hom-Yetter-Drinfeld modules whose objects are Yetter-Drinfeld modules over the monoidal Hom-bialgebra $(H,\alpha)$ and morphisms are the ones that are right $(H,\alpha)$-linear and right $(H,\alpha)$-colinear.

\begin{proposition}\label{tensor-prod-Y-D-category} Let $(H,\alpha)$ be a monoidal Hom-bialgebra and $(M,\mu)$,$(N,\nu)$ be two $(H,\alpha)$-Hom-Yetter-Drinfeld modules. Then $(M\otimes N, \mu\otimes\nu)$ becomes a $(H,\alpha)$-Hom-Yetter-Drinfeld module with the following structure maps
\begin{equation}(M\otimes N)\otimes H\to M\otimes N,\:\: (m\otimes n)\otimes h\mapsto m\lhd h_1\otimes n\lhd h_2=(m\otimes n)\lhd h,\end{equation}
\begin{equation}M\otimes N\to (M\otimes N)\otimes H ,\:\: m\otimes n\mapsto (m_{(0)}\otimes n_{(0)})\otimes m_{(1)}n_{(1)}. \end{equation}
\end{proposition}

\begin{proof} $(M\otimes N, \mu\otimes\nu)$ is both right $(H,\alpha)$-Hom-module and a right $(H,\alpha)$-Hom-comodule; to verify this one can see Propositions 2.6 and 2.8 in \cite{CaenepeelGoyvaerts} for the left case. We only prove that the Hom-Yetter-Drinfeld condition is fulfilled for $(M\otimes N, \mu\otimes\nu)$: For $h\in H$, $m\in M$ and $n\in N$,

\begin{eqnarray*}\lefteqn{((m\otimes n)\lhd h_2)_{(0)}\otimes \alpha^{-1}(h_1)\alpha^{-1}(((m\otimes n)\lhd h_2)_{(1)})}\hspace{3em}\\
 &=&(m\lhd h_{21}\otimes n\lhd h_{22})_{(0)}\otimes \alpha^{-1}(h_1)\alpha^{-1}((m\lhd h_{21}\otimes n\lhd h_{22})_{(1)})\\
 &=&(m\lhd h_{12}\otimes n\lhd \alpha^{-1}(h_2))_{(0)}\otimes h_{11}\alpha^{-1}((m\lhd h_{12}\otimes n\lhd \alpha^{-1}(h_2))_{(1)})\\
 &=&(m\lhd h_{12})_{(0)}\otimes (n\lhd \alpha^{-1}(h_2))_{(0)}\otimes h_{11}\alpha^{-1}((m\lhd h_{12})_{(1)}(n\lhd \alpha^{-1}(h_2))_{(1)})\\
 &=&(m\lhd h_{12})_{(0)}\otimes (n\lhd \alpha^{-1}(h_2))_{(0)}\otimes h_{11}\alpha^{-1}((m\lhd h_{12})_{(1)})\alpha^{-1}((n\lhd \alpha^{-1}(h_2))_{(1)})\\
 &=&(m\lhd h_{12})_{(0)}\otimes (n\lhd \alpha^{-1}(h_2))_{(0)}\otimes(\alpha^{-1}( h_{11})\alpha^{-1}((m\lhd h_{12})_{(1)}))(n\lhd \alpha^{-1}(h_2))_{(1)}\\
 &=&m_{(0)}\lhd \alpha^{-1}( h_{11})\otimes(n\lhd \alpha^{-1}(h_2))_{(0)}\otimes( m_{(1)}\lhd \alpha^{-1}( h_{12}))(n\lhd \alpha^{-1}(h_2))_{(1)}\\
 &=&m_{(0)}\lhd \alpha^{-2}( h_1)\otimes (n\lhd h_{22})_{(0)}\otimes \alpha(m_{(1)})(\alpha^{-1}( h_{21})\alpha^{-1}((n\lhd h_{22})_{(1)}))\\
 &=&m_{(0)}\lhd \alpha^{-2}( h_1)\otimes n_{(0)}\lhd\alpha^{-1}( h_{21})\otimes\alpha(m_{(1)}(n_{(1)}\alpha^{-1}(h_{22}))\\
 &=&m_{(0)}\lhd \alpha^{-1}( h_{11})\otimes n_{(0)}\lhd\alpha^{-1}( h_{12})\otimes (m_{(1)}n_{(1)})\alpha^{-1}(h_2)\\
 &=&(m_{(0)}\otimes n_{(0)})\lhd \alpha^{-1}(h_1)\otimes (m_{(1)}n_{(1)})\alpha^{-1}(h_2)\\
 &=&(m\otimes n)_{(0)}\lhd\alpha^{-1}(h_1)\otimes(m\otimes n)_{(1)}\lhd\alpha^{-1}(h_2).
 \end{eqnarray*}
\end{proof}

\begin{proposition}\label{associator-Y-D-category}Let $(H,\alpha)$ be a monoidal Hom-bialgebra and $(M,\mu)$, $(N,\nu)$, $(P,\pi)$ be $(H,\alpha)$-Hom-Yetter-Drinfeld modules. Then the $k$-linear map $\tilde{a}_{M,N,P}:(M\otimes N)\otimes P\to M\otimes(N\otimes P), \tilde{a}_{M,N,P}((m\otimes n)\otimes p)=(\mu(m)\otimes(n\otimes \pi^{-1}(p))) $ is a right $(H,\alpha)$-linear and right $(H,\alpha)$-colinear isomorphism.
\end{proposition}

\begin{proof} The bijectivity of $\tilde{a}_{M,N,P}$ is obvious with the inverse $\tilde{a}^{-1}_{M,N,P}(m\otimes (n\otimes p))=((\mu^{-1}(m)\otimes n)\otimes \pi(p))$.
\begin{eqnarray*}\tilde{a}_{M,N,P}(((m\otimes n)\otimes p)\lhd h)&=&\tilde{a}_{M,N,P}((m\lhd h_{11}\otimes n\lhd h_{12})\otimes p\lhd h_2)\\
&=&\mu(m\lhd h_{11})\otimes( n\lhd h_{12}\otimes \pi^{-1}(p\lhd h_2))\\
&=&\mu(m)\lhd \alpha(h_{11})\otimes( n\lhd h_{12}\otimes \pi^{-1}(p)\lhd \alpha^{-1}(h_2))\\
&=&\mu(m)\lhd h_1\otimes( n\lhd h_{21}\otimes \pi^{-1}(p)\lhd h_{22})\\
&=&\mu(m)\lhd h_1\otimes(( n\otimes \pi^{-1}(p))\lhd h_2)\\
&=&(\mu(m)\otimes(n\otimes \pi^{-1}(p)))\lhd h=\tilde{a}_{M,N,P}((m\otimes n)\otimes p)\lhd h,
\end{eqnarray*}
which proves the $(H,\alpha)$-linearity.
Below we show the $(H,\alpha)$-colinearity:
\begin{eqnarray*}\rho^{M\otimes(N\otimes P)}(\tilde{a}_{M,N,P}((m\otimes n)\otimes p))&=&\rho^{M\otimes(N\otimes P)}(\mu(m)\otimes(n\otimes \pi^{-1}(p)))\\
&=&(\mu(m)_{(0)}\otimes(n\otimes \pi^{-1}(p))_{(0)})\otimes\mu(m)_{(1)}(n\otimes \pi^{-1}(p))_{(1)}\\
&=&(\mu(m_{(0)})\otimes(n_{(0)}\otimes\pi^{-1}(p_{(0)})))\otimes\alpha(m_{(1)})(n_{(1)}\alpha^{-1}(p_{(1)}))\\
&=&(\mu(m_{(0)})\otimes(n_{(0)}\otimes\pi^{-1}(p_{(0)})))\otimes (m_{(1)}n_{(1)})p_{(1)},
\end{eqnarray*}

\begin{eqnarray*}\lefteqn{(\tilde{a}_{M,N,P}\otimes id_H)(\rho^{(M\otimes N)\otimes P}((m\otimes n)\otimes p))}\hspace{6em}\\
&=&(\tilde{a}_{M,N,P}\otimes id_H)(((m\otimes n)_{(0)}\otimes p_{(0)})\otimes(m\otimes n)_{(1)}p_{(1)})\\
&=&(\tilde{a}_{M,N,P}\otimes id_H)(((m_{(0)}\otimes n_{(0)})\otimes p_{(0)})\otimes(m_{(1)} n_{(1)})p_{(1)})\\
&=&(\mu(m_{(0)})\otimes (n_{(0)}\otimes \pi^{-1}(p_{(0)})))\otimes(m_{(1)} n_{(1)})p_{(1)},
\end{eqnarray*}
where $\rho^{Q}$ denotes the right $(H,\alpha)$-Hom-comodule structure of a  Hom-Yetter-Drinfeld module $Q$.
\end{proof}

\begin{proposition}\label{unitors-Y-D-category}Let $(H,\alpha)$ be a monoidal Hom-bialgebra and $(M,\mu)\in \widetilde{\mathcal{H}}(\mathcal{YD})_H^H$. Then the k-linear maps given by
\begin{equation}\tilde{l}_M:k\otimes M\to M,\: x\otimes m\mapsto x\mu(m), \end{equation}
\begin{equation}\tilde{r}_M:M\otimes k\to M,\: m\otimes x\mapsto x\mu(m) \end{equation}
are isomorphisms of right $(H,\alpha)$-Hom-modules and right $(H,\alpha)$-Hom-comodules.
\end{proposition}
\begin{proof} In the category $\mathcal{M}_k$ of $k$-modules, $k$ itself is the unit object; so one can easily show that $(k,id_k)$ is the unit object in $\widetilde{\mathcal{H}}(\mathcal{YD})_H^H$ with the trivial right Hom-action $k\otimes H \to k,\: x\otimes h\mapsto \varepsilon(h)x$ and the right Hom-coaction $k\to k\otimes H,\: x\mapsto x\otimes 1_H$ for any $x$ in $k$ and $h$ in $H$.
It is obvious that $\tilde{l}_M$ is a $k$-isomorphism with the inverse $\tilde{l}_M^{-1}:M \to k\otimes M, \: m\mapsto 1\otimes \mu^{-1}(m)$. It can easily be shown that the relation $\mu\circ\tilde{l}_M=\tilde{l}_M\circ(id_k\otimes\mu)$ holds. Now we prove the right $(H,\alpha)$-linearity and right $(H,\alpha)$-colinearity of $\tilde{l}_M$: For all $x\in k$, $h\in H$ and $m\in M$,

\begin{eqnarray*}\tilde{l}_M((x\otimes m)\lhd h)&=&\tilde{l}_M(\varepsilon(h_1)x\otimes m\lhd h_2)=\varepsilon(h_1)x \mu(m\lhd h_2)\\
&=&x\mu(m)\lhd \alpha(\varepsilon(h_1)h_2)=x\mu(m)\lhd \alpha(\alpha^{-1}(h))\\
&=&\tilde{l}_M(x\otimes m)\lhd h,
\end{eqnarray*}

\begin{eqnarray*}((\tilde{l}_M\otimes id_H)\circ\rho^{k\otimes M})(x\otimes m)&=&(\tilde{l}_M\otimes id_H)((x\otimes m_{(0)})\otimes 1_Hm_{(1)})\\
&=&x\mu(m_{(0)})\otimes \alpha(m_{(1)})\\
&=&x((\mu\otimes\alpha)\circ \rho^{M})(m)\\
&=&\rho^{M}(x\mu(m))\\
&=&(\rho^{M}\circ\tilde{l}_M)(x\otimes m).
\end{eqnarray*}
The same argument holds for $\tilde{r}$.
\end{proof}

\begin{proposition}\label{braiding-Y-D-category}Let $(H,\alpha)$ be a monoidal Hom-bialgebra and  $(M,\mu)$, $(N,\nu)$ be $(H,\alpha)$-Hom-Yetter-Drinfeld modules. Then the $k$-linear map
\begin{equation}c_{M,N}:M\otimes N\to N\otimes M,\: m\otimes n\mapsto \nu(n_{(0)})\otimes \mu^{-1}(m)\lhd n_{(1)}\end{equation}
is a right $(H,\alpha)$-linear and right $(H,\alpha)$-colinear morphism. In case $(H,\alpha)$ is a monoidal Hom-Hopf-algebra with an invertible antipode it is also a bijection.
\end{proposition}

\begin{proof} We have the relation $(\nu\otimes\mu)\circ c_{M,N}=c_{M,N}\circ (\mu\otimes\nu)$ by the computation

\begin{eqnarray*}(\nu\otimes\mu)(c_{M,N}(m\otimes n))&=&(\nu\otimes\mu)(\nu(n_{(0)})\otimes \mu^{-1}(m)\lhd n_{(1)})\\
&=&\nu(\nu(n)_{(0)})\otimes m \lhd \alpha(n_{(1)})\\
&=&\nu(\nu(n)_{(0)})\otimes \mu^{-1}(\mu(m))\lhd \nu(n)_{(1)}\\
&=&c_{M,N}(\mu(m)\otimes \nu(n)).
\end{eqnarray*}
The $(H,\alpha)$-linearity holds as follows

\begin{eqnarray*}c_{M,N}((m\otimes n)\lhd h)&=&c_{M,N}(m\lhd h_1\otimes n\lhd h_2)\\
&=&\nu((n\lhd h_2)_{(0)})\otimes\mu^{-1}(m\lhd h_1)\lhd(n\lhd h_2)_{(1)}\\
&=&\nu((n\lhd h_2)_{(0)})\otimes(\mu^{-1}(m)\lhd\alpha^{-1}(h_1))\lhd(n\lhd h_2)_{(1)}\\
&=&\nu((n\lhd h_2)_{(0)})\otimes m\lhd(\alpha^{-1}(h_1)\alpha^{-1}((n\lhd h_2)_{(1)}))\\
&=&\nu(n_{(0)}\lhd \alpha^{-1}(h_1))\otimes m\lhd (n_{(1)}\alpha^{-1}(h_2))\\
&=&\nu(n_{(0)}\lhd h_1\otimes(\mu^{-1}(m)\lhd n_{(1)})\lhd h_2\\
&=&(\nu(n_{(0)})\otimes \mu^{-1}(m)\lhd n_{(1)})\lhd h\\
&=&c_{M,N}(m\otimes n)\lhd h,
\end{eqnarray*}
where in the fifth equality the twisted Yetter-Drinfeld condition has been used. We now show that $c_{M,N}$ is $(H,\alpha)$-colinear: In fact,

\begin{eqnarray*}(\rho^{N\otimes M}\circ c_{M,N})(m\otimes n)&=&\rho^{N\otimes M}(\nu(n_{(0)})\otimes \mu^{-1}(m)\lhd n_{(1)})\\
&=&(\nu(n_{(0)})_{(0)}\otimes(\mu^{-1}(m)\lhd n_{(1)})_{(0)})\otimes \nu(n_{(0)})_{(1)}(\mu^{-1}(m)\lhd n_{(1)})_{(1)}\\
&=&(\nu(n_{(0)(0)})\otimes\mu^{-1}((m\lhd\alpha( n_{(1)}))_{(0)}))\otimes \alpha(n_{(0)(1)})\alpha^{-1}((m\lhd \alpha(n_{(1)}))_{(1)})\\
&=&(n_{(0)}\otimes\mu^{-1}((m\lhd\alpha^{2}( n_{(1)2}))_{(0)}))\otimes \alpha(n_{(1)1})\alpha^{-1}((m\lhd \alpha^{2}(n_{(1)2}))_{(1)})\\
&=&(n_{(0)}\otimes\mu^{-1}(m_{(0)}\lhd \alpha^{-1}(\alpha^{2}(n_{(1)1}))))\otimes m_{(1)}\alpha^{-1}(\alpha^{2}(n_{(1)2}))\\
&=&(\nu(n_{(0)(0)})\otimes\mu^{-1}(m_{(0)}\lhd\alpha( n_{(0)(1)})))\otimes m_{(1)}n_{(1)}\\
&=&(c_{M,N}\otimes id_H)((m_{(0)}\otimes n_{(0)})\otimes m_{(1)}n_{(1)})\\
&=&(c_{M,N}\otimes id_H)(\rho^{M\otimes N}(m\otimes n)).
\end{eqnarray*}
Let us define $$c^{-1}_{M,N}:N\otimes M\to M\otimes N,\: n\otimes m\mapsto \mu^{-1}(m)\lhd S^{-1}(n_{(1)})\otimes \nu(n_{(0)}).$$
We verify that $c^{-1}_{M,N}$ is the inverse of $c_{M,N}$:
\begin{eqnarray*}c^{-1}_{M,N}(c_{M,N}(m\otimes n))&=&c^{-1}_{M,N}(\nu(n_{(0)})\otimes \mu^{-1}(m)\lhd n_{(1)})\\
&=&\mu^{-1}(\mu^{-1}(m)\lhd n_{(1)})\lhd S^{-1}(\nu(n_{(0)})_{(1)})\otimes \nu(\nu(n_{(0)})_{(0)})\\
&=&(\mu^{-2}(m)\lhd \alpha^{-1}(n_{(1)}))\lhd S^{-1}(\alpha(n_{(0)(1)}))\otimes \nu^{2}(n_{(0)(0)})\\
&=&\mu^{-1}(m)\lhd (\alpha^{-1}(n_{(1)})S^{-1}(n_{(0)(1)}))\otimes \nu^{2}(n_{(0)(0)})\\
&=&\mu^{-1}(m)\lhd (n_{(1)2}S^{-1}(n_{(1)1}))\otimes \nu(n_{(0)})\\
&=&\mu^{-1}(m)\lhd 1_H\otimes \nu(n_{(0)}\varepsilon(n_{(1)}))\\
&=&m\otimes n,
\end{eqnarray*}
and on the other hand we have
\begin{eqnarray*}c_{M,N}(c^{-1}_{M,N}(n\otimes m))&=&c_{M,N}(\mu^{-1}(m)\lhd S^{-1}(n_{(1)})\otimes \nu(n_{(0)}))\\
&=&\nu(\nu(n_{(0)})_{(0)})\otimes \mu^{-1}(\mu^{-1}(m)\lhd S^{-1}(n_{(1)}))\lhd \nu(n_{(0)})_{(1)}\\
&=&\nu^{2}(n_{(0)(0)})\otimes (\mu^{-2}(m)\lhd \alpha^{-1}(S^{-1}(n_{(1)})))\lhd \alpha(n_{(0)(1)})\\
&=&\nu(n_{(0)})\otimes \mu^{-1}(m)\lhd (S^{-1}(n_{(1)2})n_{(1)1})\\
&=&n\otimes m.
\end{eqnarray*}
\end{proof}

\begin{theorem}Let $(H,\alpha)$ be a monoidal Hom-bialgebra. Then $\widetilde{\mathcal{H}}(\mathcal{YD})^H_H$ is a prebraided monoidal category. It is a braided monoidal one under the requirement $(H,\alpha)$ be a monoidal Hom-Hopf algebra with a bijective antipode.
\end{theorem}

\begin{proof} The definition of tensor product is given in Proposition (\ref{tensor-prod-Y-D-category}),
the associativity constraint is described in Proposition (\ref{associator-Y-D-category}), left and right unitors are given in Proposition (\ref{unitors-Y-D-category}) and the (pre-)braiding is defined in Proposition (\ref{braiding-Y-D-category}).
The Hexagon Axiom for $c$ are left to be verified to finish the proof.

Let $(M,\mu)$,$(N,\nu)$, $(P,\pi)$ be in $\widetilde{\mathcal{H}}(\mathcal{YD})^H_H$; we show that the first hexagon axiom holds for $c$:

\begin{eqnarray*}\lefteqn{((id_N\otimes c_{M,P})\circ\tilde{a}_{N,M,P}\circ(c_{M,N}\otimes id_P))((m\otimes n)\otimes p)}\hspace{4em}\\
&=&((id_N\otimes c_{M,P})\circ\tilde{a}_{N,M,P})((\nu(n_{(0)})\otimes \mu^{-1}(m)\lhd n_{(1)})\otimes p)\\
&=&(id_N\otimes c_{M,P})(\nu^{2}(n_{(0)})\otimes (\mu^{-1}(m)\lhd n_{(1)}\otimes\pi^{-1}(p)))\\
&=&\nu^{2}(n_{(0)})\otimes(\pi(\pi^{-1}(p)_{(0)})\otimes \mu^{-1}(\mu^{-1}(m)\lhd n_{(1)})\lhd\pi^{-1}(p)_{(1)})\\
&=&\nu^{2}(n_{(0)})\otimes(p_{(0)}\otimes (\mu^{-2}(m)\lhd\alpha^{-1}(n_{(1)}))\lhd \alpha^{-1}(p_{(1)}))\\
&=&\nu^{2}(n_{(0)})\otimes(p_{(0)}\otimes \mu^{-1}(m)\lhd(\alpha^{-1}(n_{(1)})\alpha^{-2}(p_{(1)})))\\
&=&\nu^{2}(n_{(0)})\otimes (p_{(0)}\otimes \mu^{-1}(m \lhd (n_{(1)}\alpha^{-1}(p_{(1)}))))\\
&=&\tilde{a}_{N,P,M}((\nu(n_{(0)})\otimes p_{(0)})\otimes m \lhd (n_{(1)}\alpha^{-1}(p_{(1)})))\\
&=&\tilde{a}_{N,P,M}((\nu\otimes\pi)((n\otimes \pi^{-1}(p))_{(0)})\otimes \mu^{-1}(\mu(m))\lhd (n\otimes \pi^{-1}(p))_{(1)})\\
&=&(\tilde{a}_{N,P,M}\circ c_{M,N\otimes P})(\mu(m)\otimes(n\otimes \pi^{-1}(p)))\\
&=&(\tilde{a}_{N,P,M}\circ c_{M,N\otimes P}\circ\tilde{a}_{M,N,P})((m\otimes n)\otimes p).
\end{eqnarray*}

Lastly, we prove the second hexagon axiom:

\begin{eqnarray*}\lefteqn{\tilde{a}^{-1}_{P,M,N}\circ c_{M\otimes N,P}\circ\tilde{a}^{-1}_{M,N,P}(m\otimes(n\otimes p))}\hspace{4em}\\
&=&(\tilde{a}^{-1}_{P,M,N}\circ c_{M\otimes N,P})((\mu^{-1}(m)\otimes n)\otimes \pi(p))\\
&=&\tilde{a}^{-1}_{P,M,N}(\pi(\pi(p)_{(0)})\otimes ((\mu^{-1}\otimes\nu^{-1})(\mu^{-1}(m)\otimes n)\lhd\pi(p)_{(1)}))\\
&=&\tilde{a}^{-1}_{P,M,N}(\pi^{2}(p_{(0)})\otimes(\mu^{-2}(m)\otimes \nu^{-1}(n))\lhd \alpha(p_{(1)}))\\
&=&\tilde{a}^{-1}_{P,M,N}(\pi^{2}(p_{(0)})\otimes(\mu^{-2}(m)\lhd\alpha(p_{(1)})_1\otimes \nu^{-1}(n)\lhd\alpha(p_{(1)})_2))\\
&=&\tilde{a}^{-1}_{P,M,N}(\pi^{3}(p_{(0)(0)})\otimes (\mu^{-2}(m)\lhd\alpha(p_{(0)(1)})\otimes \nu^{-1}(n)\lhd p_{(1)}))\\
&=&(\pi^{2}(p_{(0)(0)})\otimes\mu^{-2}(m)\lhd\alpha(p_{(0)(1)}))\otimes n\lhd p_{(1)}\\
&=&(\pi(\pi(p_{(0)})_{(0)})\otimes \mu^{-1}(\mu^{-1}(m))\lhd\pi(p_{(0)})_{(1)})\otimes n\lhd \alpha(p_{(1)}))\\
&=&(c_{M,P}\otimes id_N)((\mu^{-1}(m)\otimes \pi(p_{(0)}))\otimes n\lhd \alpha(p_{(1)}))\\
&=&(c_{M,P}\otimes id_N)((\mu^{-1}(m)\otimes \pi(p_{(0)}))\otimes \nu(\nu^{-1}(n)\lhd p_{(1)}))\\
&=&((c_{M,P}\otimes id_N)\circ\tilde{a}^{-1}_{M,P,N})(m\otimes(\pi(p_{(0)})\otimes \nu^{-1}(n)\lhd p_{(1)}))\\
&=&((c_{M,P}\otimes id_N)\circ\tilde{a}^{-1}_{M,P,N}\circ (id_M\otimes c_{N,P}))(m\otimes (n\otimes p)).
\end{eqnarray*}

\end{proof}

Together with Theorem \ref{one-to-one-LeftCov-RightMod} and Theorem \ref{one-to-one-Bicov-YD}, Theorem (\ref{inverse-functors}) provides:

\begin{theorem}Let $(H,\alpha)$ be a monoidal Hom-Hopf algebra. Then the equivalences in (\ref{inverse-functors}):
$$F=(H\otimes -,\alpha\otimes -):\widetilde{\mathcal{H}}(\mathcal{M}_k) \to\: _{H}^{H}\widetilde{\mathcal{H}}(\mathcal{M}_k),$$
$$G=\: ^{coH}(-):\: _{H}^{H}\widetilde{\mathcal{H}}(\mathcal{M}_k) \to \widetilde{\mathcal{H}}(\mathcal{M}_k),$$
 induce tensor equivalences between
\begin{enumerate}
   \item \label{tens-equivalence1}the category of right $(H,\alpha)$-Hom-modules and the category of left-covariant $(H,\alpha)$-Hom-bimodules,
   \item\label{tens-equivalence2} the category of right-right $(H,\alpha)$-Hom-Yetter-Drinfeld modules and the category of bicovariant $(H,\alpha)$-Hom-bimodules.
 \end{enumerate}
\end{theorem}

\begin{proof} The right $(H,\alpha)$-Hom-module structure on $(H\otimes V,\alpha\otimes\mu)$ for a right $(H,\alpha)$-Hom-module $(V,\mu)$ is given in Proposition (\ref{left-cov-structure}) and the right $(H,\alpha)$-Hom-comodule structure on $(H\otimes W,\alpha\otimes\nu)$ for a right $(H,\alpha)$-Hom-comodule $(W,\nu)$ is given in Proposition (\ref{bicovariant-bimod-structure}). It remains only to prove that one of the inverse equivalences, say $F$ together with $\varphi_2(V,W):(H\otimes V)\otimes_H(H\otimes W)\to H\otimes(V\otimes W)$ given by
\begin{equation*}\varphi_2(V,W)((g\otimes v)\otimes_H(h\otimes w))=g\alpha(h_1)\otimes(\mu^{-1}(v)\lhd h_2\otimes w) \end{equation*}
for all $g,h \in H$, $v\in V$, $w\in W$, is a tensor functor in each case. Define
$$\varphi_2(V,W)^{-1}:H\otimes(V\otimes W)\to (H\otimes V)\otimes_H(H\otimes W),\:h\otimes(v\otimes w)\mapsto (\alpha^{-1}(h)\otimes v)\otimes_H(1_H\otimes w),$$ which is an inverse of $\varphi_2(V,W)$: For $h,g\in H$, $v\in V$ and $w\in W$,

\begin{eqnarray*}\varphi_2(V,W)(\varphi_2(V,W)^{-1}(h\otimes(v\otimes w)))&=&\varphi_2(V,W)((\alpha^{-1}(h)\otimes v)\otimes_H(1_H\otimes w))\\
&=&\alpha^{-1}(h)1_H\otimes(\mu^{-1}(m)\lhd 1_H\otimes w)\\
&=&h\otimes(v\otimes w),
\end{eqnarray*}

\begin{eqnarray*}\varphi_2(V,W)^{-1}(\varphi_2(V,W)((g\otimes v)\otimes_H(h\otimes w)))&=&\varphi_2(V,W)^{-1}(g\alpha(h_1)\otimes(\mu^{-1}(v)\lhd h_2\otimes w))\\
&=&(\alpha^{-1}(g\alpha(h_1))\otimes \mu^{-1}(v)\lhd h_2)\otimes_H (1_H\otimes w)\\
&=&(\alpha^{-1}(g)\otimes \mu^{-1}(v))h\otimes_H (1_H\otimes w)\\
&=&(g\otimes v)\otimes_H h(1_H\otimes \nu^{-1}(w))\\
&=&(g\otimes v)\otimes_H (h\otimes w),
\end{eqnarray*}
and one can also show that the relation $(\alpha\otimes(\mu\otimes\nu))\circ\varphi_2(V,W)=\varphi_2(V,W)\circ((\alpha\otimes\mu)\otimes_H(\alpha\otimes\nu))$ holds.
We now verify that the coherence condition on $F$ is fulfilled:

\begin{eqnarray*}\lefteqn{(\varphi_2(U,V\otimes W)\circ(id\otimes\varphi_2(V,W))\circ\tilde{a}_Q)(((g\otimes u)\otimes_H (h\otimes v))\otimes_H(k\otimes w))}\hspace{1em}\\
&=&(\varphi_2(U,V\otimes W)\circ(id\otimes\varphi_2(V,W)))((\alpha(g)\otimes\mu(u))\otimes_H((h\otimes v)\otimes_H(\alpha^{-1}(k)\otimes \pi^{-1}(w))))\\
&=&\varphi_2(U,V\otimes W)((\alpha(g)\otimes\mu(u))\otimes_H(hk_1\otimes(\nu^{-1}(v)\lhd\alpha^{-1}(k_2)\otimes\pi^{-1}(w) )))\\
&=&\alpha(g)\alpha(h_1k_{11})\otimes(u\lhd h_2k_{12}\otimes(\nu^{-1}(v)\lhd\alpha^{-1}(k_2)\otimes\pi^{-1}(w) ))\\
&=&\alpha(g)(\alpha(h_1)k_1)\otimes(u\lhd h_2k_{21}\otimes(\nu^{-1}(v)\lhd k_{22}\otimes\pi^{-1}(w) ))\\
&=&(id\otimes \tilde{a}_{Q'})(\alpha(g)(\alpha(h_1)k_1)\otimes((\mu^{-1}(u)\lhd \alpha^{-1}(h_2k_{21})\otimes\nu^{-1}(v)\lhd k_{22})\otimes w))\\
&=&(id\otimes\tilde{a}_{Q'})(\alpha(g)(\alpha(h_1)k_1)\otimes(((\mu^{-2}(u)\lhd \alpha^{-1}(h_2))\lhd k_{21}\otimes\nu^{-1}(v)\lhd k_{22})\otimes w))\\
&=&(id\otimes\tilde{a}_{Q'})((g\alpha(h_1))\alpha(k_1)\otimes((\mu^{-2}(u)\lhd \alpha^{-1}(h_2)\otimes\nu^{-1}(v))\lhd k_2\otimes w))\\
&=&((id\otimes\tilde{a}_{Q'})\circ\varphi_2(U\otimes V,W))((g\alpha(h_1)\otimes(\mu^{-1}\lhd h_2\otimes v))\otimes_H(k\otimes w))\\
&=&((id\otimes\tilde{a}_{Q'})\circ\varphi_2(U\otimes V,W)\circ(\varphi_2(U,V)\otimes id))(((g\otimes u)\otimes_H (h\otimes v))\otimes_H(k\otimes w)).
\end{eqnarray*}

For (\ref{tens-equivalence1}) we verify that the $k$-isomorphism $\varphi_2(V,W)$ is a morphism of left-covariant $(H,\alpha)$-Hom-bimodules, that is, we prove its left $(H,\alpha)$-linearity, $(H,\alpha)$-colinearity, and right $(H,\alpha)$-linearity, respectively:

\begin{eqnarray*}\varphi_2(V,W)(k((g\otimes v)\otimes_H(h\otimes w)))&=&\varphi_2(V,W)(\alpha^{-1}(k)(g\otimes v)\otimes_H(\alpha(h)\otimes \nu(w))))\\
&=&\varphi_2(V,W)((\alpha^{-2}(k)g\otimes \mu(v))\otimes_H(\alpha(h)\otimes \nu(w)))\\
&=&(\alpha^{-2}(k)g)\alpha^{2}(h_1)\otimes (v\lhd \alpha(h_2)\otimes \nu(w))\\
&=&\alpha^{-1}(k)(g\alpha(h_1))\otimes ((\mu\otimes\nu)(\mu^{-1}(v)\lhd h_2\otimes w))\\
&=&k(g\alpha(h_1)\otimes(\mu^{-1}(v)\lhd h_2\otimes w))\\
&=&k\varphi_2(V,W)((g\otimes v)\otimes_H(h\otimes w)),
\end{eqnarray*}

\begin{eqnarray*}\lefteqn{(id \otimes\varphi_2(V,W))(^{Q}\rho((g\otimes v)\otimes_H(h\otimes w)))}\hspace{6em}\\
&=&(id \otimes\varphi_2(V,W))(\alpha(g_1)\alpha(h_1)\otimes((g_2\otimes\mu^{-1}(v))\otimes_H(h_2\otimes \nu^{-1}(w))))\\
&=&\alpha(g_1)\alpha(h_1)\otimes(g_2\alpha(h_{21})\otimes (\mu^{-2}(v)\lhd h_{22}\otimes \nu^{-1}(w)))\\
&=&\alpha(g_1)\alpha^{2}(h_{11})\otimes(g_2\alpha(h_{12})\otimes (\mu^{-2}(v)\lhd \alpha^{-1}(h_2)\otimes \nu^{-1}(w)))\\
&=&\alpha((g\alpha(h_1))_1)\otimes ((g\alpha(h_1))_2\otimes (\mu^{-1}(\mu^{-1}(v)\lhd h_2)\nu^{-1}(w)))\\
&=&^{Q'}\rho(g\alpha(h_1)\otimes (\mu^{-1}(v)\lhd h_2\otimes w))\\
&=&^{Q'}\rho(\varphi_2(V,W)((g\otimes v)\otimes_H(h\otimes w))),
\end{eqnarray*}

\begin{eqnarray*}\lefteqn{\varphi_2(V,W)(((g\otimes v)\otimes_H(h\otimes w))k)}\hspace{3em}\\
&=&\varphi_2(V,W)((\alpha(g)\otimes\mu(v))\otimes_H(h\otimes w)\alpha^{-1}(k))\\
&=&\varphi_2(V,W)((\alpha(g)\otimes\mu(v))\otimes_H(h\alpha^{-1}(k_1)\otimes w\lhd\alpha^{-1}(k_2)))\\
&=&\alpha(g)(\alpha(h_1)k_{11})\otimes(v\lhd(h_2\alpha^{-1}(k_{12}))\otimes w\lhd\alpha^{-1}(k_2))\\
&=&\alpha(g)(\alpha(h_1)\alpha^{-1}(k_1))\otimes(v\lhd(h_2\alpha^{-1}(k_{21}))\otimes w\lhd k_{22})\\
&=&\alpha(g)(\alpha(h_1)\alpha^{-1}(k_1))\otimes(\mu(\mu^{-1}(v))\lhd(h_2\alpha^{-1}(k_{21}))\otimes w\lhd k_{22})\\
&=&(g\alpha(h_1))k_1\otimes ((\mu^{-1}(v)\lhd h_2)\lhd k_{21}\otimes w\lhd k_{22})\\
&=&(g\alpha(h_1))k_1\otimes (\mu^{-1}(v)\lhd h_2\otimes w)k_2\\
&=&\varphi_2(V,W)((g\otimes v)\otimes_H(h\otimes w))k.
\end{eqnarray*}

For (\ref{tens-equivalence2}) we need only to check that $\varphi_2(V,W)$ is right $(H,\alpha)$-colinear. Let us denote by $\sigma^{Q'}$ and $\sigma^{Q}$ the right $(H,\alpha)$-Hom-comodule structures on $Q'=H\otimes(V\otimes W)$ and $Q=(H\otimes V)\otimes_H(H\otimes W)$. Then

\begin{eqnarray*}\lefteqn{\sigma^{Q'}(\varphi_2(V,W)((g\otimes v)\otimes_H(h\otimes w)))}\hspace{2em}\\
&=&\sigma^{Q'}(g\alpha(h_1)\otimes(\mu^{-1}(v)\lhd h_2\otimes w))\\
&=&((g\alpha(h_1))_1\otimes ((\mu^{-1}(v)\lhd h_2)_{[0]}\otimes w_{[0]}))\otimes (g\alpha(h_1))_2((\mu^{-1}(v)\lhd h_2)_{[1]}w_{[1]})\\
&=&(g_1\alpha(h_{11})\otimes ((\mu^{-1}(v)\lhd h_2)_{[0]}\otimes w_{[0]}))\otimes \alpha(g_2)(\alpha(h_{12})(\alpha^{-1}((\mu^{-1}(v)\lhd h_2)_{[1]})\alpha^{-1}(w_{[1]})))\\
&=&(g_1\alpha(h_{11})\otimes ((\mu^{-1}(v)\lhd h_2)_{[0]}\otimes w_{[0]}))\otimes \alpha(g_2)((h_{12}\alpha^{-1}((\mu^{-1}(v)\lhd h_2)_{[1]}))w_{[1]})\\
&=&(g_1h_1\otimes ((\mu^{-1}(v)\lhd \alpha(h_{22}))_{[0]}\otimes w_{[0]}))\otimes \alpha(g_2)((h_{21}\alpha^{-1}((\mu^{-1}(v)\lhd \alpha(h_{22}))_{[1]}))w_{[1]})\\
&=&(g_1h_1\otimes ((\mu^{-1}(v)\lhd \alpha(h_2)_2)_{[0]}\otimes w_{[0]}))\otimes \alpha(g_2)(\alpha^{-1}(\alpha(h_2)_1(\mu^{-1}(v)\lhd \alpha(h_2)_2)_{[1]})w_{[1]})\\
&=&(g_1h_1\otimes(\mu^{-1}(v)_{[0]}\lhd \alpha^{-1}(\alpha(h_2)_1)\otimes w_{[0]}))\otimes \alpha(g_2)((\mu^{-1}(v)_{[1]}\alpha^{-1}(\alpha(h_2)_2))w_{[1]})\\
&=&(g_1h_1\otimes(\mu^{-1}(v_{[0]})\lhd h_{21}\otimes w_{[0]}))\otimes \alpha(g_2)((\alpha^{-1}(v_{[1]})h_{22})w_{[1]})\\
&=&(g_1\alpha(h_{11})\otimes(\mu^{-1}(v_{[0]})\lhd h_{12}\otimes w_{[0]}))\otimes \alpha(g_2)((\alpha^{-1}(v_{[1]})\alpha^{-1}(h_2))w_{[1]})\\
&=&(g_1\alpha(h_{11})\otimes(\mu^{-1}(v_{[0]})\lhd h_{12}\otimes w_{[0]}))\otimes (g_2v_{[1]})(h_2w_{[1]})\\
&=&(\varphi_2(V,W)\otimes id_H)(((g_1\otimes v_{[0]})\otimes_H(h_1\otimes w_{[0]}))\otimes(g_2v_{[1]})(h_2w_{[1]}))\\
&=&(\varphi_2(V,W)\otimes id_H)(\sigma^{Q}((g\otimes v)\otimes_H(h\otimes w))),
\end{eqnarray*}
where we have used the twisted Yetter-Drinfeld condition in the seventh equality.
\end{proof}

\begin{corollary}Let $(H,\alpha)$ be a monoidal Hom-Hopf algebra. The categories $^{H}_{H}\widetilde{\mathcal{H}}(\mathcal{M}_k)_H^H$ and $\widetilde{\mathcal{H}}(\mathcal{YD})_H^H$ are equivalent as prebraided monoidal categories. The tensor equivalence between them is braided whenever $(H,\alpha)$ has a bijective antipode.
\end{corollary}

\begin{proof} It suffices to regard the case of bicovariant $(H,\alpha)$-Hom-bimodules $(M,\mu')=(H\otimes V,\alpha\otimes\mu)$ and $(N,\nu')=(H\otimes W,\alpha\otimes\nu)$ with $(V,\mu)$ and $(W,\nu)$ $(H,\alpha)$-Hom-Yetter-Drinfeld modules. Thus for $h\in H$, $v\in V$ and $w \in W$ we have

\begin{eqnarray*}&&(\varphi_2(W,V)\circ c_{M,N}\circ\varphi_2(V,W)^{-1})(h\otimes(v\otimes w))\\
&=&\varphi_2(W,V)(c_{M,N}( (\alpha^{-1}(h)\otimes v)\otimes_H(1_H\otimes w)))\\
&=&\varphi_2(W,V)(h_1((1_H\otimes w_{[0][0]})S(\alpha(w_{[0][1]})))\otimes_H (S(h_{21})(\alpha^{-1}(h_{22})\otimes\mu^{-2}(v)))\alpha(w_{[1]}))\\
&=&\varphi_2(W,V)(h_1((1_H\otimes \nu^{-1}(w_{[0]}))S(\alpha(w_{[1]1})))\otimes_H (\alpha^{-1}(S(h_{21})h_{22})\otimes\mu^{-1}(v))\alpha^{2}(w_{[1]2}))\\
&=&\varphi_2(W,V)((h_1\varepsilon(h_2))((1_H\otimes \nu^{-1}(w_{[0]}))S(\alpha(w_{[1]1})))\otimes_H (1_H\otimes\mu^{-1}(v))\alpha^{2}(w_{[1]2}))\\
&=&\varphi_2(W,V)((\alpha^{-2}(h)(1_H\otimes \nu^{-1}(w_{[0]})))S(\alpha^{2}(w_{[1]1}))\otimes_H (1_H\otimes\mu^{-1}(v))\alpha^{2}(w_{[1]2}))\\
&=&\varphi_2(W,V)((\alpha^{-3}(h)1_H\otimes\nu(\nu^{-1}(w_{[0]})))S(\alpha^{2}(w_{[1]1}))\otimes_H (1_H\alpha^{2}(w_{[1]2})_1\otimes\mu^{-1}(v)\lhd\alpha^{2}(w_{[1]2})_2))\\
&=&\varphi_2(W,V)((\alpha^{-2}(h)\otimes w_{[0]})S(\alpha^{2}(w_{[1]1}))\otimes_H (\alpha^{3}(w_{[1]21})\otimes\mu^{-1}(v)\lhd\alpha^{2}(w_{[1]22})))\\
&=&\varphi_2(W,V)((\alpha^{-1}(h)\otimes \nu(w_{[0]}))\otimes_H S(\alpha^{2}(w_{[1]1}))(\alpha^{2}(w_{[1]21})\otimes\mu^{-2}(v)\lhd\alpha(w_{[1]22})))\\
&=&\varphi_2(W,V)((\alpha^{-1}(h)\otimes \nu(w_{[0]}))\otimes_H (S(\alpha(w_{[1]1}))\alpha^{2}(w_{[1]21})\otimes\mu^{-1}(v)\lhd\alpha^{2}(w_{[1]22})))\\
&=&\varphi_2(W,V)((\alpha^{-1}(h)\otimes \nu(w_{[0]}))\otimes_H (S(\alpha^{2}(w_{[1]11}))\alpha^{2}(w_{[1]12})\otimes\mu^{-1}(v)\lhd\alpha(w_{[1]2})))\\
&=&\varphi_2(W,V)((\alpha^{-1}(h)\otimes \nu(w_{[0]}))\otimes_H (\alpha^{2}(S(w_{[1]11})w_{[1]12})\otimes\mu^{-1}(v)\lhd\alpha(w_{[1]2})))\\
&=&\varphi_2(W,V)((\alpha^{-1}(h)\otimes \nu(w_{[0]}))\otimes_H (1_H\otimes\mu^{-1}(v)\lhd\alpha(\varepsilon(w_{[1]1})w_{[1]2})))\\
&=&\varphi_2(W,V)((\alpha^{-1}(h)\otimes \nu(w_{[0]}))\otimes_H (1_H\otimes\mu^{-1}(v)\lhd w_{[1]}))\\
&=&\alpha^{-1}(h)1_H\otimes(\nu^{-1}(\nu(w_{[0]}))\lhd 1_H\otimes \mu^{-1}(v)\lhd w_{[1]})\\
&=&h\otimes(\nu(w_{[0]})\otimes\mu^{-1}(v)\lhd w_{[1]})\\
&=&(id_H\otimes c_{V,W})(h\otimes(v\otimes w)),
\end{eqnarray*}
which demonstrates that $F$ is a (pre-)braided tensor equivalence.
\end{proof}

\section{Acknowledgments}
The author would like to thank Professor Christian Lomp for his helpful suggestions and valuable comments. This research was funded by the European Regional Development Fund through the programme COMPETE and by the Portuguese Government through the FCT- Fundação para a Ciência e a Tecnologia under the project PEst-C/MAT/UI0144/2013. The author was supported by the grant SFRH/BD/51171/2010.

\end{document}